\documentclass[11pt, a4paper,leqno]{amsart}

\usepackage{amsmath,amsthm,amscd,amssymb,amsfonts, amsbsy}
\usepackage{latexsym}
\usepackage{txfonts}
\usepackage{exscale}
\usepackage{mathrsfs}

\usepackage{esint} % For \fint symbol
\usepackage{enumerate}

\usepackage[colorlinks=true, linkcolor=blue, citecolor=red, urlcolor=red, backref=page]{hyperref} % Add coloured links

\usepackage[utf8]{inputenc}

\calclayout
\allowdisplaybreaks

\numberwithin{equation}{section}

\theoremstyle{plain}
\newtheorem{theorem}[equation]{Theorem}

\newtheorem{lemma}[equation]{Lemma}

\newtheorem{corollary}[equation]{Corollary}

\theoremstyle{definition}
\newtheorem{definition}[equation]{Definition}

\theoremstyle{remark}
\newtheorem{remark}[equation]{Remark}

\makeatletter
\@namedef{subjclassname@2020}{%
  \textup{2020} Mathematics Subject Classification}
\makeatother

\def\R{\mathbb{R}}

\newcommand{\oLo}{\ensuremath{\omega_{L_0}}}

\newcommand{\ZZ}{{\mathbb{Z}}}

\newcommand{\dist}{\operatorname{dist}}

\newcommand{\re}{\mathbb{R}}

\newcommand{\ree}{\mathbb{R}^{n+1}}
\newcommand{\dd}{\mathbb{D}}

\newcommand{\F}{\mathcal{F}}

\newcommand{\M}{\mathcal{M}}
\newcommand{\W}{\mathcal{W}}

\newcommand{\G}{\mathcal{G}}

\newcommand{\pom}{\partial\Omega}

\newcommand{\loc}{\mathrm{loc}}

\renewcommand{\emptyset}{\mbox{\textup{\O}}}
\newcommand{\tinyemptyset}{\mbox{\tiny \textup{\O}}}

\DeclareMathOperator{\supp}{supp}

\DeclareMathOperator{\diam}{diam}
\DeclareMathOperator{\interiorior}{int}

\DeclareMathOperator*{\Lip}{Lip}

\def\div{\mathop{\operatorname{div}}\nolimits}
\def\Cap{\mathop{\operatorname{Cap}_2}\nolimits}

\newcommand{\normmm}[1]{{\left\vert\kern-0.25ex\left\vert\kern-0.25ex\left\vert #1
		\right\vert\kern-0.25ex\right\vert\kern-0.25ex\right\vert}}

\def\w{\omega}
\def\D{\mathbb{D}}
\def\Z{\mathbb{Z}}
\def\interior{\operatorname{int}}

\begin{document}
	\allowdisplaybreaks

\title[Elliptic operators in 1-sided NTA domains satisfying the capacity density condition]{On the $A_\infty$ condition for elliptic operators in 1-sided NTA domains satisfying the capacity density condition}
	
\author{Mingming Cao}

\address{Mingming Cao
	\\
	Instituto de Ciencias Matemáticas CSIC-UAM-UC3M-UCM
	\\
	Consejo Superior de Investigaciones Cien\-tí\-ficas
	\\
	C/ Nicolás Cabrera, 13-15
	\\
	E-28049 Madrid, Spain}
\email{mingming.cao@icmat.es}

\author{Óscar Domínguez}

\address{Óscar Domínguez\\
	Departamento de Análisis Matemático y Matemática Aplicada
	\\
	Facultad de Matemáticas, Universidad Complutense de Madrid
	\\
	Plaza de Ciencias 3, E-28040 Madrid, Spain.}
\email{oscar.dominguez@ucm.es}
	
\author{José María Martell}
	
\address{José María Martell
\\
Instituto de Ciencias Matemáticas CSIC-UAM-UC3M-UCM
\\
Consejo Superior de Investigaciones Científicas
\\
C/ Nicolás Cabrera, 13-15
\\
E-28049 Madrid, Spain}
\email{chema.martell@icmat.es}
	
\author{Pedro Tradacete}

\address{Pedro Tradacete
	\\
	Instituto de Ciencias Matemáticas CSIC-UAM-UC3M-UCM
	\\
	Consejo Superior de Investigaciones Científicas
	\\
	C/ Nicolás Cabrera, 13-15
	\\
	E-28049 Madrid, Spain}
\email{pedro.tradacete@icmat.es}

\thanks{The first author is supported by Spanish Ministry of Science and Innovation, through the Juan de la Cierva-Formación 2018, FJC2018-038526-I. The first, third, and last authors acknowledge financial support from the Spanish Ministry of Science and Innovation, through the ``Severo Ochoa Programme for Centres of Excellence in R\&D'' (CEX2019-000904-S) and  from the Spanish National Research Council, through the ``Ayuda extraordinaria a Centros de Excelencia Severo Ochoa'' (20205CEX001). The second  author was partially supported by Spanish Ministry of Science and Innovation through grant MTM2017-84058-P (AEI/FEDER,UE). The third author also acknowledges that the research leading to these results has received funding from the European Research Council under the European Union's Seventh Framework Programme (FP7/2007-2013)/ ERC agreement no. 615112 HAPDEGMT. The third author was partially supported by the Spanish Ministry of Science and Innovation,  MTM PID2019-107914GB-I00. The last author was partially supported by Spanish Ministry of Science and Innovation through grant MTM2016-75196-P (AEI/FEDER,UE)}

\date{\today}



\begin{abstract}
Let $\Omega \subset \mathbb{R}^{n+1}$, $n\ge 2$, be a 1-sided non-tangentially accessible domain, that is, a set which is quantitatively open and path-connected. Assume also that $\Omega$ satisfies the capacity density condition. Let $L_0 u=-\div(A_0 \nabla u)$,  $Lu=-\div(A\nabla u)$ be two real (not necessarily symmetric) uniformly elliptic operators in $\Omega$, and write $\omega_{L_0}, \omega_L$ for the respective associated elliptic measures. We establish the equivalence between  the following properties: (i) $\omega_L \in A_{\infty}(\omega_{L_0})$, (ii) $L$ is $L^p(\omega_{L_0})$-solvable for some $p\in (1,\infty)$, (iii) bounded null solutions of $L$ satisfy Carleson measure estimates with respect to $\omega_{L_0}$, (iv) $\mathcal{S}<\mathcal{N}$ (i.e., the conical square function is controlled by the non-tangential maximal function) in $L^q(\omega_{L_0})$ for some (or for all) $q\in (0,\infty)$ for any null solution of $L$, and (v) $L$ is $\mathrm{BMO}(\omega_{L_0})$-solvable. Moreover, in each of the properties (ii)-(v) it is enough to consider the class of solutions given by characteristic functions of Borel sets (i.e, $u(X)=\omega_L^X(S)$ for an arbitrary Borel set $S\subset\partial\Omega$).

Also, we obtain a qualitative analog of the previous equivalences. Namely, we characterize the absolute continuity of $\omega_{L_0}$ with respect to $\omega_L$ in terms of some qualitative local $L^2(\omega_{L_0})$ estimates for the truncated conical square function for any bounded null solution of $L$. This is also equivalent to the finiteness $\omega_{L_0}$-almost everywhere of the truncated conical square function for any bounded null solution of $L$.  As applications, we show that $\omega_{L_0}$ is absolutely continuous with respect to $\omega_L$ if the disagreement of the coefficients satisfies some qualitative quadratic estimate in truncated cones for $\omega_{L_0}$-almost everywhere vertex. Finally, when $L_0$ is either the transpose of $L$ or its symmetric part, we obtain the corresponding absolute continuity upon assuming that the antisymmetric part of the coefficients has some controlled oscillation in truncated cones for $\omega_{L_0}$-almost every vertex.
\end{abstract}

\maketitle

\setcounter{tocdepth}{1}
\tableofcontents

%\bigskip

\section{Introduction}

The solvability of the Dirichlet problem \eqref{D} on rough domains has been of great interest in the last fifty years. Given a domain $\Omega \subset \R^{n+1}$ and a uniformly elliptic operator $L$ on $\Omega$, it consists on finding a solution $u$ (satisfying natural conditions in accordance to what is known for the boundary data $f$) to the boundary value problem
\begin{equation}\label{D}
	\begin{cases}
		L u =0 & \quad\text{in}\quad\Omega,\\
		u =f & \quad\text{on}\quad\partial \Omega.
	\end{cases}
\end{equation}
To address this question, one typically investigates the properties of the corresponding elliptic measure, since it is the fundamental tool that enables us to construct solutions of \eqref{D}. The techniques from harmonic analysis and geometric measure theory have allowed us to study the regularity of elliptic measures and hence understand this subject well. Conversely, the good properties of elliptic measures allow us to effectively use the machinery from these fields to obtain information about the topology and the regularity of the domains. These ideas have led to a quite active research at the intersection of harmonic analysis, partial differential equations, and geometric measure theory.

The connection between the geometry of a domain and the absolute continuity properties of its harmonic measure goes back to the classical result of F. and M. Riesz \cite{RR}, which showed that for a simply connected domain in the plane, the rectifiability of its boundary implies that harmonic measure is mutually absolutely continuous with respect to the surface measure. After that, considerable attention has focused on establishing higher dimensional analogues and the converse of the F. and M. Riesz theorem. For a planar domain, Bishop and Jones \cite{BJ} proved that if only a portion of the boundary is rectifiable, harmonic measure is absolutely continuous with respect to arclength on that portion. A counter-example was also constructed to show that the result of \cite{RR} may fail in the absence of some strong connectivity property (like simple connectivity). In dimensions greater than $2$, Dahlberg \cite{D77} established a  quantitative version of the absolute continuity of harmonic measures with respect to surface measure on the boundary of a Lipschitz domain. This result was extended to $\mathrm{BMO}_1$ domains by Jerison and Kenig \cite{JK82}, and to chord-arc domains by David and Jerison \cite{DJ} (see also \cite{AHMNT, HM1, HMU} for the case of  1-sided chord-arc domains). In this direction, this was culminated in the recent results of \cite{AHMMT} under some optimal background hypothesis (an open set in $\re^{n+1}$ satisfying an interior corkscrew condition with an $n$-dimensional Ahlfors-David regular boundary). Indeed, \cite{AHMMT} gives a complete picture of the relationship between the quantitative absolute continuity of harmonic measure
with respect to surface measure (or, equivalently, the solvability of \eqref{D} for singular data, see \cite{HLe}) and the rectifiability of the boundary plus some weak local John condition (that is, local accessibility by non-tangential paths to some pieces of the boundary). Another significant extension of the F. and M. Riesz theorem was obtained in \cite{AHMMMTV}, where it was proved that, in any dimension and in the absence of any connectivity condition, every piece of the boundary with finite surface measure is rectifiable, provided surface measure is absolutely continuous with respect to harmonic measure on that piece. It is worth pointing out that all the aforementioned results are restricted to the $n$-dimensional boundaries of domains in $\R^{n+1}$. Some analogues have been obtained in \cite{DEM, DFM, DM, MZ} on lower-dimensional sets.

On the other hand, the solvability of the Dirichlet problem \eqref{D} is closely linked with the absolute continuity properties of elliptic measures. The importance of the quantitative absolute continuity of the elliptic measure with respect to the surface measure comes from the fact that $\omega_L \in RH_q(\sigma)$ (short for the Reverse H\"older class with respect to $\sigma$, being $\sigma$ the surface measure) is equivalent to the $L^{q'}(\sigma)$-solvability of the Dirichlet problem (see, e.g. \cite{HLe}). In 1984, Dahlberg formulated a conjecture concerning the optimal conditions on a matrix of coefficients guaranteeing that the Dirichlet problem \eqref{D} with $L^p$ data for some $p \in (1, \infty)$ is solvable. Kenig and Pipher \cite{KP} made the first attempt on bounded Lipschitz domains and gave an affirmative answer to Dahlberg's conjecture. More precisely, they showed that elliptic measure is quantitatively absolutely continuous with respect to surface measure whenever the gradient of the coefficients satisfies a Carleson measure condition. This was done in Lipschitz domains but can be naturally extended to chord-arc domains. In some sense, some recent results have shown that this class of domains is optimal. First, \cite{HM1, HMU, AHMNT} show that in the case of the Laplacian and for 1-sided chord-arc domains,  the fact that the harmonic measure is quantitatively absolutely continuous with respect to surface measure (equivalently, the $L^{p}(\sigma)$-Dirichlet problem is solvable for some finite $p$) implies that the domains must have exterior corkscrews, hence they are chord-arc domains. Indeed, in a first attempt to generalize this to the class of Kenig-Pipher operators, Hofmann, the third author of the present paper, and Toro \cite{HMT1} were able to consider variable coefficients whose gradient satisfies some $L^1$-Carleson condition (in turn, stronger than the one in \cite{KP}). The general case, on which the operators are in the optimal Kenig-Pipher--class (that is, the gradient of the coefficients satisfies an $L^2$-Carleson condition) has been recently solved by Hofmann et al.~\cite{HMMTZ}.

In another direction, one can consider perturbations of elliptic operators in rough domains. That is, one seeks for conditions on the disagreement of two coefficient matrices so that the solvability of the Dirichlet problem or the quantitative absolute continuity with respect to the surface measure of the elliptic measure for one elliptic operator could be transferred to the other operator. This problem was initiated by Fabes, Jerison and Kenig \cite{FJK} in the case of continuous and symmetric coefficients, and extended by Dahlberg \cite{D86} to a more general setting under a vanishing Carleson measure condition. Soon after, working again in the domain $\Omega=B(0, 1)$ and with symmetric operators, Fefferman \cite{F} improved Dahlberg's result by formulating the boundedness of a conical square function, which allows one to preserve the $A_{\infty}$ property of elliptic measures, but without preserving the reverse H\"{o}lder exponent (see \cite[Theorem~2.24]{FKP}). A major step forward was made by Fefferman, Kenig and Pipher \cite{FKP} by giving an optimal Carleson measure perturbation on Lipschitz domains. Additionally, they established another kind of perturbation to study the quantitative absolute continuity between two elliptic measures. Beyond the Lipschitz setting, these results were extended to chord-arc domains \cite{MPT1, MPT2}, 1-sided chord-arc domains \cite{CHM, CHMT}, and 1-sided non-tangentially accessible domains satisfying the capacity density condition \cite{AHMT}. It is worth mentioning that the so-called extrapolation of Carleson measure was utilized in \cite{AHMT, CHM}. Nevertheless, a simpler and novel argument was presented in \cite{CHMT} to get the large constant perturbation. More specifically, $A_{\infty}$ property of elliptic measures can be characterized by the fact that every bounded weak solution of $L$ satisfies Carleson measure estimates. Also, it is worth mentioning that \cite{AHMT} considers for the first time perturbation results on sets with bad surface measures.

The goal of this paper is to continue with the line of research initiated in \cite{AHMT}. We work with $\Omega \subset \R^{n+1}$, $n\ge 2$, a 1-sided non-tangentially accessible domain  satisfying the capacity density condition. We consider two real (not necessarily symmetric) uniformly elliptic operators $L_0 u=-\div(A_0 \nabla u)$ and $Lu=-\div(A\nabla u)$ in $\Omega$, and denote by $\omega_{L_0}, \omega_L$ the respective associated elliptic measures. The paper \cite{AHMT} considered the perturbation theory in this context providing natural conditions on the disagreement of the coefficients so that $\omega_L$ is quantitatively absolutely continuous with respect to $\omega_{L_0}$ (see also \cite{FKP}). In our first main result we single out the latter property and characterize it in terms of the solvability of the Dirichlet problem or some other properties that certain solutions satisfy. In a nutshell, we show that such condition is equivalent to the fact that null solutions of $L$ have a good behavior with respect to $\omega_{L_0}$. The precise statement is as follows:

\begin{theorem}\label{thm:main}
Let $\Omega\subset\mathbb{R}^{n+1}$, $n\ge 2$, be a 1-sided NTA domain (cf. Definition \ref{def1.1nta}) satisfying the capacity density condition  (cf. Definition \ref{def-CDC}), and let $Lu=-\div(A\nabla u)$ and $L_0u=-\div(A_0\nabla u)$ be real (non-necessarily symmetric) elliptic operators. Bearing in mind the notions introduced in Definition~\ref{defs:several}, the following statements are equivalent:

\begin{list}{}{\usecounter{enumi}\leftmargin=1.2cm \labelwidth=1.2cm \itemsep=0.2cm \topsep=.2cm \renewcommand{\theenumi}{\alph{enumi}}}
 	
	\item[\textup{(a)\phantom{$'$}}] $\omega_L \in A_\infty(\partial \Omega, \omega_{L_0})$ (cf.~Definition~\ref{def:Ainfty}).
	
	\item[\textup{(b)\phantom{$'$}}] $L$ is  $L^{p}(\omega_{L_0})$-solvable for some $p\in (1,\infty)$.% (cf.~Definition~\ref{defs:several}).
	
	\item[\textup{(b)$'$}] $L$ is $L^{p}(\omega_{L_0})$-solvable for characteristic functions for some $p\in (1,\infty)$.%, that is,  \textup{(b)} holds for  $f=\mathbf{1}_S$ for any arbitrary Borel set $S\subset N\Delta_0$.

	\item[\textup{(c)\phantom{'}}] $L$ satisfies $\mathrm{CME}(\omega_{L_0})$.% (cf.~Definition~\ref{defs:several}).
	
	\item[\textup{(c)$'$}] $L$ satisfies $\mathrm{CME}(\omega_{L_0})$ for characteristic functions.%, that is, \textup{(c)} holds for all solutions of the form $u(X)=\omega_L^X(S)$, $X\in\Omega$, with $S\subset\pom$ being  an arbitrary Borel set.

	\item[\textup{(d)\phantom{$'$}}] $L$ satisfies $\mathcal{S}<\mathcal{N}$ in $L^q(\omega_{L_0})$ for some (or all) $q\in (0,\infty)$. %(cf.~Definition~\ref{defs:several}).
	
	\item[\textup{(d)$'$}] $L$ satisfies $\mathcal{S}<\mathcal{N}$ in $L^q(\omega_{L_0})$ for characteristic functions for some (or all) $q\in (0,\infty)$.%, that is, \textup{(d)} holds for all solutions of the form $u(X)=\omega_L^X(S)$, $X\in\Omega$, with $S\subset\pom$ being  an arbitrary Borel set.

	\item[\textup{(e)\phantom{$'$}}] $L$ is $\mathrm{BMO}(\omega_{L_0})$-solvable.%  (cf.~Definition~\ref{defs:several}).

	\item[\textup{(e)$'$}] $L$ is $\mathrm{BMO}(\omega_{L_0})$-solvable for characteristic functions.%, that is, 	\textup{(e)} holds for $f=\mathbf{1}_S$ and $u(X)=\omega_L^X(S)$, $X\in\Omega$, for any arbitrary Borel set $S\subset\pom$.

	\item[\textup{(f)\phantom{$'$}}]  $L$ is $\mathrm{BMO}(\omega_{L_0})$-solvable in the sense of \cite{HLe}.%  (cf.~Definition~\ref{defs:several}).

	\item[\textup{(f)$'$}] $L$ is $\mathrm{BMO}(\omega_{L_0})$-solvable in the sense of \cite{HLe} for characteristic functions.%, that is, 	\textup{(f)} holds for $f=\mathbf{1}_S$ and $u(X)=\omega_L^X(S)$, $X\in\Omega$, for any arbitrary Borel set $S\subset\pom$.

\end{list}
Furthermore, for any $p\in (1,\infty)$ there hold
\begin{equation*}
\textup{(a)${}_{p'}$ }	\omega_L \in RH_{p'}(\partial \Omega, \omega_{L_0}) \iff \textup{(b)${}_p$ }  L  \text{ is $L^{p}(\omega_{L_0})$-solvable},
\end{equation*}
\begin{equation*}
 \textup{(b)${}_p$ }  L  \text{ is $L^{p}(\omega_{L_0})$-solvable}\ \Longrightarrow\
	\textup{(b)${}_p'$ }  L  \text{ is $L^{p}(\omega_{L_0})$-solvable for characteristic functions},
\end{equation*}
and
\begin{equation*}
	\textup{(b)${}_p$ }  L  \text{ is $L^{p}(\omega_{L_0})$-solvable}\ \Longrightarrow\ \textup{(b)${}_q$ }  L  \text{ is $L^{q}(\omega_{L_0})$-solvable for all $q\ge p$.}
\end{equation*}
\end{theorem}

\medskip

\begin{remark}
Note that in Definition~\ref{defs:several}  the $L^{p}(\omega_{L_0})$-solvability depends on some fixed $\alpha$ and $N$. However, in the previous result what we prove is that if $\mathrm{(a)}$ holds then $\mathrm{(b)}$ is valid for all $\alpha$ and $N$. For the converse we see that if $\mathrm{(b)}$ holds for some $\alpha$ and $N$ then we get $\mathrm{(a)}$. This eventually says that if  $\mathrm{(b)}$ holds for some $\alpha$ and $N$, then it also holds for every $\alpha$ and $N$. The same occurs with  $\mathrm{(d)}$ where now there is only $\alpha$.
\end{remark}

As an immediate consequence of Theorem~\ref{thm:main}, if we take $L_0=L$, in which case we clearly have $\omega_L \in A_\infty(\partial \Omega, \omega_{L_0})$ (indeed, $\omega_L \in RH_p(\partial \Omega, \omega_{L_0})$ for any $1<p<\infty$), then we obtain the following estimates for the null solutions of $L$ (note that (ii) and (iii) coincide with \cite[Theorems~5.1 and 5.3]{AHMT} respectively):

\begin{corollary}\label{corol:main}
	Let $\Omega\subset\mathbb{R}^{n+1}$, $n\ge 2$, be a 1-sided NTA domain (cf. Definition \ref{def1.1nta}) satisfying the capacity density condition  (cf. Definition \ref{def-CDC}), and let $Lu=-\div(A\nabla u)$ be a real (non-necessarily symmetric) elliptic operator. Bearing in mind the notions introduced in Definition~\ref{defs:several}, the following statements hold:
	
	\begin{list}{\textup{(\theenumi)}}{\usecounter{enumi}\leftmargin=1.2cm \labelwidth=1.2cm \itemsep=0.2cm \topsep=.2cm \renewcommand{\theenumi}{\roman{enumi}}}
				
		\item $L$ is  $L^{p}(\omega_{L})$-solvable, and also $L^{p}(\omega_{L})$-solvable for characteristic functions, for all $p\in (1,\infty)$. % (cf.~Definition~\ref{defs:several}).
		
		\item  $L$ satisfies $\mathrm{CME}(\omega_{L})$.  		
		
		\item $L$ satisfies $\mathcal{S}<\mathcal{N}$ in $L^q(\omega_{L})$ for all $q\in (0,\infty)$.% (cf.~Definition~\ref{defs:several}). 			
		
		\item $L$ is $\mathrm{BMO}(\omega_{L})$-solvable, and also $\mathrm{BMO}(\omega_{L})$-solvable for characteristic functions.%, (cf.~Definition~\ref{defs:several}).

		\item $L$ is $\mathrm{BMO}(\omega_{L})$-solvable, and also $\mathrm{BMO}(\omega_{L})$-solvable for characteristic functions, in the sense of \cite{HLe}.% (cf.~Definition~\ref{defs:several}).
	\end{list}
\end{corollary}

\begin{remark}
We would like to emphasize that in $\mathrm{(i)}$ the $L^{p}(\omega_{L_0})$-solvability holds for all  $\alpha$ and $N$, the same occurs with  $\mathrm{(iii)}$ which holds for all $\alpha$, see Definition~\ref{defs:several}.
\end{remark}

Our second application is a direct consequence of \cite[Theorems~1.5, 1.10]{AHMT} and Theorem \ref{thm:main}:

%%%%%%%%%%%%%%%%%%%% COROLLARY COROLLARY COROLLARY %%%%%%%%%%%%%%%%%%%
\begin{corollary}\label{cor:Ainfty}
	Let $\Omega\subset\mathbb{R}^{n+1}$, $n\ge 2$, be a 1-sided NTA domain (cf. Definition \ref{def1.1nta}) satisfying the capacity density condition  (cf. Definition \ref{def-CDC}), and let $Lu=-\div(A\nabla u)$ and $L_0u=-\div(A_0\nabla u)$ be real (non-necessarily symmetric) elliptic operators.  Define
	\begin{equation}\label{eqn:discrepancy}
		\varrho(A, A_0)(X) :=\sup_{Y \in B(X, \delta(X)/2)} |A(Y) - A_0(Y)|, \qquad X \in \Omega,
	\end{equation}
and	
	\begin{equation*}
		\normmm{\varrho(A, A_0)} := \sup_{B} \sup_{B'} \frac{1}{\omega^{X_{\Delta}}_{L_0}(\Delta')}
		\iint_{B' \cap \Omega} \varrho(A, A_0)(X)^2 \frac{G_{L_0}(X_{\Delta}, X)}{\delta(X)^2} dX,
	\end{equation*}
	where $\Delta=B \cap \Omega$, $\Delta'=B' \cap \Omega$, and the sup is taken respectively over all balls
	$B=B(x, r)$ with $x \in \partial \Omega$ and $0<r<\diam(\partial \Omega)$, and $B'=B(x', r)$ with $x' \in 2 \Delta$
	and $0<r'<c_0r/4$, and $c_0$ is the Corkscrew constant.  We also define
	\begin{align*}
		\mathscr{A}_{\alpha}(\varrho(A, A_0))(x) := \left(\iint_{\Gamma^{\alpha}(x)}
		\frac{\varrho(A, A_0)(X)^2}{\delta(X)^{n+1}} dX \right)^{\frac12}, \qquad x \in \partial \Omega,
	\end{align*}
	where $\Gamma^{\alpha}(x) := \{X \in \Omega: |X-x|\leq (1+\alpha) \delta(X) \}$.
	
	If
	\begin{equation}\label{eq:AAAA}
		\normmm{\varrho(A, A_0)} < \infty \qquad\text{or}\qquad
		\mathscr{A}_{\alpha}(\varrho(A, A_0)) \in L^{\infty}(\partial \Omega, \omega_{L_0}),
	\end{equation}
	then all the properties \textup{(a)}--\textup{(f)$'$} in Theorem \ref{thm:main} are satisfied.
	
	Moreover, if given $1<p<\infty$, there exists $\varepsilon_p>0$ (depending only on dimension,  the 1-sided NTA and CDC constants, the ellipticity constants of $L_0$ and $L$, and $p$) such that if
	\[
	 \normmm{\varrho(A, A_0)} \leq\varepsilon_p \qquad\text{or}\qquad  \|\mathscr{A}_{\alpha}(\varrho(A, A_0))\|_{L^\infty(\omega_{L_0})}\le \varepsilon_p,
	\]
	then $\omega_L\in RH_{p'}(\pom,\omega_{L_0})$ and hence $L$ is $L^{q}(\omega_{L_0})$-solvable for $q\ge p$.
\end{corollary}
%%%%%%%%%%%%%%%%%%%% COROLLARY COROLLARY COROLLARY %%%%%%%%%%%%%%%%%%%

Our next goal is to state a qualitative version of Theorem~\ref{thm:main} in line with \cite{CMO}. The $A_\infty$ condition will turn into absolute continuity. The qualitative analog of $\mathcal{S}<\mathcal{N}$ is going to be that the conical square function satisfies $L^q$ estimates in some pieces of the boundary. On the other hand, as seen from the proof of Theorem~\ref{thm:main} (see Lemma~\ref{lemma:CME-S} and \eqref{S<CME}), the CME condition, more precisely, the left-hand side term of \eqref{CME:Linfty}  is connected with the local $L^2$-norm of the conical square function. Thus, the $L^2$-estimates for the conical square function are the qualitative version of CME. In turn, all these are equivalent to the simple fact that the truncated conical square function is finite almost everywhere with respect to the elliptic measure $\omega_{L_0}$.

\begin{theorem}\label{thm:abs}
	Let $\Omega\subset\mathbb{R}^{n+1}$, $n\ge 2$, be a 1-sided NTA domain (cf. Definition \ref{def1.1nta}) satisfying the capacity density condition  (cf. Definition \ref{def-CDC}).  There exists $\alpha_0>0$ (depending only on the 1-sided NTA and CDC constants) such that for each fixed $\alpha \geq \alpha_0$ and for every real (not necessarily symmetric) elliptic operators $L_0 u = -\div(A_0 \nabla u)$ and $Lu=-\div(A \nabla u)$ the following statements are equivalent:
	\begin{list}{}{\usecounter{enumi}\leftmargin=1cm \labelwidth=1cm \itemsep=0.2cm
			\topsep=.2cm \renewcommand{\theenumi}{\alph{enumi}}}
		
		\item[\textup{(a)\phantom{$'$}}] \label{list:abs-1} $\omega_{L_0} \ll \omega_L$ on $\partial \Omega$.
		
		\item[\textup{(b)\phantom{$'$}}]\label{list:abs-2} $\partial \Omega=\bigcup_{N \geq 0} F_N$, where $\omega_{L_0}(F_0)=0$, for each $N \geq 1$,
		$F_N=\partial \Omega \cap \partial \Omega_N$ for some bounded 1-sided NTA domain $\Omega_N \subset \Omega$ satisfying the capacity density condition, and $\mathcal{S}^{\alpha}_r u \in L^q(F_N, \omega_{L_0})$ for every weak solution $u \in W_{\loc}^{1,2}(\Omega) \cap L^{\infty}(\Omega)$ of $Lu=0$ in $\Omega$, for all (or for some) $r>0$, and for all (or for some) $q\in(0,\infty)$.
		
		\item[\textup{(b)$'$}]\label{list:abs-2prime} $\partial \Omega=\bigcup_{N \geq 0} F_N$, where $\omega_{L_0}(F_0)=0$, for each $N \geq 1$,
		$F_N=\partial \Omega \cap \partial \Omega_N$ for some bounded 1-sided NTA domain $\Omega_N \subset \Omega$ satisfying the capacity density condition, and $\mathcal{S}^{\alpha}_r u \in L^q(F_N, \omega_{L_0})$ where $u(X)=\omega_L^X(S)$, $X\in\Omega$,
		for any arbitrary Borel set $S\subset\pom$, for all (or for some) $r>0$, and for all (or for some) $q\in(0,\infty)$.

		\item[\textup{(c)\phantom{$'$}}]\label{list:abs-3} $\mathcal{S}^{\alpha}_r u(x)<\infty$ for $\omega_{L_0}$-a.e.~$x \in \partial \Omega$, for every weak solution $u \in W_{\loc}^{1,2}(\Omega)\cap L^{\infty}(\Omega)$ of  $Lu=0$ in $\Omega$ and for all (or for some) $r>0$.
		
		\item[\textup{(c)$'$}]\label{list:abs-3prime} $\mathcal{S}^{\alpha}_r u(x)<\infty$ for $\omega_{L_0}$-a.e.~$x \in \partial \Omega$ where $u(X)=\omega_L^X(S)$, $X\in\Omega$, 		  for any arbitrary Borel set $S\subset\pom$, and for all (or for some) $r>0$.
		
		\item[\textup{(d)\phantom{$'$}}]\label{list:abs-4} For every weak solution $u \in W_{\loc}^{1,2}(\Omega)\cap L^{\infty}(\Omega)$ of
		$Lu=0$ in $\Omega$ and for $\omega_{L_0}$-a.e.~$x \in \partial \Omega$ there exists $r_x>0$ such that
		$\mathcal{S}^{\alpha}_{r_x} u(x)<\infty$.

		\item[\textup{(d)$'$}]\label{list:abs-4prime} For every Borel set $S\subset\pom$ and for $\omega_{L_0}$-a.e.~$x \in \partial \Omega$ there exists $r_x>0$ such that
		  $\mathcal{S}^{\alpha}_{r_x} u(x)<\infty$, where $u(X)=\omega_L^X(S)$, $X\in\Omega$.

	\end{list}
\end{theorem}

Our first application of the previous result is a qualitative version of \cite[Theorem~1.10]{AHMT}:

%%%%%%%%%%%%%%%%%%%%%% THEOREM THEOREM THEOREM %%%%%%%%%%%%%%%%%%%%%%
\begin{theorem}\label{thm:wL}
Let $\Omega\subset\mathbb{R}^{n+1}$, $n\ge 2$, be a 1-sided NTA domain (cf. Definition \ref{def1.1nta}) satisfying the capacity density condition  (cf. Definition \ref{def-CDC}).  There exists $\alpha_0>0$ (depending only on the 1-sided NTA and CDC constants) such that, if the real (not necessarily symmetric) elliptic operators $L_0 u = -\div(A_0 \nabla u)$ and $L u = -\div(A \nabla u)$ satisfy for some $\alpha \ge\alpha_0$ and for some $r>0$
	\begin{align}\label{eq:rhoAA}
		\iint_{\Gamma^{\alpha}_{r}(x)} \frac{\varrho(A, A_0)(X)^2}{\delta(X)^{n+1}} dX < \infty,
		\qquad \text{for $\omega_{L_0}$-a.e.~} x \in \partial \Omega,
	\end{align}
	where $\varrho(A, A_0)$ is as in \eqref{eqn:discrepancy}, then $\w_{L_0} \ll \w_L$.
\end{theorem}
%%%%%%%%%%%%%%%%%%%%%% THEOREM THEOREM THEOREM %%%%%%%%%%%%%%%%%%%%%%

To present another application of Theorem \ref{thm:abs}, we introduce some notation. For any  real (not necessarily symmetric) elliptic operator $Lu=-\div(A \nabla u)$, we let $L^{\top}$ denote the transpose of $L$, and let $L^{\rm sym}=\frac{L+L^{\top}}{2}$ be the symmetric part of $L$. These are respectively the divergence form elliptic operators with associated matrices $A^\top$ (the transpose of $A$)  and $A^{\rm sym}=\frac{A+A^{\top}}{2}$.

%%%%%%%%%%%%%%%%%%%%%% THEOREM THEOREM THEOREM %%%%%%%%%%%%%%%%%%%%%%
\begin{theorem}\label{thm:wLT}
Let $\Omega\subset\mathbb{R}^{n+1}$, $n\ge 2$, be a 1-sided NTA domain (cf. Definition \ref{def1.1nta}) satisfying the capacity density condition  (cf. Definition \ref{def-CDC}).  There exists $\alpha_0>0$ (depending only on the 1-sided NTA and CDC constants) such that, if $Lu=-\div(A \nabla u)$ is a real (not necessarily symmetric) elliptic operator, and we assume that $(A-A^{\top}) \in \Lip_{\loc}(\Omega)$ and that for some $\alpha\ge \alpha_0$ and for some $r>0$ one has
	\begin{equation}\label{eq:divCAA}
		\mathscr{F}^{\alpha}_r(x; A) :=
		\iint_{\Gamma^{\alpha}_{r}(x)}
		\left|\div_C A-A^{\top})(X)\right|^2 \delta(X)^{1-n} dX<\infty,
		\qquad \text{for $\w_L$-a.e.~} x \in \partial \Omega,
	\end{equation}
	where
	\begin{align*}
		\div_C(A-A^{\top})(X):=\bigg(\sum_{i=1}^{n+1} \partial_i(a_{i,j}-a_{j,i})(X) \bigg)_{1 \leq j \leq n+1},
		\qquad X \in \Omega,
	\end{align*}
	then $\w_L \ll \w_{L^{\top}}$ and $\w_L \ll \w_{L^{\rm sym}}$.
	
	Moreover, if
	\begin{equation}\label{eq:www}
		\mathscr{F}^{\alpha}_r(x; A)<\infty,
		\qquad \text{for $\w_L $-a.e.~and \, $\w_{L^{\top}}$-a.e.~} x \in \partial \Omega,
	\end{equation}
	then $\w_L \ll \w_{L^{\top}} \ll \w_L \ll \w_{L^{\rm sym}}$.
	
\end{theorem}
%%%%%%%%%%%%%%%%%%%%%% THEOREM THEOREM THEOREM %%%%%%%%%%%%%%%%%%%%%%

 The structure of this paper is as follows. Section~\ref{section:prelim} contains some preliminaries, definitions, and tools that will be used throughout. Also, for convenience of the reader, we gather in Section~\ref{section:PDE} several facts concerning elliptic measures and Green functions which can be found in the upcoming \cite{HMT:book}. The proof of Theorem~\ref{thm:main} is in Section~\ref{section:proof-main}.  Section \ref{sec:abs} is devoted to proving Theorem \ref{thm:abs}. In Section \ref{sec:perturbation}, we will present the proofs of Theorems \ref{thm:wL} and \ref{thm:wLT} which follow easily from a more general perturbation result which is interesting in its own right.

We note that some interesting related work has  been carried out while this manuscript was in preparation due to Feneuil and Poggi \cite{FP}. This work can be particularized to our setting and contains some results which overlap with ours. First, \cite[Theorem~1.22]{FP} corresponds to $\mathrm{(c)'}\ \Longrightarrow\ \mathrm{(a)}$ in Theorem~\ref{thm:main}. It should be mentioned that both arguments use the ideas originated in \cite{KKiPT} (see also \cite{KKoPT}) which present some problems when extended to the 1-sided NTA setting. Namely, elliptic measure may not always be a probability and also it could happen that for a uniformly bounded number of generations the dyadic children of a given cube may agree with that cube. These two issues have been carefully addressed in \cite[Lemma 3.10]{CHMT} (see Lemma~\ref{sq-function->M} with $\beta>0$) and although such a result is stated in the setting of 1-sided CAD it is straightforward to see that it readily adapts to our case. Our proof of $\mathrm{(c)'}\ \Longrightarrow\ \mathrm{(a)}$ in Theorem~\ref{thm:main} follows easily from that lemma. Second,  \cite[Theorem~1.27]{FP} (see also \cite[Corollary~1.33]{FP}) shows $\mathrm{(d)}$ in Theorem~\ref{thm:main}  with $q=2$ for a class of perturbations of $L$. In our setting, we are showing that $\mathrm{(d)}$ follows if  $\mathrm{(a)}$ holds for any given operator $L$ (whether or not it is a generalized perturbation of $L_0$.)

\section{Preliminaries}\label{section:prelim}

\subsection{Notation and conventions}\label{section:notation}

\begin{list}{$\bullet$}{\leftmargin=0.4cm  \itemsep=0.2cm}
	
	\item We use the letters $c,C$ to denote harmless positive constants, not necessarily the same at each occurrence, which depend only on dimension and the
	constants appearing in the hypotheses of the theorems (which we refer to as the ``allowable parameters'').  We shall also sometimes write $a\lesssim b$ and $a \approx b$ to mean, respectively, that $a \leq C b$ and $0< c \leq a/b\leq C$, where the constants $c$ and $C$ are as above, unless
	explicitly noted to the contrary.   Unless otherwise specified upper case constants are greater than $1$  and lower case constants are smaller than $1$. In some occasions it is important to keep track of the dependence on a given parameter $\gamma$, in that case we write $a\lesssim_\gamma b$ or $a\approx_\gamma b$ to emphasize  that the implicit constants in the inequalities depend on $\gamma$.
	
	\item  Our ambient space is $\ree$, $n\ge 2$.
	
	\item Given $E\subset\ree$ we write $\diam(E)=\sup_{x,y\in E}|x-y|$ to denote its diameter.
	
	\item Given an open set $\Omega \subset \ree$, we shall
	use lower case letters $x,y,z$, etc., to denote points on $\partial \Omega$, and capital letters
	$X,Y,Z$, etc., to denote generic points in $\ree$ (especially those in $\ree\setminus \partial\Omega$).
	
	\item The open $(n+1)$-dimensional Euclidean ball of radius $r$ will be denoted
	$B(x,r)$ when the center $x$ lies on $\partial \Omega$, or $B(X,r)$ when the center
	$X \in \ree\setminus \partial\Omega$.  A \textit{surface ball} is denoted
	$\Delta(x,r):= B(x,r) \cap\partial\Omega$, and unless otherwise specified it is implicitly assumed that $x\in\pom$.
	
	\item If $\pom$ is bounded, it is always understood (unless otherwise specified) that all surface balls have radii controlled by the diameter of $\pom$, that is, if $\Delta=\Delta(x,r)$ then $r\lesssim \diam(\pom)$. Note that in this way $\Delta=\pom$ if $\diam(\pom)<r\lesssim \diam(\pom)$.
	
	%\item Given a Euclidean ball $B$ or surface ball $\Delta$, its radius will be denoted
	%$r_B$ or $r_\Delta$, respectively.
	
	%\item Given a Euclidean or surface ball $B= B(X,r)$ or $\Delta = \Delta(x,r)$, its concentric
	%dilate by a factor of $\kappa >0$ will be denoted
	%by $\kappa B := B(X,\kappa r)$ or $\kappa \Delta := \Delta(x,\kappa r).$

	\item For $X \in \ree$, we set $\delta(X):= \dist(X,\partial\Omega)$.
	
	\item We let $\mathcal{H}^n$ denote the $n$-dimensional Hausdorff measure.
	%	$\sigma := H^n|_{\partial\Omega}$ denote the surface measure on $\partial \Omega$.
	
	\item For a Borel set $A\subset \ree$, we let $\mathbf{1}_A$ denote the usual
	indicator function of $A$, i.e. $\mathbf{1}_A(X) = 1$ if $X\in A$, and $\mathbf{1}_A(X)= 0$ if $X\notin A$.

	%\item For a Borel set $A\subset \ree$,  we let $\interiorior(A)$ denote the interior of $A$.
	%If $A\subset \partial\Omega$, then $\interiorior(A)$ will denote the relative interior, i.e., the largest relatively open set in $\partial\Omega$ contained in $A$.  Thus, for $A\subset \partial\Omega$,
	%the boundary is then well defined by $\partial A := \overline{A} \setminus {\rm int}(A)$.
	
	%\item For a Borel set $A$, we denote by $\mathcal{C}(A)$ the space of continuous functions on
	%$A$, by $\mathcal{C}_0(A)$ the subspace of $\mathcal{C}(A)$
	%with compact support in $A$.
	
	%	
	%	\item For a Borel subset $A\subset\partial\Omega$, with $0<\sigma(A)<\infty$, we
	%	set $\aver{A} f d\sigma := \sigma(A)^{-1} \int_A f d\sigma$.
	
	\item We shall use the letter $I$ (and sometimes $J$)
	to denote a closed $(n+1)$-dimensional Euclidean cube with sides
	parallel to the coordinate axes, and we let $\ell(I)$ denote the side length of $I$.
	We use $Q$ to denote  dyadic ``cubes''
	on $E$ or $\partial \Omega$.  The
	latter exist as a consequence of Lemma \ref{lemma:dyadiccubes} below.
	
	%	\item Given a domain $\om$, and an elliptic operator $L$, we let $\omega_L^X$ denote the $L$-elliptic measure for $\om$
	%	with pole at $X$, and if $\hm_L^X\ll\sigma$, we let $k_L^X:=d\hm_L^X/d\sigma$ be the corresponding Poisson kernel.
	%	When the operator $L$ is understood, we will at times suppress its appearance in the notation, and write simply
	%	$\hm^X$, $k^X$ in place of $\hm_L^X$ and $k_L^X$.
\end{list}

\subsection{Some definitions}\label{section:defs}
\begin{definition}[\bf Corkscrew condition]\label{def1.cork}
	Following \cite{JK82}, we say that a domain $\Omega\subset \ree$
	satisfies the \textit{Corkscrew condition} if for some uniform constant $0<c_0<1$ and
	for every $x\in \partial\Omega$ and $0<r<\diam(\partial\Omega)$, if we write $\Delta:=\Delta(x,r)$, there is a ball
	$B(X_\Delta,c_0r)\subset B(x,r)\cap\Omega$.  The point $X_\Delta\subset \Omega$ is called
	a \textit{Corkscrew point relative to} $\Delta$ (or, relative to $B$). We note that  we may allow
	$r<C\diam(\pom)$ for any fixed $C$, simply by adjusting the constant $c_0$.
\end{definition}

%\begin{remark}\label{rem:modified-CKS}
%In view of the previous definition, given $\Delta=\Delta(x,r)$ with $x\in \partial\Omega$ and
%$0<r<\diam(\partial\Omega)/(8c_0^{-1})$ we set $\widehat{X}_\Delta=X_{8 c_0e^{-1}\Delta}$ so that $\widehat{X}_\Delta\in\Omega\setminus B(x,8r)$.
%\end{remark}

\begin{definition}[\bf Harnack Chain condition]\label{def1.hc}
	Again following \cite{JK82}, we say
	that $\Omega$ satisfies the \textit{Harnack Chain condition} if there are uniform constants $C_1,C_2>1$ such that for every pair of points $X, X'\in \Omega$
	there is a chain of balls $B_1, B_2, \dots, B_N\subset \Omega$ with $N \leq  C_1(2+\log_2^+ \Pi)$
	where
	\begin{equation}\label{cond:Lambda}
		\Pi:=\frac{|X-X'|}{\min\{\delta(X), \delta(X')\}}
	\end{equation}
	such that $X\in B_1$, $X'\in B_N$, $B_k\cap B_{k+1}\neq\emptyset$ and for every $1\le k\le N$
	\begin{equation}\label{preHarnackball}
		C_2^{-1} \diam(B_k) \leq \dist(B_k,\partial\Omega) \leq C_2 \diam(B_k).
	\end{equation}
	The chain of balls is called a \textit{Harnack Chain}.
\end{definition}

We note that in the context of the previous definition if $\Pi\le 1$ we can trivially form the Harnack chain $B_1=B(X,3\delta(X)/5)$ and $B_2=B(X', 3\delta(X')/5)$ where \eqref{preHarnackball} holds with $C_2=3$. Hence the Harnack chain condition is non-trivial only when $\Pi> 1$.

\begin{definition}[\bf 1-sided NTA and NTA]\label{def1.1nta}
	We say that a domain $\Omega$ is a \textit{1-sided non-tangentially accessible domain} (1-sided NTA)  if it satisfies both the Corkscrew and Harnack Chain conditions.
	Furthermore, we say that $\Omega$ is a \textit{non-tangentially accessible domain} 	(NTA  domain)	if it is a 1-sided NTA domain and if, in addition, $\Omega_{\rm ext}:= \ree\setminus \overline{\Omega}$ also satisfies the Corkscrew condition.
\end{definition}
\begin{remark}
	In the literature, 1-sided NTA domains are also called \textit{uniform domains}. We remark that the 1-sided NTA condition is a quantitative form of openness and path connectedness.
\end{remark}

\begin{definition}[\bf Ahlfors  regular]\label{def1.ADR}
	We say that a closed set $E \subset \ree$ is \textit{$n$-dimensional Ahlfors regular} (AR for short) if
	there is some uniform constant $C_1>1$ such that
	\begin{equation} \label{eq1.ADR}
		C_1^{-1}\, r^n \leq \mathcal{H}^n(E\cap B(x,r)) \leq C_1\, r^n,\qquad x\in E, \quad 0<r<\diam(E).
	\end{equation}
\end{definition}

\begin{definition}[\bf 1-sided CAD and CAD]\label{defi:CAD}
	A \textit{1-sided chord-arc domain} (1-sided CAD) is a 1-sided NTA domain with AR boundary.
	A \textit{chord-arc domain} (CAD) is an NTA domain with AR boundary.
\end{definition}

We next recall the definition of the capacity of a set.  Given an open set $D\subset \ree$ (where we recall that we always assume that $n\ge 2$) and a compact set $K\subset D$  we define the capacity of $K$ relative to $D$ as
\[
\Cap(K, D)=\inf\left\{\iint_{D} |\nabla v(X)|^2 dX:\, \, v\in \mathscr{C}^{\infty}_{c}(D),\, v(x)\geq  1 \mbox{ in }K\right\}.
\]

\begin{definition}[\textbf{Capacity density condition}]\label{def-CDC}
	An open set $\Omega$ is said to satisfy the \textit{capacity density condition} (CDC for short) if there exists a uniform constant $c_1>0$ such that
	\begin{equation}\label{eqn:CDC}
		\frac{\Cap(\overline{B(x,r)}\setminus \Omega, B(x,2r))}{\Cap(\overline{B(x,r)}, B(x,2r))} \geq c_1
	\end{equation}
	for all $x\in \partial\Omega$ and $0<r<\diam(\pom)$.
\end{definition}

The CDC is also known as the uniform 2-fatness as studied by Lewis in \cite{Lew}. Using \cite[Example 2.12]{HKM} one has that
\begin{equation}\label{cap-Ball}
	\Cap(\overline{B(x,r)}, B(x,2r))\approx r^{n-1}, \qquad \mbox{for all $x\in\ree$ and $r>0$},
\end{equation}
and hence the CDC is a quantitative version of the Wiener regularity, in particular every $x\in\pom$ is Wiener regular. It is easy to see that the exterior Corkscrew condition implies CDC. Also, it was proved in \cite[Section 3]{Zhao} and \cite[Lemma 3.27]{HLMN} that a set with Ahlfors regular boundary satisfies the capacity density condition with constant $c_1$ depending only on $n$ and the Ahlfors regular constant.

\subsection{Dyadic grids and sawtooths}\label{section:dyadic}

In this section we introduce a dyadic grid from \cite[Lemma~2.33]{AHMT} along the lines of that obtained in \cite{C} but using the dyadic structure from \cite{HK1, HK2, HMMM}:

\begin{lemma}[\textbf{Existence and properties of the ``dyadic grid''}, {\cite[Lemma~2.33]{AHMT}}]\label{lemma:dyadiccubes}
	Let $E\subset\re^{n+1}$ be a closed set. Then there exists a constant $C\ge 1$ depending just on $n$ such that for each $k\in\mathbb{Z}$ there is a collection of Borel sets  (called ``cubes'')
	$$
	\dd_k:=\big\{Q_j^k\subset E:\ j\in\mathfrak{J}_k\big\},
	$$
	where $\mathfrak{J}_k$ denotes some (possibly finite) index set depending on $k$ satisfying:
	\begin{list}{$(\theenumi)$}{\usecounter{enumi}\leftmargin=1cm \labelwidth=1cm \itemsep=0.2cm \topsep=.2cm \renewcommand{\theenumi}{\alph{enumi}}}
		\item $E=\bigcup_{j\in \mathfrak{J}_k}Q_j^k$ for each $k\in\mathbb{Z}$.
		\item If $m\le k$ then either $Q_j^k \subset Q_i^m$ or $Q_i^m\cap Q_j^k=\emptyset$.
		\item For each $k\in\mathbb{Z}$, $j\in\mathfrak{J}_k$, and $m<k$, there is a unique $i\in\mathfrak{J}_m $ such that $Q_j^k\subset Q_i^m$.
		\item For each  $k\in\mathbb{Z}$, $j\in\mathfrak{J}_k$ there is $x_j^k\in E$ such that
		\[B(x_j^k, C^{-1}2^{-k})\cap E\subset Q_j^k \subset B(x_j^k, C 2^{-k})\cap E.\]

		%		\item $\diam(Q_j^k)< C\,2^{-k}$.
		%		\item Each $Q_j^k$ contains some ``surface ball'' $\Delta(x_j^k,C^{-1}2^{-k})=B(x_j^k,C^{-1}2^{-k})\cap E$ with $x_j^k\in E$.
		%\item $H^n\big(\big\{x\in Q_j^k:\,\dist(x,E\setminus Q_j^k)\leq\tau 2^{-k}\big\}\big)\leq C_1\tau^\eta H^n(Q_j^k)$, for all $j,k\in\mathbb{Z}$ and for all $\tau\in(0,a_0)$.
	\end{list}
\end{lemma}

In what follows given $B=B(x,r)$ with $x\in E$ we will denote $\Delta=\Delta(x,r)=B\cap E$. A few remarks are in order concerning this lemma.  Note that  within the same generation (that is, within each $\dd_k$) the cubes are pairwise disjoint (hence, there are no repetitions). On the other hand, we allow repetitions in the different generations, that is, we could have that $Q\in\dd_k$ and $Q'\in\dd_{k-1}$ agree. Then, although $Q$ and $Q'$ are the same set,  as cubes we understand that they are different. In short, it is then understood that $\dd$ is an indexed collection of sets where repetitions of sets are allowed in the different generations but not within the same generation. With this in mind, we can give a proper definition of the ``length'' of a cube (this concept has no geometric meaning in this context). For every $Q\in\dd_k$, we set $\ell(Q)=2^{-k}$, which is called the ``length'' of $Q$. Note that the ``length'' is well defined when considered on $\dd$, but it is not well-defined on the family of sets induced by $\dd$. It is important to observe that the ``length'' refers to the way the cubes are organized in the dyadic grid.  It is clear from $(d)$ that $\diam(Q)\lesssim \ell(Q)$. When $E=\pom$, with $\Omega$ being a 1-sided NTA domain satisfying the CDC condition, the converse holds, hence $\diam(Q)\approx \ell(Q)$, see \cite[Remark~2.73]{AHMT}. This means that the ``length'' is related to the diameter of the cube.

Let us observe that if $E$ is bounded and $k\in\ZZ$ is such that $\diam(E)<C^{-1}2^{-k}$, then there cannot be two distinct cubes in $\dd_k$. Thus, $\dd_k=\{Q^k\}$ with $Q^k=E$. Therefore,  we are going to ignore those $k\in\mathbb{Z}$ such that $2^{-k}\gtrsim\diam(E)$. Hence, we shall denote by $\dd(E)$ the collection of all relevant $Q_j^k$, i.e.,
$$
\dd(E):=\bigcup_k\dd_k,
$$
where, if $\diam(E)$ is finite, the union runs over those $k\in\mathbb{Z}$ such that $2^{-k}\lesssim\diam(E)$.   We write $\Xi=2C^2$, with $C$ being the constant in Lemma \ref{lemma:dyadiccubes}, which is purely dimensional. For $Q\in\dd(E)$ we will set $k(Q)=k$ if $Q\in\dd_k$. Property $(d)$ implies that for each cube $Q\in\dd(E)$, there exist $x_Q\in E$ and $r_Q$, with $\Xi^{-1}\ell(Q)\leq r_Q\leq\ell(Q)$ (indeed $r_Q= (2C)^{-1}\ell(Q)$), such that
\begin{equation}\label{deltaQ}
	\Delta(x_Q,2r_Q)\subset Q\subset\Delta(x_Q,\Xi r_Q).
\end{equation}
We shall denote these balls and surface balls by
\begin{equation}\label{deltaQ2}
	B_Q:=B(x_Q,r_Q),\qquad\Delta_Q:=\Delta(x_Q,r_Q),
\end{equation}
\begin{equation}\label{deltaQ3}
	\widetilde{B}_Q:=B(x_Q,\Xi r_Q),\qquad\widetilde{\Delta}_Q:=\Delta(x_Q,\Xi r_Q),
\end{equation}
and we shall refer to the point $x_Q$ as the ``center'' of $Q$.

Let $Q\in\dd_k$ and consider  the family of its dyadic children $\{Q'\in \dd_{k+1}: Q'\subset Q\}$. Note that for any two distinct children $Q', Q''$, one has $|x_{Q'}-x_{Q''}|\ge r_{Q'}=r_{Q''}=r_Q/2$, otherwise $x_{Q''}\in Q''\cap \Delta_{Q'}\subset Q''\cap Q'$, contradicting the fact that $Q'$ and $Q''$ are disjoint. Also $x_{Q'}, x_{Q''}\in Q\subset \Delta(x_Q,\Xi r_Q)$, hence by the geometric doubling property we have a purely dimensional bound for the number of such $x_{Q'}$ and hence the number of dyadic children of a given dyadic cube is uniformly bounded.

We next introduce the ``discretized Carleson region'' relative to $Q\in \dd(E)$, $\dd_{Q}=\{Q'\in\dd:Q'\subset Q\}$. Let $\mathcal{F}=\{Q_i\}\subset\dd(E)$ be a family of pairwise disjoint cubes. The ``global discretized sawtooth'' relative to $\mathcal{F}$ is the collection of cubes $Q\in\dd(E)$ that are not contained in any $Q_i\in\mathcal{F}$, that is,
\[
\dd_\mathcal{F}:=\dd\setminus\bigcup_{Q_i\in\mathcal{F}}\dd_{Q_i}.
\]
For a given $Q\in\dd(E)$, the ``local discretized sawtooth'' relative to $\mathcal{F}$ is the collection of cubes in $\dd_Q$ that are not contained in any $Q_i\in\mathcal{F}$ or, equivalently,
\[
\dd_{\mathcal{F},Q}:=\dd_{Q}\setminus\bigcup_{Q_i\in\mathcal{F}}\dd_{Q_i}=\dd_\mathcal{F}\cap\dd_Q.
\]
We also allow $\F$ to be the null set in which case $\dd_{\tinyemptyset}=\dd(E)$ and $\dd_{\tinyemptyset,Q}=\dd_Q$.

In the sequel, $\Omega\subset\re^{n+1}$, $n\geq 2$, will be a 1-sided NTA domain satisfying the CDC. Write $\dd=\dd(\pom)$ for the dyadic grid obtained from Lemma \ref{lemma:dyadiccubes} with $E=\pom$. In \cite[Remark~2.73]{AHMT} it is shown that under the present assumptions one has that $\diam(\Delta)\approx r_{\Delta}$ for every surface ball $\Delta$ and $\diam(Q)\approx\ell(Q)$ for every $Q\in\dd$. Given $Q\in\dd$ we define the ``Corkscrew point relative to $Q$'' as $X_Q:=X_{\Delta_Q}$. We note that
$$
\delta(X_Q)\approx\dist(X_Q,Q)\approx\diam(Q).
$$

We also introduce the ``geometric'' Carleson regions and sawtooths.  Given $Q\in\dd$ we want to define some associated regions which inherit the good properties of $\Omega$. Let $\mathcal{W}=\mathcal{W}(\Omega)$ denote a collection of (closed) dyadic Whitney cubes of $\Omega\subset\re^{n+1}$, so that the cubes in $\mathcal{W}$
form a covering of $\Omega$ with non-overlapping interiors, and satisfy
\begin{equation}\label{constwhitney}
	4\diam(I)\leq\dist(4I,\partial\Omega)\leq\dist(I,\partial\Omega)\leq 40\diam(I),\qquad\forall I\in\mathcal{W},
\end{equation}
and
$$
\diam(I_1)\approx\diam(I_2),\,\text{ whenever }I_1\text{ and }I_2\text{ touch}.
$$
Let $X(I)$ denote the center of $I$, let $\ell(I)$ denote the side length of $I$, and write $k=k_I$ if $\ell(I)=2^{-k}$.

Given $0<\lambda<1$ and $I\in\mathcal{W}$ we write $I^*=(1+\lambda)I$ for the ``fattening'' of $I$. By taking $\lambda$ small enough, we can arrange matters, so that, first, $\dist(I^*,J^*)\approx\dist(I,J)$ for every $I,J\in\mathcal{W}$. Secondly, $I^*$ meets $J^*$ if and only if $\partial I$ meets $\partial J$ (the fattening thus ensures overlap of $I^*$ and $J^*$ for any pair $I,J\in\mathcal{W}$ whose boundaries touch, so that the Harnack Chain property then holds locally in $I^*\cup J^*$, with constants depending upon $\lambda$). By picking $\lambda$ sufficiently small, say $0<\lambda<\lambda_0$, we may also suppose that there is $\tau\in(\frac12,1)$ such that for distinct $I,J\in\mathcal{W}$, we have that $\tau J\cap I^*=\emptyset$. In what follows we will need to work with dilations $I^{**}=(1+2\lambda)I$ or $I^{***}=(1+4\lambda)I$, and in order to ensure that the same properties hold we further assume that $0<\lambda<\lambda_0/4$.

Given $\vartheta\in\mathbb{N}$,  for every cube $Q \in \dd$ we set
\begin{equation}\label{eq:WQ}
	\W_Q^\vartheta :=\left\{I \in \W: 2^{-\vartheta}\ell(Q) \leq \ell(I) \leq 2^\vartheta\ell(Q), \text { and } \dist(I, Q) \leq 2^\vartheta \ell(Q) \right\}.
\end{equation}
We will choose $\vartheta\ge \vartheta_0$,  with $\vartheta_0$ large enough depending on the constants of the Corkscrew condition (cf. Definition  \ref{def1.cork}) and in the dyadic cube construction
(cf. Lemma \ref{lemma:dyadiccubes}), so that $X_Q \in I$ for some $I \in \W_Q^\vartheta$, and for each dyadic child $Q^j$ of $Q$, the respective corkscrew points
$X_{Q^j} \in I^j$ for some $I^j  \in \W_Q^\vartheta$. Moreover, we may always find an $I \in \W_Q^\vartheta$ with the slightly more precise
property that $\ell(Q)/2 \leq \ell(I) \leq \ell(Q)$ and
\begin{equation*}
	\W_{Q_1}^\vartheta \cap \W_{Q_2}^\vartheta \neq \emptyset, \quad \text { whenever }
	1 \leq \frac{\ell(Q_2)}{\ell(Q_1)} \leq 2, \text { and } \dist(Q_1, Q_2) \leq 1000 \ell(Q_2).
\end{equation*}

For each $I \in \W_Q^\vartheta$, we form a Harnack chain from the center $X(I)$ to the Corkscrew point $X_Q$ and call it $H(I)$.
We now let $\W_{Q}^{\vartheta, *}$ denote the collection of all Whitney cubes which meet at least one ball in the Harnack chain $H(I)$
with $I \in \W_Q^\vartheta$, that is,
\begin{equation*}
	\W_{Q}^{\vartheta, *}:=\{J \in \W: \text{ there exists } I \in \W_Q^\vartheta \text{ such that } H(I) \cap J \neq \emptyset\}.
\end{equation*}
We also define
\begin{equation*}
	U_{Q}^\vartheta :=\bigcup_{I \in \W_{Q}^{\vartheta, *}}(1+\lambda) I=: \bigcup_{I \in \W_{Q}^{\vartheta, *}} I^{*}.
\end{equation*}
By construction, we then have that
\begin{equation*}
	\W_{Q}^\vartheta  \subset \W_{Q}^{\vartheta, *} \subset \W \quad \text{and}\quad X_Q \in U_Q^{\vartheta}, \quad X_{Q^{j}} \in U_{Q}^{\vartheta},
\end{equation*}
for each child $Q^j$ of $Q$. It is also clear that there is a uniform constant $k^*$ (depending only on the
1-sided CAD constants and $\vartheta$) such that
\begin{align*}
	2^{-k^*} \ell(Q) \leq \ell(I) \leq 2^{k^*}\ell(Q), &\quad \forall\,I \in \W_{Q}^{\vartheta, *},
	\\%%%%%%%%%%
	X(I) \rightarrow_{U_Q^\vartheta} X_Q, &\quad \forall\,I \in \W_{Q}^{\vartheta, *},
	\\%%%%%%%%%
	\dist(I, Q) \leq 2^{k^*} \ell(Q), &\quad \forall\,I \in \W_{Q}^{\vartheta, *}.
\end{align*}
Here, $X(I) \to_{U_Q^\vartheta} X_Q$ means that the interior of $U_Q^\vartheta$ contains all balls in Harnack Chain (in $\Omega$) connecting $X(I)$ to $X_Q$, and moreover, for any point $Z$ contained in any ball in the Harnack Chain, we have
$\dist(Z, \partial \Omega) \approx \dist(Z, \Omega \setminus U_Q^\vartheta)$ with uniform control of implicit constants. The
constant $k^*$ and the implicit constants in the condition $X(I) \to_{U_Q^\vartheta} X_Q$, depend at most on the allowable
parameters, on $\lambda$, and on $\vartheta$. Moreover, given $I \in \W$ we have that $I \in \W^{\vartheta,*}_{Q_I}$, where $Q_I \in \dd$ satisfies $\ell(Q_I)=\ell(I)$, and contains any fixed $\widehat{y} \in \partial \Omega$ such that
$\dist(I, \partial \Omega)=\dist(I, \widehat{y})$. The reader is referred to \cite{HM1, HMT:book} for full details. We note however that in \cite{HM1} the parameter $\vartheta$ is fixed. Here we need to allow $\vartheta$ to depend on the aperture of the cones and hence it is convenient to include the superindex $\vartheta$.

For a given $Q\in\dd$, the ``Carleson box'' relative to $Q$ is defined by
$$
T_Q^\vartheta:=\interiorior\bigg(\bigcup_{Q'\in\dd_Q}U_{Q'}^\vartheta\bigg).
$$
For a given family $\mathcal{F}=\{Q_i\}\subset\dd$ of pairwise disjoint cubes and a given $Q\in\dd$, we define the ``local sawtooth region'' relative to $\mathcal{F}$ by
\begin{equation}
	\label{defomegafq}
	\Omega_{\mathcal{F},Q}^\vartheta:=\interiorior\bigg(\bigcup_{Q'\in\dd_{\mathcal{F},Q}}U_{Q'}^\vartheta\bigg)=\interiorior\bigg(\bigcup_{I\in\mathcal{W}_{\mathcal{F},Q}^\vartheta} I^*\bigg),
	\end{equation}
where $\mathcal{W}_{\mathcal{F},Q}^\vartheta:=\bigcup_{Q'\in\dd_{\mathcal{F},Q}}\mathcal{W}_{Q'}^{\vartheta,*}$. Note that in the previous definition we may allow $\F$ to be empty in which case clearly $\Omega^\vartheta_{\tinyemptyset ,Q}=T_Q^\vartheta$. Similarly, the ``global sawtooth region'' relative to $\mathcal{F}$ is defined as
\begin{equation}
	\label{defomegafq-global}
	\Omega_{\mathcal{F}}^\vartheta:=\interiorior\bigg(\bigcup_{Q'\in\dd_{\mathcal{F}}}U_{Q'}^\vartheta\bigg)=\interiorior\bigg(\bigcup_{I\in\mathcal{W}_{\mathcal{F}}^\vartheta}I^*\bigg),
	\end{equation}
where $\mathcal{W}_{\mathcal{F}}^\vartheta:=\bigcup_{Q'\in\dd_{\mathcal{F}}}\mathcal{W}_{Q'}^{\vartheta,*}$. If $\F$ is the empty set clearly $\Omega_{\tinyemptyset}^\vartheta=\Omega$.
For a given $Q\in\dd$  and $x\in \pom$ let us introduce the ``truncated dyadic cone''
\[
\Gamma_{Q}^\vartheta(x) := \bigcup_{x\in Q'\in\dd_{Q}}  U_{Q'}^\vartheta,
\]
where it is understood that $\Gamma_{Q}^\vartheta(x)=\emptyset$ if $x\notin Q$.
Analogously, we can slightly fatten the Whitney boxes and use $I^{**}$ to define new fattened Whitney regions and sawtooth domains. More precisely, for every $Q\in\dd$,
\[
T_Q^{\vartheta,*}:=\interiorior\bigg(\bigcup_{Q'\in\dd_Q}U_{Q'}^{\vartheta,*}\bigg),\quad\Omega^{\vartheta,*}_{\mathcal{F},Q}:=\interiorior\bigg(\bigcup_{Q'\in\dd_{\F,Q}}U_{Q'}^{\vartheta,*}\bigg), \quad
\Gamma^{\vartheta,*}_{Q}(x) := \bigcup_{x\in Q'\in\dd_{Q_0}}  U_{Q'}^{\vartheta,*},
\]
where $U_{Q}^{\vartheta,*}:=\bigcup_{I\in\mathcal{W}_Q^{\vartheta,*}}I^{**}$.
Similarly, we can define $T_Q^{\vartheta,**}$, $\Omega^{\vartheta,**}_{\mathcal{F},Q}$, $\Gamma_Q^{\vartheta,**}(x)$, and $U^{\vartheta,**}_{Q}$ by using $I^{***}$ in place of $I^{**}$.

To define  the ``Carleson box'' $T_\Delta^\vartheta$ associated with a surface ball $\Delta=\Delta(x,r)$, let $k(\Delta)$ denote the unique $k\in\mathbb{Z}$ such that $2^{-k-1}<200r\leq 2^{-k}$, and set
\begin{equation}\label{D-delta}
	\dd^{\Delta}:=\big\{Q\in\dd_{k(\Delta)}:\:Q\cap 2\Delta\neq\emptyset\big\}.
\end{equation}
We then define
\begin{equation}
	\label{def:T-Delta}
	T_{\Delta}^\vartheta:=\interiorior\bigg(\bigcup_{Q\in\dd^\Delta}\overline{T_Q^\vartheta}\bigg).
\end{equation}
We can also consider fattened versions of $T_\Delta^\vartheta$ given by
$$
T_{\Delta}^{\vartheta,*}:=\interiorior\bigg(\bigcup_{Q\in\dd^\Delta}\overline{T_Q^{\vartheta,*}}\bigg),\qquad T_{\Delta}^{\vartheta,**}:=\interiorior\bigg(\bigcup_{Q\in\dd^\Delta}\overline{T_Q^{\vartheta,**}}\bigg).
$$

Following \cite{HM1, HMT:book}, one can easily see that there exist constants $0<\kappa_1<1$ and $\kappa_0\geq 16\Xi$ (with $\Xi$ the constant in \eqref{deltaQ}), depending only on the allowable parameters and on $\vartheta$, so that
\begin{gather}\label{definicionkappa12}
	\kappa_1B_Q\cap\Omega\subset T_Q^\vartheta\subset T_Q^{\vartheta,*}\subset T_Q^{\vartheta,**}\subset \overline{T_Q^{\vartheta,**}}\subset\kappa_0B_Q\cap\overline{\Omega}=:\tfrac{1}{2}B_Q^*\cap\overline{\Omega},
	\\[6pt]
	\label{definicionkappa0}
	\tfrac{5}{4}B_\Delta\cap\Omega\subset T_\Delta^{\vartheta}\subset T_\Delta^{\vartheta,*}\subset T_\Delta^{\vartheta,**}\subset\overline{T_\Delta^{\vartheta,**}}\subset\kappa_0B_\Delta\cap\overline{\Omega}=:\tfrac{1}{2}B_\Delta^*\cap\overline{\Omega},
\end{gather}
and also
\begin{equation}\label{propQ0}
	Q\subset\kappa_0B_\Delta\cap\partial\Omega=\tfrac{1}{2}B_\Delta^*\cap\partial\Omega=:\tfrac{1}{2}\Delta^*,\qquad\forall\,Q\in\dd^{\Delta},
\end{equation}
where $B_Q$ is defined as in \eqref{deltaQ2}, $\Delta=\Delta(x,r)$ with $x\in\partial\Omega$, $0<r<\diam(\partial \Omega)$, and $B_{\Delta}=B(x,r)$ is so that $\Delta=B_\Delta\cap\partial\Omega$. From our choice of the parameters one also has that $B_Q^*\subset B_{Q'}^*$ whenever $Q\subset Q'$.

\begin{lemma}[{\cite[Lemma~2.54]{AHMT} and \cite[Appendices A.1-A.2]{HM1}}]\label{lemma:CDC-inherit}
	Let $\Omega\subset\mathbb{R}^{n+1}$, $n\ge 2$, be a 1-sided NTA domain satisfying the CDC. For every $\vartheta\ge \vartheta_0$ all of its Carleson boxes $T_Q^{\vartheta}, T_Q^{\vartheta,*}, T_Q^{\vartheta,**}$ and $T_\Delta^{\vartheta}, T_\Delta^{\vartheta,*}, T_\Delta^{\vartheta,**}$, and sawtooth regions $\Omega_\F^{\vartheta},\Omega_\F^{\vartheta,*}, \Omega_\F^{\vartheta,**}$, and $\Omega_{\F,Q}^{\vartheta}, \Omega_{\F,Q}^{\vartheta, *}, \Omega_{\F,Q}^{\vartheta, **}$ are 1-sided NTA domains and satisfy the CDC with uniform implicit constants depending only on dimension, the corresponding	constants for $\Omega$, and $\vartheta$.
\end{lemma}

Given $Q$ we define the ``localized  dyadic conical square function''
\begin{equation}\label{def:SF}
	\mathcal{S}_{Q}^\vartheta u(x):=\bigg(\iint_{\Gamma_{Q}^\vartheta (x)}|\nabla u(Y)|^2\delta(Y)^{1-n}\,dY\bigg)^{\frac12}, \qquad x\in \pom,
\end{equation}
for every $u\in W^{1,2}_{\rm loc}  (T_{Q}^\vartheta)$. Note that $\mathcal{S}_{Q}^\vartheta u(x)=0$ for every $x\in\pom\setminus Q$ since $\Gamma_Q^\vartheta (x)=\emptyset$ in such a case.  The ``localized dyadic non-tangential maximal function'' is given by
\begin{equation}\label{def:NT}
	\mathcal{N}_{Q}^\vartheta u(x)
	: =
	\sup_{Y\in \Gamma^{\vartheta,*}_{Q}(x)} |u(Y)|,
	\qquad x\in \pom,
\end{equation}
for every $u\in \mathscr{C}(T_{Q}^{\vartheta,*})$, where it is understood that $\mathcal{N}_{Q}^\vartheta u(x)= 0$ for every $x\in\pom\setminus Q$.

Given $\alpha>0$ and $x\in\pom$ we introduce the ``cone with vertex at $x$ and aperture $\alpha$'' defined as $\Gamma^\alpha(x) = \{X \in \Omega: |X - x| \leq (1+\alpha) \delta(X)\}$. One can also introduce the ``truncated cone'', for every $x\in\pom$ and $0<r<\infty$ we set  $\Gamma_{r}^\alpha(x) = B(x,r)\cap \Gamma^{\alpha}(x)$.

%Given $\alpha>0$ and $x\in\pom$ we introduce the ``cone with vertex at $x$ and aperture $\alpha$'' defined as $\Gamma^\alpha(x) = \{X \in \Omega: |X - x| \leq (1+\alpha) \delta(X)\}$. One can truncate the cones in two different ways. Given $B=B(y,r)$ and $\Delta=B\cap\pom$ with $y\in \pom$ and $0<r<\diam(\pom)$, and $y\in\pom$ we introduced the ``truncated cone'' $\Gamma_{\Delta}^\alpha(x) = B\cap \Gamma^\alpha(x)$. One can also introduce a different ``truncated cone'', for every $x\in\pom$ and $0<r<\diam(\pom)$ we set  $\Gamma_{r}^\alpha(x) = B(x,r)\cap \Gamma_\alpha(x)$. Note that for every $x\in\pom$ and $0<r<\diam(\pom)$ we have $\Gamma_{\Delta(x,r)}^\alpha(x)=\Gamma_{r}^\alpha(x)$ and in general if $x\in\Delta=\Delta(y_\Delta,r_\Delta)$ with $y_\Delta\in \pom$ and $0<r_\Delta<\diam(\pom)$ then $\Gamma_{r_\Delta}^\alpha(x)\subset \Gamma_{2\,\Delta}^\alpha(x)$ and  $\Gamma_{\Delta}^\alpha(x)\subset \Gamma_{2r_\Delta}^\alpha(x)$, hence both truncations are almost equivalent.

The ``conical square function'' and  the ``non-tangential maximal function'' are defined respectively as
\begin{equation}\label{def:SF-NT:global}
	\mathcal{S}^\alpha u(x):=\bigg(\iint_{\Gamma^\alpha(x)}|\nabla u(Y)|^2\delta(Y)^{1-n}\,dY\bigg)^{\frac12},
	\qquad
	\mathcal{N}^\alpha u (x) := \sup_{X \in \Gamma^\alpha (x)} |u(X)|,\qquad x\in\pom,
\end{equation}
for every $u\in W^{1,2}_{\rm loc}(\Omega)$ and $u\in \mathscr{C}(\Omega)$ respectively. Analogously,  the ``truncated conical square function'' and  the ``truncated non-tangential maximal function'' are defined respectively as
\begin{equation}\label{def:SF-NT:truncated}
	\mathcal{S}_{r}^\alpha u(x):=\bigg(\iint_{\Gamma_{r}^\alpha(x)}|\nabla u(Y)|^2\delta(Y)^{1-n}\,dY\bigg)^{\frac12},
	\quad
	\mathcal{N}^\alpha_{r} u (x) := \sup_{X \in \Gamma^\alpha_{r} (x)} |u(X)|, \quad
	x\in\pom,\ 0<r<\infty,	
\end{equation}
for every $u\in W^{1,2}_{\rm loc}(\Omega\cap B(x,r))$ and  $u\in \mathscr{C}(\Omega\cap B(x,r))$ respectively.

We would like to note that truncated dyadic cones are never empty. Indeed, in our construction we have made sure that $X_Q\in U_Q^\vartheta$ for every $Q\in\dd$, hence for any $Q\in\dd$ and $x\in Q$ one has $X_Q\in\Gamma_Q^\vartheta(x)$. Moreover, $X_{Q'}\in\Gamma_Q^\vartheta(x)$ for every $Q'\in\dd_Q$ with $Q'\ni x$. For the regular truncated cones it could happen that $\Gamma^\alpha_r(x)=\emptyset$ unless $\alpha$ is sufficiently large. Suppose for instance that $\Omega=\{X=(x_1,\dots,x_{n+1})\in \ree: x_{1},\dots, x_{n+1}>0\}$ is the first orthant, then $\Gamma_{r}^\alpha(0)=\emptyset$ for any $0<r<\infty$ if $\alpha<\sqrt{n+1}-1$. On the other hand, if $\alpha$ is sufficiently large, more precisely, if $\alpha\ge c_0^{-1}-1$, where $c_0$ is the corkscrew constant (cf.~Definition~\ref{def1.cork}), then
\begin{equation}\label{cone-CKS}
X_{\Delta(x,r)}\in \Gamma_{r}^\alpha(x),\qquad
\forall\,x\in\pom,\ 0<r<\diam(\pom).
\end{equation}

\section[Elliptic operators, elliptic measure and the Green function]{Uniformly elliptic operators, elliptic measure and the Green function}\label{section:PDE}
Next, we recall several facts concerning elliptic measures and Green functions. To set the stage let $\Omega\subset\re^{n+1}$ be an open set. Throughout we consider
elliptic operators $L$ of the form $Lu=-\div(A\nabla u)$ with $A(X)=(a_{i,j}(X))_{i,j=1}^{n+1}$ being a real (non-necessarily symmetric) matrix such that $a_{i,j}\in L^{\infty}(\Omega)$ and there exists $\Lambda\geq 1$ such that the following uniform ellipticity condition holds
\begin{align}
	\label{e:elliptic}
	\Lambda^{-1} |\xi|^{2} \leq A(X) \xi \cdot \xi,
	\qquad\qquad
	|A(X) \xi \cdot\eta|\leq \Lambda |\xi|\,|\eta|
\end{align}
for all $\xi,\eta \in\mathbb{R}^{n+1}$ and for almost every $X\in\Omega$. We write $L^\top$ to denote the transpose of $L$, or, in other words, $L^\top u = -\div(A^\top
\nabla u)$ with $A^\top$ being the transpose matrix of $A$.

We say that $u$ is a weak solution to $Lu=0$ in $\Omega$ provided that $u\in W_{\rm loc}^{1,2}(\Omega)$ satisfies
\[
\iint_\Omega A(X)\nabla u(X)\cdot \nabla\phi(X) dX=0  \quad\mbox{whenever}\,\, \phi\in \mathscr{C}^{\infty}_{c}(\Omega).
\]
Associated with $L$ one can construct an elliptic measure $\{\omega_L^X\}_{X\in\Omega}$ and a Green function $G_L$ (see \cite{HMT:book} for full details). %Sometimes, in order to emphasize the dependence on $\Omega$, we will write $\omega_{L,\Omega}$ and $G_{L,\Omega}$.
If $\Omega$ satisfies the CDC then it follows that all boundary points are Wiener regular and hence for a given $f\in \mathscr{C}_c(\partial\Omega)$ we can define
\[
u(X):=\int_{\partial\Omega} f(z)d\omega^{X}_{L}(z), \quad \mbox{whenever}\, \, X\in\Omega,
\]
and $u:=f$ on $\pom$ and obtain that $u\in W^{1,2}_{\rm loc}(\Omega)\cap \mathscr{C}(\overline{\Omega})$ and  $Lu=0$ in the weak sense in $\Omega$. Moreover, if $f\in \Lip(\partial \Omega)$ then $u\in W^{1,2}(\Omega)$.

We first define the reverse H\"older class and the $A_\infty$ classes with respect to a fixed elliptic measure in $\Omega$.  One reason we take this approach is that we do not know whether $\mathcal{H}^{n}|_{\partial\Omega}$ is well-defined since we do not assume any Ahlfors regularity in Theorem \ref{thm:main}. Hence we have to develop these notions in terms of elliptic measures.  To this end, let $\Omega$ satisfy the CDC and let $L_0$ and $L$ be two real (non-necessarily symmetric) elliptic operators associated with $L_0u=-\div(A_0\nabla u)$ and $L u=-\div(A\nabla u)$ where $A$ and $A_0$ satisfy \eqref{e:elliptic}. Let $\omega^{X}_{L_0}$ and $\omega_{L}^{X}$ be the elliptic measures of $\Omega$ associated with the operators $L_0$ and $L$ respectively with pole at $X\in\Omega$. Note that if we further assume that $\Omega$ is connected, then Harnack's inequality readily implies that $\omega_{L}^{X}\ll\omega_L^{Y}$ on $\pom$ for every $X,Y\in\Omega$. Hence if $\omega_L^{X_0}\ll\omega_{L_0}^{Y_0}$ on $\pom$  for some $X_0,Y_0\in\Omega$ then $\omega_L^{X}\ll\omega_{L_0}^{Y}$ on $\pom$ for every $X,Y\in\Omega$ and thus  we can simply write $\omega_{L}\ll \omega_{L_0}$ on $\pom$. In the latter case we will use the notation
\begin{equation}\label{def-RN}
	h(\cdot\,;L, L_0, X)=\frac{d\omega_L^{X}}{d\omega_{L_0}^{X}}
\end{equation}
to denote the Radon-Nikodym derivative of $\omega_{L}^{X}$ with respect to $\omega_{L_0}^{X}$,
which is a well-defined function $\omega_{L_0}^{X}$-almost everywhere on $\pom$.

\begin{definition}[Reverse H\"older and $A_\infty$ classes]\label{def:Ainfty}
	Fix $\Delta_0=B_0\cap \pom$ where $B_0=B(x_0,r_0)$ with $x_0\in\pom$ and $0<r_0<\diam(\pom)$. Given $1<p<\infty$, we say that $\omega_L\in RH_p(\Delta_0,\omega_{L_0})$, provided that $\omega_L\ll \omega_{L_0}$ on $\Delta_0$, and there exists $C\geq 1$ such that
	\begin{align}\label{eqn:def:RHp}
		\left(\fint_{\Delta}h(y;L,L_0,X_{\Delta_0} )^p d \omega_{L_0}^{X_{\Delta_0}}(y)\right)^{\frac1p}
		\leq
		C
		\fint_{\Delta} h(y;L,L_0,X_{\Delta_0} ) d \omega_{L_0}^{X_{\Delta_0}}(y)
		=
		C\frac{\omega_L^{X_{\Delta_0}}(\Delta)}{\omega_{L_0}^{X_{\Delta_0}}(\Delta)},
	\end{align}
	for every $\Delta=B\cap \partial\Omega$ where $B\subset B(x_0,r_0)$, $B=B(x,r)$ with  $x\in \partial\Omega$, $0<r<\diam(\partial\Omega)$. The infimum of the constants $C$ as above is denoted by $[\omega_{L}]_{RH_p(\Delta_0,\omega_{L_0})}$.
	
	Similarly, we say that $\omega_L\in RH_p(\pom,\omega_{L_0})$ provided that for every $\Delta_0=\Delta(x_0,r_0)$ with $x_0\in\pom$ and $0<r_0<\diam(\pom)$ one has $\omega_L\in RH_p(\Delta_0,\omega_{L_0})$ uniformly on $\Delta_0$, that is,
	\[
	[\omega_{L}]_{RH_p(\pom,\omega_{L_0})}
	:=\sup_{\Delta_0} [\omega_{L}]_{RH_p(\Delta_0,\omega_{L_0})}<\infty.
	\]

	Finally,
	\[
	A_\infty(\Delta_0,\omega_{L_0}):=\bigcup_{p>1} RH_p(\Delta_0,\omega_{L_0})
	\quad\mbox{and}\quad
	A_\infty(\partial\Omega,\omega_{L_0}):=\bigcup_{p>1} RH_p(\partial\Omega,\omega_{L_0})
	.\]
\end{definition}

\begin{definition}[$\mathrm{BMO}$]
	Fix $\Delta_0=B_0\cap \pom$ where $B_0=B(x_0,r_0)$ with $x_0\in\pom$ and $0<r_0<\diam(\pom)$. We say that $f\in \mathrm{BMO}(\Delta_0, \omega_L)$ provided $f \in L^1_{\loc}(\Delta_0, \omega_L^{X_{\Delta_0}})$ and
	\begin{equation*}
		\|f\|_{\mathrm{BMO}(\Delta_0, \omega_L)} := \sup_{\Delta} \inf_{c \in \R} \fint_{\Delta} |f(x)-c| \, d \omega_L^{X_{\Delta_0}}(x) < \infty,
	\end{equation*}
	where the sup is taken over all surface balls $\Delta=B\cap \partial\Omega$ where $B\subset B(x_0,r_0)$, $B=B(x,r)$ with  $x\in \partial\Omega$, $0<r<\diam(\partial\Omega)$.
	
	Similarly, we say that $f\in \mathrm{BMO}(\partial \Omega, \omega_L)$ provided that for every $\Delta_0=\Delta(x_0,r_0)$ with $x_0\in\pom$ and $0<r_0<\diam(\pom)$ one has $f\in \mathrm{BMO}(\Delta_0, \omega_L)$ uniformly on $\Delta_0$, that is,
	$f \in L^1_{\loc}(\partial \Omega, \omega_L)$ (that is,  $\|f\,\mathbf{1}_\Delta\|_{L^1(\pom,\omega_L^X)}<\infty$ for every surface ball $\Delta\subset\pom$ and for every $X\in\Omega$ ---albeit with a constant that may depend on $\Delta$ and $X$) and satisfies
	\begin{equation*}
		\|f\|_{\mathrm{BMO}(\partial \Omega, \omega_L)} = \sup_{\Delta_0} \sup_{\Delta} \inf_{c \in \R} \fint_{\Delta} |f(x)-c| \, d \omega_L^{X_{\Delta_0}}(x) < \infty,
	\end{equation*}
	where the sups are taken respectively over all surface balls $\Delta_0=B(x_0,r_0)\cap\pom$ with $x_0\in \pom$ and $0<r_0<\diam(\pom)$,
	and $\Delta=B\cap\pom$, $B=B(x,r)\subset B_0$ with $x\in \pom$ and $0<r<\diam(\pom)$.
	\end{definition}

\begin{definition}[Solvability, CME, $\mathcal{S}<\mathcal{N}$] \label{defs:several}
Let $\Omega\subset\mathbb{R}^{n+1}$, $n\ge 2$, be a 1-sided NTA domain (cf. Definition \ref{def1.1nta}) satisfying the capacity density condition  (cf. Definition \ref{def-CDC}), and let $Lu=-\div(A\nabla u)$ and $L_0u=-\div(A_0\nabla u)$ be real (non-necessarily symmetric) elliptic operators.

\begin{list}{$\bullet$}{\leftmargin=0.4cm  \itemsep=0.2cm}

\item Given $1<p<\infty$, we say that \textit{$L$ is  $L^{p}(\omega_{L_0})$-solvable} if for a given $\alpha>0$ and  $N\ge 1$ there exists $C_{\alpha, N}\ge 1$ (depending only on $n$, the 1-sided NTA constants, the CDC constant, the ellipticity of $L_0$ and $L$, $\alpha$, $N$, and $p$) such that for every $\Delta_0 = \Delta(x_0, r_0)$ with $x_0 \in \partial \Omega, 0 < r_0 < \diam (\partial \Omega)$, and for every $f \in \mathscr{C}(\partial \Omega)$ with $\supp f \subset N\Delta_0$ if one sets
\begin{equation}\label{u-elliptic:L-sol}
	u(X):=\int_{\pom} f(y)\,d\omega_{L}^X(y),\qquad X\in\Omega,
\end{equation}
then
\begin{equation}\label{solv-Lp}
	\|\mathcal{N}^\alpha_{r_0} u\|_{L^p(\Delta_0, \omega_{L_0}^{X_{\Delta_0}})} \leq C_{\alpha,N} \|f\|_{L^p(N\Delta_0, \omega_{L_0}^{X_{\Delta_0}})}.
\end{equation}

\item We say that \textit{$L$ is $\mathrm{BMO}(\omega_{L_0})$-solvable}, if there exists $C\ge 1$ (depending only on $n$, the 1-sided NTA constants, the CDC constant, and the ellipticity of $L_0$ and $L$) such that for every $f \in \mathscr{C}(\partial \Omega)\cap L^\infty(\pom, \omega_{L_0})$ if one takes  $u$ as in \eqref{u-elliptic:L-sol} and we set $u_{L,\Omega}(X):=\omega_{L}^X(\pom)$, $X\in\Omega$, then
\begin{equation}\label{CME:BMO}
	\sup_B \sup_{B'} \frac{1}{\omega_{L_0}^{X_{\Delta}} (\Delta')} \iint_{B'\cap \Omega} |\nabla (u-f_{\Delta, L_0}u_{L,\Omega})(X)|^2 G_{L_0}(X_\Delta,X) \, d X
	\le C \|f\|^2_{\mathrm{BMO}(\partial \Omega, \omega_{L_0})},
\end{equation}
where $\Delta=B\cap\pom$, $\Delta'=B'\cap\pom$,  $f_{\Delta, L_0}=	\displaystyle\fint_{\Delta} f\,d\omega_{L_0}^{X_\Delta}$, 	
and the sups are taken respectively over all balls $B=B(x,r)$ with $x\in \pom$ and $0<r<\diam(\pom)$,
and $B'=B(x',r')$ with $x'\in2\Delta$ and $0<r'<r c_0/4$, and $c_0$ is the Corkscrew constant.

\item We say that \textit{$L$ is $\mathrm{BMO}(\omega_{L_0})$-solvable in the sense of \cite{HLe}}, that is, there exists $C\ge 1$ (depending only on $n$, the 1-sided NTA constants, the CDC constant, and the ellipticity of $L_0$ and $L$) such that for every $\varepsilon\in (0,1]$ there exists $\varrho(\varepsilon)\ge 0$ such that $\varrho(\varepsilon)\longrightarrow 0$ as $\varepsilon\to 0^+$ in such a way that for every $f \in \mathscr{C}(\partial \Omega)\cap L^\infty(\pom, \omega_{L_0})$ if one takes  $u$ as in \eqref{u-elliptic:L-sol}, then
\begin{equation}\label{CME:BMO-Linfty}
	\sup_{B_\varepsilon} \sup_{B'} \frac{1}{\omega_{L_0}^{X_{\Delta_\varepsilon}} (\Delta')} \iint_{B'\cap \Omega} |\nabla u(X)|^2 G_{L_0}(X_{\Delta_\varepsilon},X) \, d X
	\le C \big(\|f\|^2_{\mathrm{BMO}(\partial \Omega, \omega_{L_0})} + \varrho(\varepsilon)\|f\|_{L^\infty(\pom, \omega_{L_0})}^2\big),
\end{equation}
where $\Delta_\varepsilon=B_\varepsilon\cap\pom$, $\Delta'=B'\cap\pom$,  and the sups are taken respectively over all balls $B_\varepsilon=B(x,\varepsilon r)$ with $x\in \pom$ and $0<r<\diam(\pom)$,
and $B'=B(x',r')$ with $x'\in2\Delta_\varepsilon$ and $0<r'<\varepsilon r c_0/4$, and $c_0$ is the Corkscrew constant.

\item We say that \textit{$L$ satisfies $\mathrm{CME}(\omega_{L_0})$}, if there exists $C\ge 1$ (depending only on $n$, the 1-sided NTA constants, the CDC constant, and the ellipticity of $L_0$ and $L$) such that for every $u\in W^{1,2}_{\loc}(\Omega)\cap L^\infty(\Omega)$ satisfying $Lu=0$ in the weak sense in $\Omega$ the following estimate holds
\begin{equation}\label{CME:Linfty}
	\sup_B \sup_{B'} \frac{1}{\omega_{L_0}^{X_{\Delta}} (\Delta')} \iint_{B'\cap \Omega} |\nabla u(X)|^2 G_{L_0}(X_\Delta,X) \, d X
	\leq C \|u\|_{L^\infty(\Omega)}^2,
\end{equation}
where $\Delta=B\cap\pom$, $\Delta'=B'\cap\pom$,  and the sups are taken respectively over all balls $B=B(x,r)$ with $x\in \pom$ and $0<r<\diam(\pom)$,
and $B'=B(x',r')$ with $x'\in2\Delta$ and $0<r'<r c_0/4$, and $c_0$ is the Corkscrew constant.

\item  Given $0<q<\infty$, we say that \textit{$L$ satisfies $\mathcal{S}<\mathcal{N}$ in $L^q(\omega_{L_0})$} if for some given $\alpha>0$,
there exists $C_{\alpha}\ge 1$ (depending only on $n$, the 1-sided NTA constants, the CDC constant, the ellipticity of $L_0$ and $L$, $\alpha$, and $q$) such that for every $\Delta_0 = \Delta(x_0, r_0)$ with $x_0 \in \partial \Omega, 0 < r_0 < \diam (\partial \Omega)$, and for every
$u\in W^{1,2}_{\loc}(\Omega)$ satisfying $Lu=0$ in the weak sense in $\Omega$ the following estimate holds
\begin{equation}\label{S<N}
	\|\mathcal{S}^\alpha_{r_0} u\|_{L^q(\Delta_0, \omega_{L_0}^{X_{\Delta_0}})}
	\leq
	C_\alpha \|\mathcal{N}^\alpha_{r_0} u\|_{L^q(5\Delta_0, \omega_{L_0}^{X_{\Delta_0}})}.
\end{equation}	

\item We say that any of the previous properties holds \textit{for characteristic functions} if the corresponding estimate is valid for all
solutions of the form $u(X)=\omega_L^X(S)$, $X\in\Omega$, with $S\subset\pom$ being an arbitrary Borel set (with $S\subset N\Delta_0$ in the case of $L^{p}(\omega_{L_0})$-solvability).

\end{list}

\end{definition}

\begin{remark}
	We would like to observe that when either $\Omega$ and $\pom$ are both bounded or  $\pom$ is unbounded, the elliptic measure is a probability (that is, $u_{L,\Omega}(X)=\omega_L^X(\pom)\equiv 1$ for every $X\in\Omega$). Hence, it has vanishing gradient and  one can then  remove the term $f_{\Delta, L_0}u_{L,\Omega}$ in \eqref{CME:BMO}. This means that the only case on which subtracting $f_{\Delta, L_0}u_{L,\Omega}$ is relevant is that where $\Omega$ is unbounded and $\pom$ is bounded (e.g., the complementary of a ball).  As a matter of fact, one must subtract that term or a similar one  for \eqref{CME:BMO} to hold. To see this,  take $f\equiv 1\in \mathrm{BMO}(\partial \Omega, \omega_{L_0})$ so that
	$\|f\|_{\mathrm{BMO}(\partial \Omega, \omega_{L_0})}=0$ and let $u=u_{L,\Omega}$ be the associated elliptic measure solution. One can see (cf.~\cite{HMT:book}) that the function $u_{L,\Omega}$ is non-constant (it decays at infinity), hence $0<u_{L,\Omega}(X)< 1$ for every $X\in\Omega$ and $|\nabla u_{L,\Omega}|\not\equiv 0$. This means that the version of \eqref{CME:BMO} without the term $f_{\Delta, L_0}u_{L,\Omega}$ cannot hold. Moreover, note that in this case \eqref{CME:BMO}  is trivial: $f_{\Delta, L_0}u_{L,\Omega}=u_{L,\Omega}$ and the left-hand side of  \eqref{CME:BMO} vanishes.
\end{remark}

\begin{remark}
	As just explained in the previous remark when either $\Omega$ and $\pom$ are both bounded or  $\pom$ is unbounded, the left-hand sides of \eqref{CME:BMO} and \eqref{CME:BMO-Linfty} are the same, as a result (e) clearly implies (f) ---and (e)$'$ implies (f)$'$--- upon taking $\varrho(\varepsilon)\equiv 0$ (we will see in the course of the proof that these two implications always hold). Much as before, when $\Omega$ is unbounded and $\pom$ is bounded, \eqref{CME:BMO-Linfty} needs to incorporate the term $\varrho(\varepsilon)\|f\|_{L^\infty(\pom, \omega_{L_0})}^2$, otherwise it would fail for $u=u_{L,\Omega}$.
\end{remark}

\begin{remark}\label{remark:BMO-other-cons}
	In  \eqref{CME:BMO} one can replace $f_{\Delta, L_0}$ by  $f_{\Delta', L_0}$ (see  Remark~\ref{remark:BMO-other-cons:proof} below). Also, when $\Omega$ is unbounded and $\pom$ bounded one can subtract  a constant that does not depend on $\Delta$ nor $\Delta'$. Namely, let $X_\Omega\in\Omega$ satisfy $\delta(X_\Omega)\approx \diam(\pom)$ (say, $X_\Omega=X_{\Delta(x_0, r_0)}$ with $x_0\in \pom$ and $r_0\approx\diam(\pom)$). Then in \eqref{CME:BMO} one can replace $f_{\Delta, L_0}$  by $f_{\pom, L_0}=\displaystyle\fint_{\pom} f\,d\omega_{L_0}^{X_\Omega}$, see Remark~\ref{remark:BMO-other-cons:proof}.
\end{remark}

The following lemmas state some properties of Green functions and elliptic measures, proofs may be found in \cite{HMT:book}.

\begin{lemma}\label{lemma:Greensf}
	Suppose that $\Omega\subset\re^{n+1}$, $n\ge 2$, is an open set satisfying the CDC. Given a real (non-necessarily symmetric) elliptic operator $L=-\div(A\nabla)$, there exist $C>1$ (depending only on dimension and on the ellipticity constant of $L$) and $c_\theta>0$ (depending on the above parameters and on $\theta\in (0,1)$) such that $G_L$,  the Green function associated with $L$, satisfies
	\begin{gather}\label{sizestimate}
		G_L(X,Y)\leq C|X-Y|^{1-n};
		\\[0.15cm]
		c_\theta|X-Y|^{1-n}\leq G_L(X,Y),\quad\text{if }\,|X-Y|\leq\theta\delta(X),\quad\theta\in(0,1);
		\\[0.15cm]
		G_L(\cdot,Y)\in \mathscr{C}\big(\overline{\Omega}\setminus\{Y\}\big)\quad\text{and}\quad G_L(\cdot,Y)|_{\partial\Omega}\equiv 0\quad\forall Y\in\Omega;
		\\[0.15cm]
		G_L(X,Y)\geq 0,\quad\forall X,Y\in\Omega,\quad X\neq Y;
		\\[0.15cm]\label{G-G-top}
		G_L(X,Y)=G_{L^\top}(Y,X),\quad\forall X,Y\in\Omega,\quad X\neq Y.
	\end{gather}
	
	Moreover, $G_L(\cdot,Y)\in W_{\rm loc}^{1,2}(\Omega\setminus\{Y\})$ for any $Y\in\Omega$ and satisfies $L G_L(\cdot,Y)=\delta_Y$ in the sense of distributions, that is,
	\begin{equation}\label{eq:G-delta}
		\iint_{\Omega}A(X)\nabla_X G_L(X,Y)\cdot\nabla\varphi(X)\,dX=\varphi(Y),\qquad\forall\, \varphi\in  \mathscr{C}_c^{\infty}(\Omega).
	\end{equation}
	In particular, $G_L(\cdot,Y)$ is a weak solution to $L G_L(\cdot,Y)=0$ in the open set $\Omega\setminus\{Y\}$.
	
	Finally, the following Riesz formula holds:
	\[
	\iint_{\Omega}A^\top(X)\nabla_XG_{L^\top}(X,Y)\cdot\nabla\varphi(X)\,dX
	=
	\varphi(Y)-\int_{\partial\Omega}\varphi\,d\omega_L^Y
	,\quad\text{for a.e. }Y\in\Omega,
	\]
	for every $\varphi\in \mathscr{C}_c^{\infty}(\re^{n+1})$.
\end{lemma}
\medskip

\begin{remark}\label{rem:GF}
	If we also assume that $\Omega$ is bounded, following \cite{HMT:book} we know that the Green function $G_L$ coincides with the one constructed in \cite{gruterwidman}. Consequently, for each $Y\in\Omega$ and $0<r<\delta(Y)$, there holds
	\begin{equation}\label{obsgreen1}
		G_L(\cdot,Y)\in W^{1,2}(\Omega\setminus B(Y,r))\cap W_0^{1,1}(\Omega).
	\end{equation}
	Moreover, for every $\varphi\in \mathscr{C}_c^{\infty}(\Omega)$ such that $0\le \varphi\le1 $ and $\varphi\equiv 1$ in $B(Y,r)$ with $0<r<\delta(Y)$, we have that
	\begin{equation}\label{obsgreen2}
		(1-\varphi)G_L(\cdot,Y)\in W_0^{1,2}(\Omega).
	\end{equation}
\end{remark}
\medskip

The following result lists a number of properties which will be used throughout the paper:

\begin{lemma}\label{lemma:proppde}
	Suppose that $\Omega\subset\re^{n+1}$, $n\ge 2$, is a 1-sided NTA domain satisfying the CDC. Let $L_0=-\div(A_0\nabla)$ and $L=-\div(A\nabla)$ be two real (non-necessarily symmetric) elliptic operators, there exist $C_1\ge 1$, $\rho\in (0,1)$ (depending only on dimension, the 1-sided NTA constants, the CDC constant, and the ellipticity of $L$) and $C_2\ge 1$ (depending on the same parameters and on the ellipticity of $L_0$), such that for every $B_0=B(x_0,r_0)$ with $x_0\in\partial\Omega$, $0<r_0<\diam(\partial\Omega)$, and $\Delta_0=B_0\cap\partial\Omega$ we have the following properties:
	\begin{list}{$(\theenumi)$}{\usecounter{enumi}\leftmargin=1cm \labelwidth=1cm \itemsep=0.1cm \topsep=.2cm \renewcommand{\theenumi}{\alph{enumi}}}
		
		\item $\omega_L^Y(\Delta_0)\geq C_1^{-1}$ for every $Y\in C_1^{-1}B_0\cap\Omega$ and $\omega_L^{X_{\Delta_0}}(\Delta_0)\ge C_1^{-1}$.

		\item If $B=B(x,r)$ with $x\in\partial\Omega$ and $\Delta=B\cap\partial\Omega$ is such that $2B\subset B_0$, then for all $X\in\Omega\setminus B_0$ we have that
		\[
		\frac{1}{C_1}\omega_L^X(\Delta)\leq r^{n-1} G_L(X,X_\Delta)\leq C_1\omega_L^X(\Delta).
		\]

		\item  If $X\in\Omega\setminus 4B_0$, then
		$$
		\omega_{L}^X(2\Delta_0)\leq C_1\omega_{L}^X(\Delta_0).
		$$

		\item  If $B=B(x,r)$ with $x\in\partial\Omega$ and $\Delta:=B\cap\partial\Omega$ is such that $B\subset B_0$, then for every $X\in\Omega\setminus 2\kappa_0B_0$ with $\kappa_0$ as in \eqref{definicionkappa0}, we have that
		$$
		\frac{1}{C_1}\omega_L^{X_{\Delta_0}}(\Delta)\leq \frac{\omega_L^X(\Delta)}{\omega_L^X(\Delta_0)}\leq C_1\omega_L^{X_{\Delta_0}}(\Delta).
		$$
		As a consequence,
		\[
		\frac1C_1 \frac1{\omega_L^X(\Delta_0)}\le  \frac{d\omega_L^{{X_{\Delta_0}}}}{d\omega_L^X}(y)\le C_1 \frac1{\omega_L^X(\Delta_0)},
		\qquad\mbox{for $\omega_L^X$-a.e. $y\in\Delta_0$}.
		\]

		\item For every $X\in B_0\cap\Omega$ and for any $j\ge 1$
		\[
		\frac{d\omega_L^X}{d\omega_L^{X_{2^j\Delta_0}}}(y)\le C_1\,\bigg(\frac{\delta(X)}{2^j\,r_0}\bigg)^{\rho},
		\qquad\mbox{for $\omega_L^X$-a.e. $y\in\pom\setminus 2^j\,\Delta_0$}.
		\]
		
		\item If $0\le u\in W^{1,2}_\loc(B_0\cap \Omega)\cap\mathscr{C}(\overline{B_0\cap \Omega})$ satisfies $Lu=0$ in the weak-sense in $B_0\cap \Omega$ and $u\equiv 0$ in $\Delta_0$ then
\[
u(X)
\le
C_1\,
\Big(\frac{\delta(X)}{r_0}\Big)^\rho u(X_{\Delta_0}), \qquad X\in\tfrac12 B_0\cap\Omega.
\]
	\end{list}

\end{lemma}

\medskip

\begin{remark}\label{remark:chop-dyadic}
	We note that from $(d)$ in the previous result and Harnack's inequality one can easily see that given $Q, Q', Q''\in\dd(\pom)$
	\begin{equation}\label{chop-dyadic}
		\frac{\omega_L^{X_{Q''}}(Q)}{\omega_L^{X_{Q''}}(Q')}\approx \omega_L^{X_{Q'}}(Q),
		\qquad\mbox{whenever }
		Q\subset Q'\subset Q''.
	\end{equation}
	Also, $(d)$, Harnack's inequality, and \eqref{deltaQ} give
	\begin{equation}\label{chop-dyadic:densities}
		\frac{d\omega_L^{X_{Q'}}}{d\omega_L^{X_{Q''}}}(y)
		\approx
		\frac1{\omega_L^{X_{Q''}}(Q')},
		\qquad
		\mbox{ for $\omega_L^{X_{Q''}}$-a.e. }y\in Q',
		\mbox{whenever }Q'\subset Q''.
	\end{equation}		
	Observe that since $\omega_L^{X_{Q''}}\ll \omega_L^{X_{Q'}}$  we can easily get an analogous inequality for the reciprocal of the Radon-Nikodym derivative.
\end{remark}

\begin{remark}\label{RHp-scale-invariant}
It is not hard to see that if $\omega_L\ll\omega_{L_0}$ then Lemma~\ref{lemma:proppde} gives the following:
\begin{equation}
\omega_L\in RH_p(\pom,\omega_{L_0})\ \Longleftrightarrow \	
\sup_{x \in \partial \Omega, 0 < r < \diam(\pom)} \|h(\cdot\,; L, L_0, X_{\Delta(x,r)})\|_{L^{p}(\Delta(x,r), \omega_{L_0}^{X_{\Delta(x,r)}})} <\infty.
\end{equation}
The left-to-right implication follows at once from \eqref{eqn:def:RHp} by taking $B=B_0$ (hence $\Delta=\Delta_0$) and Lemma~\ref{lemma:proppde} part $(a)$. For the converse, fix  $B_0=B(x_0,r_0)$ and $B=B(x,r)$ with $B\subset B_0$, $x_0,x \in \partial \Omega$ and $0 < r_0,r < \diam(\pom)$. Write $\Delta_0=B_0\cap\pom$ and $\Delta=B\cap\pom$. If $r\approx r_0$ we have, by Lemma~\ref{lemma:proppde} part $(a)$,
\begin{multline*}
\left(\fint_{\Delta}h(y;L,L_0,X_{\Delta_0} )^p d \omega_{L_0}^{X_{\Delta_0}}(y)\right)^{\frac1p}
\lesssim
\|h(\cdot\,; L, L_0, X_{\Delta_0})\|_{L^{p}(\Delta_0, \omega_{L_0}^{X_{\Delta_0}})}
\\
\approx
\|h(\cdot\,; L, L_0, X_{\Delta_0})\|_{L^{p}(\Delta_0, \omega_{L_0}^{X_{\Delta_0}})}
\frac{\omega_L^{X_{\Delta_0}}(\Delta)}{\omega_{L_0}^{X_{\Delta_0}}(\Delta)}.
\end{multline*}
On the other hand, if $r\ll r_0$, we have by Lemma~\ref{lemma:proppde} part $(d)$ and the fact that $\omega_L\ll\omega_{L_0}$ that
\[
h(\cdot\,;L,L_0,X_{\Delta_0})
=
\frac{d\omega_L^{X_{\Delta_0}}}{d\omega_{L_0}^{X_{\Delta_0}}}
=
\frac{d\omega_L^{X_{\Delta_0}}}{d\omega_{L}^{X_{\Delta}}}
\frac{d\omega_L^{X_{\Delta}}}{d\omega_{L_0}^{X_{\Delta}}}
\frac{d\omega_{L_0}^{X_{\Delta}}}{d\omega_{L_0}^{X_{\Delta_0}}}
\approx
h(\cdot\,;L,L_0,X_{\Delta})\frac{\omega_{L}^{X_{\Delta_0}}(\Delta)}{\omega_{L_0}^{X_{\Delta_0}}(\Delta)},
\quad\text{$\omega_{L_0}$-a.e. in $\Delta$.}
\]
This and Lemma~\ref{lemma:proppde} part $(d)$ give
\begin{align*}
\left(\fint_{\Delta}h(y;L,L_0,X_{\Delta_0} )^p d \omega_{L_0}^{X_{\Delta_0}}(y)\right)^{\frac1p}
\approx
\|h(\cdot\,; L, L_0, X_{\Delta_0})\|_{L^{p}(\Delta, \omega_{L_0}^{X_{\Delta}})}	
\approx
\|h(\cdot\,; L, L_0, X_{\Delta})\|_{L^{p}(\Delta, \omega_{L_0}^{X_{\Delta}})}\frac{\omega_{L}^{X_{\Delta_0}}(\Delta)}{\omega_{L_0}^{X_{\Delta_0}}(\Delta)}	
\end{align*}
Thus, \eqref{eqn:def:RHp} holds and the right-to-left implication holds.
\end{remark}

\begin{remark}
It is not difficult to see that under the assumptions of Lemma~\ref{lemma:proppde} one has
	\begin{equation*}
		\|f\|_{\mathrm{BMO}(\partial \Omega, \omega_L)}
		\approx
		\sup_{\Delta \subset \partial \Omega} \inf_{c \in \R} \fint_\Delta |f(x)-c| \, d \omega_L^{X_\Delta}(x),
	\end{equation*}
	where the sup  is taken over all surface balls $\Delta=B(x,r)\cap\pom$ with $x\in \pom$ and $0<r<\diam(\pom)$. Thus, we could have taken this as the definition of $f\in\mathrm{BMO}(\partial \Omega, \omega_L)$.
\end{remark}

\begin{remark}\label{rem:JN}
Under the assumptions of Lemma~\ref{lemma:proppde}, for every $\Delta_0$ as above if $f\in \mathrm{BMO}(\Delta_0, \omega_L)$, then   John-Nirenberg's inequality holds locally in $\Delta_0$ and the implicit constants depend on the doubling property of $\omega_L^{X_{\Delta_0}}$ in $2\Delta_0$. Thus, if one further assumes that $f\in \mathrm{BMO}(\partial \Omega, \omega_L)$, then for every $1<q<\infty$ there holds
	\begin{equation}\label{eqn:JN}
		\|f\|_{\mathrm{BMO}(\partial \Omega, \omega_L)} \approx \sup_{\Delta_0} \sup_{\Delta} \inf_{c \in \R} \Big(\fint_{\Delta} |f(x)-c|^q \, d \omega_L^{X_{\Delta_0}}(x)\Big)^{\frac1q} < \infty,
	\end{equation}
		where the sups are taken respectively over all surface balls $\Delta_0=B(x_0,r_0)\cap\pom$ with $x_0\in \pom$ and $0<r_0<\diam(\pom)$,
	and $\Delta=B\cap\pom$, $B=B(x,r)\subset B_0$ with $x\in \pom$ and $0<r<\diam(\pom)$. Note that the implicit constants depend only on dimension, the 1-sided NTA constants, the CDC constant, the ellipticity of $L$, and $q$.
\end{remark}

\section{Proof of Theorem~\ref{thm:main}}\label{section:proof-main}
We first observe that for any $p\in(1,\infty)$ the equivalence $\textup{(a)}{}_{p'}\Longleftrightarrow\textup{(b)}{}_p$ easily implies $\textup{(a)}\Longleftrightarrow\textup{(b)}$. Also, since Jensen's inequality readily gives that $	\omega_L \in RH_{p'}(\partial \Omega, \omega_{L_0})$ implies $\omega_L \in RH_{q'}(\partial \Omega, \omega_{L_0})$ for all $q\ge p$, the equivalence $\textup{(a)}{}_{p'}\Longleftrightarrow\textup{(b)}{}_p$ yields $\textup{(b)}_p\Longrightarrow\textup{(b)}_q$ for all $q\ge p$. Finally,  $\textup{(b)}_p\Longrightarrow\textup{(b)}_p'$ clearly implies $\textup{(b)}\Longrightarrow\textup{(b)$'$}$. With all these in mind, we will follow the scheme
	\[
	\textup{(a)}_{p'}\Longleftrightarrow\textup{(b)}_p\Longrightarrow \textup{(b)}_p',
	\qquad
	\textup{(b)}'\Longrightarrow \textup{(a)},
	\qquad
	\textup{(a)}\Longrightarrow\textup{(d)}\Longrightarrow \textup{(d)}'\Longrightarrow \textup{(a)},
	\]
	\[
	\textup{(c)}\Longrightarrow\textup{(c)}', \qquad \textup{(e)}\Longrightarrow \textup{(f)}\Longrightarrow \textup{(c)}', \qquad
	\textup{(e)}'\Longrightarrow \textup{(f)}'\Longrightarrow \textup{(c)}'\Longrightarrow \textup{(a)},
	\]
	\[
	\textup{(a)}\Longrightarrow\textup{(c)},\qquad\textup{(a)}\Longrightarrow\textup{(e)}, \qquad \textup{(a)}\Longrightarrow\textup{(e)}'.
	\]
	
Before proving all these implications we present some auxiliary results:

\begin{lemma}\label{lemma:CME-uOmega}
Let $\Omega\subset\mathbb{R}^{n+1}$, $n\ge 2$, be a 1-sided NTA domain (cf. Definition \ref{def1.1nta}) satisfying the capacity density condition  (cf. Definition \ref{def-CDC}),  and let $Lu=-\div(A\nabla u)$ and $L_0u=-\div(A_0\nabla u)$ be real (non-necessarily symmetric) elliptic operators. There exists $\rho\in (0,1)$ (depending only on dimension, the 1-sided NTA constants, the CDC constant, and the ellipticity of $L$) and $C_1\ge 1$ (depending on the same parameters and on the ellipticity of $L_0$) such that the following holds: If $\Delta=B\cap\pom$ and $\Delta'=B'\cap\pom$,   where $B=B(x,r)$ with $x\in \pom$ and $0<r<\diam(\pom)$,
and $B'=B(x',r')$ with $x'\in2\Delta$ and $0<r'<r c_0/4$, where $c_0$ is the Corkscrew constant, and $u_{L,\Omega}(X):=\omega_L^X(\pom)$, $X\in\Omega$, then
	\begin{equation}\label{CME-uOmega}
	\frac{1}{\omega_{L_0}^{X_{\Delta}} (\Delta')} \iint_{B'\cap \Omega} |\nabla u_{L,\Omega}(X)|^2 G_{L_0}(X_\Delta,X) \, d X
	\leq C_1 \Big(\frac{r'}{\diam(\pom)}\Big)^{2\rho}.
\end{equation}
\end{lemma}

\begin{proof}
Fix $B=B(x,r)$ with $x\in \pom$ and $0<r<\diam(\pom)$ and $B'=B(x',r')$ with $x'\in2\Delta$ and $0<r'<r c_0/4$. Let $\Delta=B\cap\pom$, $\Delta'=B'\cap\pom$.	
	
We note that when either $\pom$ is unbounded or $\pom$ and $\Omega$ are both bounded then elliptic measure is a probability, hence $u_{L,\Omega}\equiv 1$ and the desired estimate is trivial. This means that we may assume that $\pom$ is bounded and $\Omega$ is unbounded (e.g, the complement of a closed ball). In that scenario, $u_{L,\Omega}$ decays at $\infty$, $0<u_{L, \Omega}<1$ in $\Omega$, and $u_{L,\Omega}|_{\pom}\equiv 1$. Define $v:=1-u_{L,\Omega}$ and note that our assumptions guarantee that $v\in W^{1,2}_\loc(\Omega)\cap\mathscr{C}(\overline{\Omega})$ with $0\le v\le 1$ and $v|_{\pom}\equiv 0$. By Lemma~\ref{lemma:proppde} part $(f)$ applied in $B(x',\diam(\pom)/2)$ we have
\[
0\le
v(X)\lesssim \Big(\frac{\delta(X)}{\diam(\pom)}\Big)^\rho v(X_{\Delta(x',\diam(\pom)/2)})
\le
\Big(\frac{\delta(X)}{\diam(\pom)}\Big)^\rho, \qquad X\in B'\cap\Omega.
\]
Set $\W_{B'}:=\{I\in\W: I\cap B'\neq\emptyset\}$ and we pick $Z_{I, B'}\in I\cap B'$ for $I \in \W_{B'}$. Caccioppoli's and Harnack's inequalities, and the previous estimate yield
\begin{align*}
\iint_I |\nabla v(X)|^2 dX
\lesssim
\ell(I)^{-2} \iint_{I^*} v(X)^2 dX
\lesssim
\ell(I)^{n-1} v(Z_{I, B'})^2
\lesssim
\ell(I)^{n-1} \Big(\frac{\ell(I)}{\diam(\pom)}\Big)^{2\rho}.
\end{align*}
Thus, Lemma~\ref{lemma:proppde} gives
\begin{align*}
\iint_{B'\cap \Omega} |\nabla v(X)|^2 G_{L_0}(X_\Delta,X) \, d X
&\lesssim
\sum_{I\in \W_{B'}} \omega_{L_0}^{X_{\Delta}}(Q_I)\ell(I)^{1-n} \iint_{I} |\nabla v(X)|^2 \, d X
\\
&\lesssim
\sum_{I\in \W_{B'}} \omega_{L_0}^{X_{\Delta}}(Q_I)  \Big(\frac{\ell(I)}{\diam(\pom)}\Big)^{2\rho}
\\
&\lesssim
\sum_{k: 2^{-k}\lesssim r'} \Big(\frac{2^{-k}}{\diam(\pom)}\Big)^{2\rho}
\sum_{I\in \W_{B'}: \ell(I)=2^{-k}} \omega_{L_0}^{X_{\Delta}}(Q_I),
\end{align*}
where $Q_I\in\dd(\pom)$ is so that $\ell(Q_I)=\ell(I)$ and contains $\widehat{y}_I \in \partial \Omega$ such that
$\dist(I, \partial \Omega)=\dist(\widehat{y}_I, I)$. It is easy to see that if $2^{-k}\lesssim r$, then the family  $\{Q_I\}_{I\in\W_{B'}, \ell(I)=2^{-k}}$ has bounded overlap uniformly on $k$, and also that $Q_I\subset C\Delta'$ for every $I\in\W_{B'}$, where $C$ is some harmless dimensional constant. Hence,
\begin{align*}
\iint_{B'\cap \Omega} |\nabla v(X)|^2 G_{L_0}(X_\Delta,X) \, d X
\lesssim
\sum_{k: 2^{-k}\lesssim r'} \Big(\frac{2^{-k}}{\diam(\pom)}\Big)^{2\rho}\omega_{L_0}^{X_{\Delta}}(C\,\Delta')
\lesssim
\Big(\frac{r'}{\diam(\pom)}\Big)^{2\rho} \omega_{L_0}^{X_{\Delta}}(\Delta').
\end{align*}
This gives the desired estimate.
\end{proof}

Given  $Q_0\in\dd(\partial\Omega), \vartheta \in \mathbb{N}$, and for every $\eta\in(0,1)$ we define the modified non-tangential cone
\begin{equation}\label{def-whitney-eta}
	\Gamma_{Q_0,\eta}^{\vartheta}(x):=\bigcup_{\shortstack{$\scriptstyle Q\in\dd_{Q_0}$\\$\scriptstyle Q\ni x$}}U_{Q,\eta^3}^{\vartheta},\qquad U_{Q,\eta^3}^{\vartheta}:=\bigcup_{\shortstack{$\scriptstyle Q'\in\dd_{Q}$\\$\scriptstyle \ell(Q')>\eta^3\ell(Q)$}}U_{Q'}^{\vartheta}.
\end{equation}
It is not hard to see that the sets $\{U_{Q,\eta^3}^\vartheta\}_{Q\in\dd_{Q_0}}$ have bounded overlap with constant depending on $\eta$.

The following result was obtained in \cite[Lemma 3.10]{CHMT} (for $\beta>0$) and in \cite[Lemma 3.40]{CMO} (for $\beta=0$), both in the context of 1-sided CAD, extending  \cite[Lemma~2.6]{KKoPT} and \cite[Lemma~2.3]{KKiPT}. It is not hard to see that the proof works with no changes in our setting:

\begin{lemma}\label{sq-function->M}
	Let $\Omega\subset\mathbb{R}^{n+1}$, $n\ge 2$, be a 1-sided NTA domain (cf. Definition \ref{def1.1nta}) satisfying the capacity density condition  (cf. Definition \ref{def-CDC}),  and let $Lu=-\div(A\nabla u)$ be a real (non-necessarily symmetric) elliptic operator.	There exist $0<\eta\ll 1$  (depending only on dimension, the 1-sided NTA constants, the CDC constant, and the ellipticity of $L$), and
	$\beta_0\in (0,1)$, $C_\eta\ge 1$ both depending on the same parameters and additionally on $\eta$, such that for every $Q_0\in\dd(\pom)$, for every $0<\beta<\beta_0$, and for every Borel set $F\subset Q_0$ satisfying $\omega_L^{X_{Q_0}}(F)\le\beta \omega_L^{X_{Q_0}}(Q_0)$, there exists a Borel set $S\subset Q_0$ such that the bounded weak solution $u(X)=\omega^X_L(S)$, $X\in\Omega$, satisfies
	\begin{equation}\label{def-mod-sq-function}
		\mathcal{S}_{Q_0,\eta}^{\vartheta}u(x):=\bigg(\iint_{\Gamma_{Q_0,\eta}^{\vartheta}(x)}|\nabla u(Y)|^2\delta(Y)^{1-n}\,dY\bigg)^{\frac12}
		\ge
		C_\eta^{-1} \big(\log(\beta^{-1})\big)^{\frac12}
		,\qquad \forall\,x\in F.
	\end{equation}
Furthermore, in the case $\beta=0$, that is, when $\omega_L^{X_{Q_0}}(F)=0$, there exists a Borel set $S\subset Q_0$ such that the bounded weak solution $u(X)=\omega^X_L(S)$, $X\in\Omega$, satisfies
\begin{equation}\label{def-mod-sq-function:infinte}
	\mathcal{S}_{Q_0,\eta}^{\vartheta}u(x)=\infty,\qquad \forall\,x\in F.
\end{equation}
\end{lemma}

\begin{lemma}\label{lemma:CME-S}
	Let $\Omega\subset\mathbb{R}^{n+1}$, $n\ge 2$, be a 1-sided NTA domain (cf. Definition \ref{def1.1nta}) satisfying the capacity density condition  (cf. Definition \ref{def-CDC}),  and let $Lu=-\div(A\nabla u)$ and $L_0u=-\div(A_0\nabla u)$ be real (non-necessarily symmetric) elliptic operators. There exists $C\ge 1$  (depending only on dimension, the 1-sided NTA constants, the CDC constant, and the ellipticity of $L$ and $L_0$) such that the following holds. Given $B=B(x,r)$ with $x\in \pom$ and $0<r<\diam(\pom)$, and $B'=B(x',r')$ with $x'\in2\Delta$ and $0<r'<r c_0/4$, let $\Delta=B\cap\pom$, $\Delta'=B'\cap\pom$, for every $u\in W^{1,2}_{\loc}(\Omega)\cap L^\infty(\Omega)$ satisfying $Lu=0$ in the weak sense in $\Omega$ there holds
	\begin{multline*}
		\frac1{\omega_{L_0}^{X_{\Delta}}(\Delta')}\iint_{B'\cap \Omega} |\nabla u(X)|^2 G_{L_0}(X_\Delta,X) \, d X
\\
	\le
	C\,\int_{2\Delta'} \mathcal{S}_{2r'}^{C\alpha} u(y)^2\,d\omega_{L_0}^{X_{2\Delta'}}(y)+ C\,\sup \{|u(Y)|: Y\in 2\,B', \delta(Y)\ge r'/C\}^2.	
	\end{multline*}
	\end{lemma}

\begin{proof}
Fix $B$, $B'$, $\Delta$, $\Delta'$, and $u$ as in the statement. Define
	$\W_{B'}:=\{I\in\W: I\cap B'\neq\emptyset\}$ and $\W_{B'}^M:=\{I\in\W_{B'}: \ell(I)<r'/M\}$ for $M\ge 1$ large enough to be taken. For each $I\in\W_{B'}$ pick $Z_I\in I\cap B'$ and  $Q_I\in\dd(\pom)$ so that $\ell(Q_I)=\ell(I)$ and contains $\widehat{y}_I \in \partial \Omega$ such that
	$\dist(I, \partial \Omega)=\dist(\widehat{y}_I, I)$. If $z\in Q_I$ and $I\in \W_{B'}^M$, then
	\[
	|z-x'|
	\le
	|z-\widehat{y}_I|+\dist(\widehat{y}_I, I)+\diam(I)+|Z_I-x'|
	\le
	C_n\ell(I)+r'
	<(1+C_n/M)r'
	<2r',
	\]
	provided $M>C_n$. Hence, $Q_I\subset 2\Delta'$ for every $I\in\W_{B'}^M$. Write $\F$ for the collection of maximal cubes in $\{Q_I\}_{I\in\W_{B'}^M}$, with respect to the inclusion (maximal cubes exist since $Q_I\subset 2\Delta'$ for every $I\in\W_{B'}^M$). Hence $Q_I\subset Q$ for some $Q\in\F$.
	Let $\vartheta=\vartheta_0$ and by construction $I\in \W_{Q_I}^\vartheta\subset \W_{Q_I}^{\vartheta,*}$ (see Section~\ref{section:dyadic}). Hence, for every $y\in Q\in\F$
	\[
	\bigcup_{I\in \W_{B'}^M: y\in Q_I\in\dd_Q} I
	\subset
	\bigcup_{I\in \W_{B'}^M: y\in Q_I\in\dd_Q} U_{Q_I}^{\vartheta}
	\subset
	\bigcup_{y\in Q'\in\dd_Q} U_{Q'}^{\vartheta}
	=
	\Gamma_{Q}^\vartheta(y).
	\]
	This gives
	\begin{align*}
		\Sigma_1:
		&=
		\sum_{I\in \W_{B'}^M} \omega_{L_0}^{X_{\Delta}}(Q_I)\iint_{I} |\nabla u(X)|^2 \, \delta(X)^{1-n} d X
		\\
		&=
		\sum_{Q\in\F} \sum_{I\in \W_{B'}^M:Q_I\in\dd_Q} \omega_{L_0}^{X_{\Delta}}(Q_I)\iint_{I} |\nabla u(X)|^2 \, \delta(X)^{1-n} d X
		\\
		&=
		\sum_{Q\in\F} \int_{Q} \sum_{I\in \W_{B'}^M: y\in Q_I\in\dd_Q} \iint_{I} |\nabla u(X)|^2 \, \delta(X)^{1-n} d X\,  d\omega_{L_0}^{X_{\Delta}}(y)
		\\
		&\le
		\sum_{Q\in\F} \int_{Q} \iint_{\Gamma_{Q}^\vartheta(y) } |\nabla u(X)|^2 \, \delta(X)^{1-n} d X  d\omega_{L_0}^{X_{\Delta}}(y)
		\\
		&=
		\sum_{Q\in\F} \int_{Q} \mathcal{S}_Q^\vartheta u(y)^2\,d\omega_{L_0}^{X_{\Delta}}(y).
	\end{align*}
	To continue let $y\in Q\in\F$ and $X\in \Gamma_Q^\vartheta(y)$. Then $X\in I^*$ with $I\in\W_{Q'}^{\vartheta,*}$ and $y\in Q'\in\dd_{Q}$. Thus,
	\[
	|X-y|
	\le
	\diam(I^*)+\dist(I,Q')+\diam(Q')
	\lesssim_{\vartheta}
	\ell(I)
	\approx
	\delta(X)
	\lesssim r'/M.
	\]
	where we have used  \eqref{definicionkappa12}  and the last estimate holds since $\ell(I)<r'/M$ for every $I\in \W_{B'}^M$. This shows that
	taking $M$ large enough
	$X\in\Gamma^{\alpha'}_{2 r'}(y)$ for some $\alpha'=\alpha'(\vartheta)$.  Note also that $2r'<r c_0/2<\diam(\pom)$, and we can now conclude that
	\begin{align*}
		\Sigma_1
		\lesssim
		\sum_{Q\in\F} \int_{Q} \mathcal{S}_{2r'}^{\alpha'} u(y)^2\,d\omega_{L_0}^{X_{\Delta}}(y)
		\lesssim
		\int_{2\Delta'} \mathcal{S}_{2r'}^{\alpha'} u(y)^2\,d\omega_{L_0}^{X_{\Delta}}(y)
		\approx
		\omega_{L_0}^{X_\Delta}(\Delta') \int_{2\Delta'} \mathcal{S}_{2r'}^{\alpha'} u(y)^2\,d\omega_{L_0}^{X_{2\Delta'}}(y),
		\end{align*}
	where we have used Lemma~\ref{lemma:proppde}.
	
	Now, we note that for each $I\in \W_{B'}\setminus \W_{B'}^M$ we have $\ell(Q_I)=\ell(I)\approx_M r'$, hence for every  $Y\in I^*$ we have
	\[
	r'\lesssim_M \delta(Y) \le |Y-Z_I|+\delta(Z_I)
	\le
	\diam(I^*)+\delta(Z_I)
	<
	\dist(I,\pom)+\delta(Z_I)
	\le
	2\,\delta(Z_I)
	\le
	2|Z_I-x'|
	<2r'.
	\]
	Also,
	\[
	|\widehat{y}_I-x'|
	+
	\dist(\widehat{y}_I, I)+\diam(I)+|Z_I-x'|
	\lesssim
	\dist(I,\pom)+|Z_I-x'|
	\le
	2|Z_I-x'|<2r'.
	\]
	Thus, Lemma~\ref{lemma:proppde} implies that $\omega_{L_0}^{X_{\Delta}}(Q_I)\approx_M \omega_{L_0}^{X_{\Delta}}(\Delta')$. As a consequence of this, we get
	\begin{align*}
		%	\iint_{B'\cap \Omega} |\nabla u(X)|^2 G_{L_0}(X_\Delta,X) \, d X
		\Sigma_2:
		&=
		\sum_{I\in \W_{B'}\setminus \W_{B'}^M} \omega_{L_0}^{X_{\Delta}}(Q_I)\iint_{I} |\nabla u(X)|^2 \, \delta(X)^{1-n} d X
		\\
		&\lesssim
		\omega_{L_0}^{X_{\Delta}}(\Delta')\sum_{I\in \W_{B'}\setminus \W_{B'}^M}  \ell(I)^{1-n} \iint_{I} |\nabla u(X)|^2 d X
		\\
		&\lesssim
		\omega_{L_0}^{X_{\Delta}}(\Delta')
		\sum_{I\in \W_{B'}\setminus \W_{B'}^M} \ell(I)^{-n-1}\iint_{I^*} |u(X)|^2 \, d X
		\\
		&\lesssim
	\omega_{L_0}^{X_{\Delta}}(\Delta')\,\#(\W_{B'}\setminus \W_{B'}^M) \sup \{|u(Y)|: Y\in 2\,B', \delta(Y)\ge r'/C\}^2
		\\
		&\lesssim_M
		\omega_{L_0}^{X_{\Delta}}(\Delta') \sup \{|u(Y)|: Y\in 2\,B', \delta(Y)\ge r'/C\}^2,
	\end{align*}
	where we have used that $\W_{B'}\setminus \W_{B'}^M$ has bounded cardinality depending on $n$ and $M$.
	
	To complete the proof we use Lemma~\ref{lemma:proppde} and the estimates proved for $\Sigma_1$ and $\Sigma_2$:
	\begin{align*}
		&\iint_{B'\cap \Omega} |\nabla u(X)|^2 G_{L_0}(X_\Delta,X) \, d X
		\le
		\sum_{I\in \W_{B'}}
		\iint_{I} |\nabla u(X)|^2 G_{L_0}(X_\Delta,X) \, d X
		\\
		&\qquad \approx
		\sum_{I\in \W_{B'}}
		\omega_{L_0}^{X_{\Delta}}(Q_I)\iint_{I} |\nabla u(X)|^2 \, \delta(X)^{1-n} d X
		\\
		&\qquad
		=
		\Sigma_1+\Sigma_2
		\\
		&\qquad  \lesssim
		\omega_{L_0}^{X_\Delta}(\Delta')\Big( \int_{2\Delta'} \mathcal{S}_{2r'}^{\alpha'} u(y)^2\,d\omega_{L_0}^{X_{2\Delta'}}(y)+ \sup \{|u(Y)|: Y\in 2\,B', \delta(Y)\ge r'/C\}^2\Big).
	\end{align*}
	This completes the proof.
\end{proof}

For the following result we need to introduce some notation:
\begin{equation*}
	\mathcal{A}_{r}^\alpha F(x):=\bigg(\iint_{\Gamma_{r}^\alpha(x)}|F(Y)|^2\,dY\bigg)^{\frac12},
	\qquad x\in\pom,\ 0<r<\infty, \alpha > 0, 	
\end{equation*}
 for any $F\in L^2_\loc(\Omega\cap B(x,r))$.

\begin{lemma}\label{lemma-ChAO}
Let $\Omega\subset\mathbb{R}^{n+1}$, $n\ge 2$, be a 1-sided NTA domain (cf. Definition \ref{def1.1nta}) satisfying the capacity density condition  (cf. Definition \ref{def-CDC}),  and let $L_0u=-\div(A_0\nabla u)$ be a real (non-necessarily symmetric) elliptic operator. Given $0<q<\infty$, $0<\alpha,\alpha'<\infty$ there exists $C\ge 1$  (depending only on dimension, the 1-sided NTA constants, the CDC constant, the ellipticity of $L_0$, $q$, $\alpha$, and $\alpha'$) such that the following holds. Given $B=B(x,r)$ with $x\in \pom$ and $0<r<\diam(\pom)$, let $\Delta=B\cap\pom$, for every $F\in L^2_{\loc}(\Omega)$ there holds
\begin{equation}\label{eqn:ChAO:S}
	\|\mathcal{A}_{r}^\alpha F\|_{L^q(\Delta, \omega_{L_0}^{X_{\Delta}})}
\le
C
\|\mathcal{A}_{3r}^{\alpha'} F\|_{L^q(3\Delta, \omega_{L_0}^{X_{3\Delta}})},
\qquad F\in L^2_{\loc}(\Omega\cap 6B),
\end{equation}
and
\begin{equation}\label{eqn:ChAO:N}
\|\mathcal{N}_{r}^\alpha F\|_{L^q(\Delta, \omega_{L_0}^{X_{\Delta}})}
\le
C
\|\mathcal{N}_{4r}^{\alpha'} F\|_{L^q(4\Delta, \omega_{L_0}^{X_{4\Delta}})},
\qquad F\in \mathscr{C}(\Omega\cap 8B).
\end{equation}
\end{lemma}

\begin{proof}
We start with \eqref{eqn:ChAO:S} and borrow some ideas from \cite[Proposition~3.2]{Martell-Prisuelos}. We may assume that $\alpha>\alpha'$, otherwise the desired estimate follows trivially. Let $v\in A_{\infty}(\pom, \omega_{L_0})$.  By the classical theory of weights (cf.~\cite{CF-1974, GR}), we can find $p\in(1,\infty)$ such for every $\Delta$ as in the statement we have
\[
C_0:=\sup_{\Delta}[v]_{A_{p}(\Delta, \omega_{L_0})}:=\sup_{\Delta}\sup_{\Delta'} \Big(\fint_{\Delta'} v(x)\,d\omega_{L_0}^{X_{\Delta}}(x)\Big)\Big(\fint_{\Delta'} v(x)^{1-p'}\,d\omega_{L_0}^{X_{\Delta}}(x)\Big)^{p-1}<\infty,
\]
where the sups are taken over all $\Delta'=B'\cap \partial\Omega$ with $B'\subset 5B$, $B'=B(x',r')$, $x'\in \partial\Omega$, $0<r'<\diam(\partial\Omega)$, and where  $C_0$ depends on $[v]_{A_{\infty}(\pom, \omega_{L_0})}$.
Note that for any such $\Delta'$ and for any Borel set $F\subset\Delta'$ we have, by H\"older's inequality,
\begin{multline}\label{doubling-Ap}
\Big(	\frac{\omega_{L_0}^{X_{\Delta}}(F)}{\omega_{L_0}^{X_{\Delta}}(\Delta')}\Big)^p
=
\Big(	\fint_{\Delta'} \mathbf{1}_F d\omega_{L_0}^{X_{\Delta}}\Big)^p
=
\Big(\fint_{\Delta'} \mathbf{1}_F v^{\frac1p}\,v^{-\frac1p}d\omega_{L_0}^{X_{\Delta}}\Big)^p
\\
\le
\Big(\fint_{\Delta'} \mathbf{1}_F v\,d\omega_{L_0}^{X_{\Delta}}\Big)
\Big(\fint_{\Delta'} v^{1-p'}d\omega_{L_0}^{X_{\Delta}}\Big)^{p-1}
\\
\le
C_0
\Big(\fint_{\Delta'} \mathbf{1}_F v\,d\omega_{L_0}^{X_{\Delta}}\Big)\Big(\fint_{\Delta'} v\,d\omega_{L_0}^{X_{\Delta}}\Big)^{-1}
=
C_0
\frac{\int_{F} v\,d\omega_{L_0}^{X_{\Delta}}}{\int_{\Delta'} v\,d\omega_{L_0}^{X_{\Delta}}}.
\end{multline}

Let $y\in\Delta$ and $X\in \Gamma_{r}^\alpha(y)$ and pick $\widehat{x}$ so that $|X-\widehat{x}|=\delta(X)$. Then one can easily see that
\[
X\in 2B,\quad \delta(X)<r,\quad y\in\Delta\big(\widehat{x}, \min\{(3+\alpha)\delta(X), 2r\}\big)=:\widetilde{\Delta},
\quad
\widetilde{B}:=B\big(\widehat{x}, \min\{(3+\alpha)\delta(X), 2r\}\big)
\subset 5B.
\]
Then, by \eqref{doubling-Ap} and Lemma~\ref{lemma:proppde} we get
\begin{align*}
\int_{\widetilde{\Delta}} v\,d\omega_{L_0}^{X_{\Delta}}
\le
C_0
\Big(	\frac{\omega_{L_0}^{X_{\Delta}}(\widetilde{\Delta})}{\omega_{L_0}^{X_{\Delta}}(\widehat{\Delta})}\Big)^{p} \int_{\widehat{\Delta}} v\,d\omega_{L_0}^{X_{\Delta}}
\lesssim_{\alpha,\alpha',p} C_0
\int_{\widehat{\Delta}} v\,d\omega_{L_0}^{X_{\Delta}},
\end{align*}
where $\widehat{\Delta}:=\Delta(\widehat{x}, \min\{\alpha',1\}\delta(X))$. Moreover, if $X\in 2B$ with $\delta(X)<r$ and $y\in \widehat{\Delta}$ one can easily show that
\[
|y-x|<3r,\quad |X-y|\le \min\{1+\alpha',2\}\delta(X).
\]
If we now combine the previous estimates, then we conclude that
\begin{align*}
\|\mathcal{A}_{r}^\alpha F\|_{L^2(\Delta, v\,d\omega_{L_0}^{X_{\Delta}})} ^2
&=
\int_{\Delta} \iint_{\Gamma_{r}^\alpha(y)}|F(X)|^2\,dX\,v(y)\,d\omega_{L_0}^{X_{\Delta}}(y)
\\
&\le
\iint_{2B\cap\{\delta(X)<r\}} |F(X)|^2\, \Big(\int_{\widetilde{\Delta}} v(y)\,d\omega_{L_0}^{X_{\Delta}}(y)\Big)  dX
\\
&
\lesssim_{\alpha,\alpha',p}
C_0
\iint_{2B\cap\{\delta(X)<r\}} |F(X)|^2\, \Big(\int_{\widehat{\Delta}} v(y)\,d\omega_{L_0}^{X_{\Delta}}(y)\Big)  dX
\\
&
\le
C_0
\int_{3\Delta} \iint_{\Gamma_{3r}^{\alpha'}(y)}|F(X)|^2\,dX\,v(y)\,d\omega_{L_0}^{X_{\Delta}}(y)
\\
&
=
C_0 \|\mathcal{A}_{3r}^{\alpha'} F\|_{L^2(3\Delta, v\,d\omega_{L_0}^{X_{\Delta}})}^2.
\end{align*}
We can now extrapolate (locally in $3\Delta$) as in \cite[Corollary 3.15]{CMP} to conclude that
\[
\|\mathcal{A}_{r}^\alpha F\|_{L^q(\Delta, v\,d\omega_{L_0}^{X_{\Delta}})}
\lesssim_{\alpha,\alpha',q}
\|\mathcal{A}_{3r}^{\alpha'} F\|_{L^1(3\Delta, v\,d\omega_{L_0}^{X_{\Delta}})}.
\]
The desired estimate follows at once by taking $v\equiv 1$ which clearly belongs to $A_{\infty}(\pom, \omega_{L_0})$.

Let us next consider \eqref{eqn:ChAO:N}. First, introduce
\[
\M^\Delta_{\omega_{L_0}} h(z)
:=
\sup_{0<s\le 3r} \fint_{\Delta(z,s)} |h|\,d\omega_{L_0}^{X_\Delta} =
\sup_{0<s\le 3r} \fint_{\Delta(z,s)} |h|\mathbf{1}_{4\Delta}\,d\omega_{L_0}^{X_\Delta}
,\qquad z\in\Delta.
\]
We proceed as in \cite[Proposition~2.2]{HMiT} and write for any $\lambda>0$ and $\beta>0$
\[
E(\beta,r,\lambda):=\{y\in\pom: \mathcal{N}_{r}^\beta F(y)>\lambda\}.
\]
Let $y\in E(\alpha,r,\lambda)\cap\Delta$. Hence, there is $X\in \Gamma_{r}^\alpha(y)$ with $|F(X)|>\lambda$. Pick $\widehat{x}\in \pom$ so that $|X-\widehat{x}|=\delta(X)$. Note that
\[
 \widetilde{\Delta}:=\Delta(\widehat{x},\min\{1,\alpha'\}\delta(X))\subset \widehat{\Delta}:=\Delta(y, \min\{(2+\alpha+\alpha')\delta(X), 3r\})
 \quad\text{and}\quad \widetilde{\Delta}\subset 2\Delta.
\]
One can easily see that if $z\in \widetilde{\Delta}$ then $X\in\Gamma_{3r}^{\alpha'}(z)$. Hence,
\[
 \widetilde{\Delta}\subset E(\alpha',3r,\lambda)\cap \widehat{\Delta}%\cap 2\Delta
\]
and
\[
\M^\Delta_{\omega_{L_0}} \mathbf{1}_{ E(\alpha',3r,\lambda)}(y)
\ge
\frac{\omega_{L_0}^{X_\Delta}( E(\alpha',3r,\lambda)\cap \widehat{\Delta})}{\omega_{L_0}^{X_\Delta}(\widehat{\Delta})}
\ge
\frac{\omega_{L_0}^{X_\Delta}(\widetilde{\Delta})}{\omega_{L_0}^{X_\Delta}(\widehat{\Delta})}
>
\gamma=\gamma_{\alpha,\alpha'},
\]
where in the last estimate we have used that
\[
\omega_{L_0}^{X_\Delta}(\widehat{\Delta})
\le
\omega_{L_0}^{X_\Delta}(\Delta(\widehat{x},\min\{(4+2\alpha+\alpha')\delta(X),5r\}))
\lesssim_{\alpha,\alpha'}
\omega_{L_0}^{X_\Delta}(\widetilde{\Delta}).
\]
We have then shown that
\[
 E(\alpha,r,\lambda)\cap\Delta\subset \{y\in\Delta:\M^\Delta_{\omega_{L_0}} \mathbf{1}_{ E(\alpha',3r,\lambda)}(y)>\gamma  \},
\]
and by the  Hardy-Littlewood maximal inequality we get
\begin{multline*}
	\omega_{L_0}^{X_\Delta}( E(\alpha,r,\lambda)\cap\Delta)
\le
\omega_{L_0}^{X_\Delta}(\{y\in\Delta:\M^\Delta_{\omega_{L_0}} \mathbf{1}_{ E(\alpha',3r,\lambda)}(y)>\gamma  \})
\\
\lesssim
\omega_{L_0}^{X_\Delta}( E(\alpha',3r,\lambda)\cap4 \Delta)
\lesssim
\omega_{L_0}^{X_{4\Delta}}( E(\alpha',4r,\lambda)\cap4 \Delta).
\end{multline*}
This readily implies \eqref{eqn:ChAO:N}.
\end{proof}

\subsection{Proof of \texorpdfstring{$\mathrm{(a)}_{p'}\ \Longrightarrow\ \mathrm{(b)}_p$}{(a)p' implies (b)p}}
Fix  $\alpha>0$ and $N\ge 1$. Take $\Delta_0 = \Delta(x_0, r_0)$ with $x_0 \in \partial \Omega$ and $0 < r_0 <\diam(\pom)$
and fix $f \in \mathscr{C}(\partial \Omega)$ with $\supp f \subset N\Delta_0$. We may assume that $N r_0<4\,\diam(\pom)$, otherwise $\pom$ is bounded and $4\diam(\pom)/N\le r_0<\diam(\pom)$ and we can work with $N'= 2\diam(\pom)/r_0\in(2, N/2]$ and $N'\Delta_0=\pom$.

Let $u$ be the associated elliptic measure $L$-solution as in \eqref{u-elliptic:L-sol}. Assume $\omega_L \in RH_{p'}(\partial \Omega, \omega_{L_0})$ and our goal is to obtain that \eqref{solv-Lp} holds. By Gehring's lemma \cite{G} (see also \cite{CF-1974}) there exists $s>1$ such that $\omega_{L}\in RH_{p's} (\pom,\omega_{L_0})$.

Introduce the family of pairwise disjoint cubes
\[
\F_{\Delta_0}:=\{Q\in\dd(\pom): (N+3\Xi)r_0<\ell(Q)\le 2(N+3\Xi) r_0,\ Q\cap 3\Xi\Delta_0\neq\emptyset\}.
\]
Take $x \in \Delta_0$ and $X \in \Gamma_{r_0}^\alpha(x)$. Let $I_X \in \mathcal{W}$ be such that $X \in I_X$.
Take $y_X \in \partial \Omega$ such that $\dist(I_X, \partial \Omega) = \dist(I_X, y_X)$ and let $Q_X \in \dd$ be the unique dyadic cube satisfying $\ell(Q_X) = \ell(I_X)$ and $y_X \in Q_X$. By construction (see Section~\ref{section:dyadic}), $I_X \in \mathcal{W}_{Q_X}^{\vartheta,*}$ and thus $I^* \subset \Gamma_{Q_X}(y_X)$. Thus, by the properties of the Whitney cubes
\[
\delta(X)
\le
|X-y_X|
\le
\diam(I_X)+\dist(I_X, y_X)
\le
\frac54\dist(I_X, \pom)
\le
\frac54\delta(X)
\]
and
\begin{align*}
4\ell(Q_X)
=
4\ell(I_X)
\le
\dist(I_X,\pom)
\le
\delta(X)
\le
\frac54\dist(I_X, \pom)
\le
50\sqrt{n+1}\ell(I_X)
=
50\sqrt{n+1}\ell(Q_X).
\end{align*}
These and the fact that  $X \in \Gamma_{r_0}^\alpha(x)$ give
\[
\ell(Q_X)
<
\frac1{4}\,\delta(X)
\le
\frac14 |X-x|
<
\frac14r_0.
\]
Also, for every $z\in Q_X$
\[
|z-x_0|
\le
|z-y_X|+|y_X-X|+|X-x|+|x-x_0|
<
2\Xi\ell(Q_X)+ \frac94 |X-x| +r_0
<
(\Xi+ 4)\,r_0
\le
3\Xi r_0,
\]
since $\Xi\ge 2$, and
\[
|z-x|
\le
|z-y_X|+|y_X-X|+ |X-x|
<
2\Xi\ell(Q_X)+(3+\alpha)\delta(X)
<
(2\Xi+\alpha) \delta(X)
=:C_\alpha\delta(X),
\]
since $X \in \Gamma_{r_0}^\alpha(x)$. Thus, $Q_X\subset 3\Xi\Delta_0\cap \Delta(x, C_\alpha\delta(X))$ and there exists a unique $\widetilde{Q}_X\in\F_{\Delta_0}$ such that $Q_X\subsetneq \widetilde{Q}_X$. In particular,  $X\in I_X\subset U_{Q_X}\subset \Gamma_{\widetilde{Q}_X}(y)$ for all $y \in Q_X$ and
\begin{equation*}
	|u(X)| \leq \mathcal{N}_{\widetilde{Q}_X} u (y), \qquad \text{for all } y \in Q_X.
\end{equation*}
Taking the average over $Q_X$ with respect to $\omega_{L_0}^{X_{\Delta_0}}$ we arrive at
\begin{multline*}
	|u(X)| \leq
	\fint_{Q_X} \mathcal{N}_{\widetilde{Q}_X} u (y) \, d \oLo^{X_{\Delta_0}}(y)
	\le
	\fint_{Q_X} \sup_{Q \in \mathcal{F}_{\Delta_0}} \mathcal{N}_Q u (y) \, d \oLo^{X_{\Delta_0}}(y)
	\\
	\lesssim_\alpha
	\fint_{\Delta(x, C_\alpha \delta(X))} \sup_{Q \in \mathcal{F}_{\Delta_0}} \mathcal{N}_Q u (y) \, d \oLo^{X_{\Delta_0}}(y)
	\le
	\sup_{0<r \leq C_\alpha r_0}\fint_{ \Delta(x,r)}  \sup_{Q \in \mathcal{F}_{\Delta_0}} \mathcal{N}_Q u (y) \, d \oLo^{X_{\Delta_0}}(y),
\end{multline*}
where in the last inequality we have used that $\delta(X)\le |X-x|< r_0$ since $\Gamma_{r_0}^\alpha(x)\subset B(x,r_0)$. Taking now the supremum over all $X \in \Gamma_{r_0}^\alpha(x)$, we arrive at
\begin{equation*}
	\mathcal{N}^\alpha_{r_0} u (x)
	\lesssim_\alpha
	\sup_{0<r \leq C_\alpha r_0}\fint_{ \Delta(x,r)}  \sup_{Q \in \mathcal{F}_{\Delta_0}} \mathcal{N}_Q u (y) \, d \oLo^{X_{\Delta_0}}(y), \qquad \text{for all }x \in \Delta_0.
\end{equation*}
Applying the Hardy-Littlewood maximal inequality and the fact that the set $\mathcal{F}_{\Delta_0}$ has bounded cardinality, we have
\begin{multline}\label{34qweva}
	\|\mathcal{N}^\alpha_{r_0} u\|_{L^{p}(\Delta_0, \omega_{L_0}^{X_{\Delta_0}})}
		\lesssim_\alpha
	\left\|  \sup_{0<r \leq C_\alpha r_0}\fint_{ \Delta(\cdot,r)}  \sup_{Q \in \mathcal{F}_{\Delta_0}} \mathcal{N}_Q u (y) \, d \oLo^{X_{\Delta_0}}(y)  \right\|_{L^{p}(\Delta_0, \omega_{L_0}^{X_{\Delta_0}})}
	\\
	 \lesssim
	\Big\|\sup_{Q \in \mathcal{F}_{\Delta_0}} \mathcal{N}_Q u \Big\|_{L^{p}(\Delta_0, \omega_{L_0}^{X_{\Delta_0}})}
\lesssim
\sup_{Q \in \mathcal{F}_{\Delta_0}}  \| \mathcal{N}_Q u\|_{L^{p}(\Delta_0, \omega_{L_0}^{X_{\Delta_0}})}
\approx_N
\sup_{Q \in \mathcal{F}_{\Delta_0}}  \| \mathcal{N}_Q u\|_{L^{p}(Q, \omega_{L_0}^{X_{\Delta_0}})},
\end{multline}
where we have used that for every $Q \in \mathcal{F}_{\Delta_0}$ we have $\supp (\mathcal{N}_Q u)\subset Q$.

Let us also observe that for every $Q\in \mathcal{F}_{\Delta_0}$ we can pick $y_Q\in Q\cap 3\Xi\Delta_0$ so that if $z\in N\Delta_0$ there holds
\[
|z-x_Q|
\le
|z-x_0|+|x_0-y_Q|+|y_Q-x_Q|
\le
(N+3\Xi) r_0+\Xi r_Q
<2\Xi r_Q.
\]
That is, $N\Delta_0\subset 2\widetilde{\Delta}_{Q}$ and we are now ready to invoke \cite[Lemma 3.20]{AHMT} to see that
\begin{equation}\label{gweravgvb}
\mathcal{N}_{Q} u (x)
\lesssim
\sup_{\substack{\Delta \ni x \\ 0 < r_\Delta < 4\Xi r_Q}} \fint_{\Delta} |f(y)| \, d \omega_{L}^{X_{Q}} (y), \quad x \in Q.
\end{equation}
To continue let $x\in Q\in \F_{\Delta_0}$ and let $\Delta$ be a surface ball such that $x \in \Delta$ and $0 < r_\Delta < 4\Xi r_Q$. In particular, $\Delta\subset C_N\Delta_0=\widetilde{\Delta}_0$ and $Q\subset \widetilde{\Delta}_0$. Note that $\omega_{L_0}^{X_{\Delta_0}}\approx_N \omega_{L_0}^{X_{\widetilde{\Delta}_0}}$ by Harnack's inequality and the fact that $\delta(X_{\Delta_0})\approx r_0$, $\delta(X_{\widetilde{\Delta}_0})\approx_N r_0$, and $|X_{\Delta_0}-X_{\widetilde{\Delta}_0}|\lesssim_N r_0$.

Recall that $\omega_L \in RH_{p' s}(\pom, \omega_{L_0})$ implies  $\omega_L \in RH_{p' s}(\widetilde{\Delta}_0, \omega^{X_{\widetilde{\Delta}_0}}_{L_0})$ (uniformly). Therefore, using H\"older's inequality and recalling that $h(\cdot; L, L_0, X)$ denotes the Radon-Nikodym derivative of $\omega_L^X$ with respect to $\omega_{L_0}^X$, we get
\begin{align*}
	\fint_{\Delta} |f(y)| \, d \omega_L^{X_{\Delta_0}} (y)
	&
	\approx_N
	\frac{\omega_{L_0}^{X_{\widetilde{\Delta}_0}} (\Delta)}{\omega_L^{X_{\widetilde{\Delta}_0}}(\Delta)} \fint_\Delta |f(y)| h(y; L, L_0, X_{\widetilde{\Delta}_0}) \, d \omega_{L_0}^{X_{\widetilde{\Delta}_0}} (y) \\
	&
	\leq
	\frac{\omega_{L_0}^{X_{\widetilde{\Delta}_0}} (\Delta)}{\omega_L^{X_{\widetilde{\Delta}_0}}(\Delta)} \bigg(\fint_{\Delta} h(y;L,L_0,X_{\widetilde{\Delta}_0})^{p's} \, d \omega_{L_0}^{X_{\widetilde{\Delta}_0}}(y) \bigg)^{\frac1{p's}} \bigg( \fint_\Delta |f(y)|^{(p's)'} \, d \omega_{L_0}^{X_{\widetilde{\Delta}_0}} (y) \bigg)^{\frac1{(p's)'}}  \\
	&
	\lesssim
	\frac{\omega_{L_0}^{X_{\widetilde{\Delta}_0}} (\Delta)}{\omega_L^{X_{\widetilde{\Delta}_0}}(\Delta)} \fint_\Delta h(y;L,L_0,X_{\widetilde{\Delta}_0}) \, d \omega_{L_0}^{X_{\widetilde{\Delta}_0}}(y)
	\bigg( \fint_\Delta |f(y)|^{(p's)'} \, d \omega_{L_0}^{X_{\widetilde{\Delta}_0}} (y) \bigg)^{^{\frac1{(p's)'}}}
	\\
	&
	=
	 \bigg( \fint_\Delta |f(y)|^{(p's)'} \, d \omega_{L_0}^{X_{\widetilde{\Delta}_0}} (y) \bigg)^{\frac1{(p's)'}}.
\end{align*}
This, \eqref{gweravgvb}, and \eqref{34qweva} yield
\begin{multline*}
\|\mathcal{N}^\alpha_{r_0} u\|_{L^{p}(\Delta_0, \omega_{L_0}^{X_{\Delta_0}})}^{p}
\lesssim_{\alpha,N}
\sup_{Q \in \mathcal{F}_{\Delta_0}}
\int_{\widetilde{\Delta}_0}
\bigg(
\sup_{\substack{\Delta \ni x \\ 0 < r_\Delta < 4\Xi r_Q}} \fint_\Delta |f(y)|^{(p's)'} \, d \omega_{L_0}^{X_{\Delta_0}} (y)
\bigg)^{\frac{p}{(p's)'}} d\omega_{L_0}^{X_{\widetilde{\Delta}_0}}(x)
\\
\lesssim
\int_{\widetilde{\Delta}_0} |f(x)|^{p} d\omega_{L_0}^{X_{\widetilde{\Delta}_0}}(x)
\approx_N
\|f\|_{L^{p}(N\Delta_0, \omega_{L_0}^{X_{\Delta_0}})}^{p},
\end{multline*}
where we have used the boundedness of the local Hardy-Littlewood maximal function in the second term on $L^{\frac{p}{(p's)'}}(\widetilde{\Delta}_0,\omega_{L_0}^{X_{\widetilde{\Delta}_0}})$, which follows from $p>(p's)'$ and the fact that $\omega_{L_0}^{X_{\widetilde{\Delta}_0}}$ is doubling in $10\widetilde{\Delta}_0$. This completes the proof of $\mathrm{(b)}_p$.\qed

\subsection{Proof of \texorpdfstring{$\mathrm{(b)}_p\ \Longrightarrow\ \mathrm{(a)}_{p'}$}{(b)p implies (a)p'}}
Fix $p\in (1,\infty)$ and assume that $L$ is $L^{p}(\omega_{L_0})$-solvable. That is, for some fixed $\alpha_0$ and some $N\ge 1$ there exists $C_{\alpha_0, N}\ge 1$ (depending only on $n$, the 1-sided NTA constants, the CDC constant, the ellipticity of $L_0$ and $L$, $\alpha_0$, $N$, and $p$) such that \eqref{solv-Lp} holds for $u$ as in \eqref{u-elliptic:L-sol} for any $f\in\mathscr{C}(\pom)$ with $\supp f\subset N\Delta_0$. From this and \eqref{eqn:ChAO:N} we conclude that we can assume that
$\alpha\ge c_0^{-1}-1$, where $c_0$ is the Corkscrew constant (cf.~Definition~\ref{def1.cork}), and we have
\begin{equation}\label{solv-Lp:proof}
\|\mathcal{N}^\alpha_{r_0} u\|_{L^p(\Delta_0, \omega_{L_0}^{X_{\Delta_0}})}
\lesssim_{\alpha,\alpha_0}	
\|\mathcal{N}^{\alpha_0}_{4r_0} u\|_{L^p(4\Delta_0, \omega_{L_0}^{X_{4\Delta_0}})}
	\leq C_{\alpha_0,N} \|f\|_{L^p(N\Delta_0, \omega_{L_0}^{X_{\Delta_0}})},
\end{equation}
for $u$ as in \eqref{u-elliptic:L-sol} with $f\in\mathscr{C}(\pom)$ with $\supp f\subset N\Delta_0$ and for any $\Delta_0 = \Delta(x_0,r_0)$, $x_0 \in \partial \Omega$ and $0 < r_0 < \diam(\partial \Omega)/4$. It is routine to see this estimate also holds with $r_0\approx\diam(\pom)$. Indeed, by splitting $f$ into its positive and negative part we may assume that $f\ge 0$. In that case if $x\in\pom$ and $X\in\Gamma^{\alpha}_{r_0}(x)\setminus\Gamma^{\alpha}_{\diam(\pom)/5}(x)$ we have that $\delta(X)\approx\diam(\pom)$ and by \eqref{cone-CKS} one has that $X':=X_{\Delta(x,\diam(\pom)/5)}\in \Gamma^{\alpha}_{\diam(\pom)/5}(x)$. Harnack's inequality implies then that $u(X)\approx u(X')$ and this shows that $\mathcal{N}^{\alpha}_{r_0} u(x)\lesssim \mathcal{N}^{\alpha}_{\diam(\pom)/5} u(x)$. Further details are left to the interested reader.

We claim that for every  $\Delta_0 = \Delta(x_0,r_0)$, $x_0 \in \partial \Omega$ and $0 < r_0 < \diam(\partial \Omega)$, and for every $ f\in\mathscr{C}(\pom)$ with $\supp f\subset N\Delta_0$
\begin{equation}\label{AWfeve}
\Big|\int_{\Delta_0} f(y)\,d\omega_L^{X_{\Delta_0}}(y)\Big| \lesssim_{\alpha, N}\|f\|_{L^{p}(N \Delta_0, \omega_{L_0}^{X_{\Delta_0}})}.
\end{equation}
To see this let $u$ be the $L$-solution with datum $|f|$ (see \eqref{u-elliptic:L-sol}). Write $X_0:=X_{\Delta_0}$ and $\widetilde{X}_0:=X_{(2+\alpha)^{-1}\Delta_0}$. Note that $\delta(X_0)\approx r_0$, $\delta(\widetilde{X}_0)\approx_\alpha r_0$, and $|X_0-\widetilde{X}_0|<2\,r_0$. Hence Harnack's inequality yields $u(\widetilde{X}_0)\approx_\alpha
u(X_0)$. The choice of $\alpha$ guarantees that $\widetilde{X}_0\in \Gamma_{(2+\alpha)^{-1}r_0}^\alpha(x_0)\subset\Gamma_{r_0}^\alpha(x_0)$, see \eqref{cone-CKS}. Let $\widetilde{x}_0\in\pom$ so that $\delta(\widetilde{X}_0)=|\widetilde{X}_0-\widetilde{x}_0|$. Clearly, for every $z\in \Delta(\widetilde{x}_0,\alpha\delta(\widetilde{X}_0))$,
\[
|\widetilde{X}_0-z|
\le
|\widetilde{X}_0-\widetilde{x}_0|+|\widetilde{x}_0-z|
<
(1+\alpha)\delta(\widetilde{X}_0)
\le
\frac{1+\alpha}{2+\alpha }r_0
<r_0,
\]
thus $\widetilde{X}_0\in \Gamma_{r_0}^\alpha(z) $ and
\begin{equation*}
\mathcal{N}^\alpha_{r_0} u(z)
\ge
u(\widetilde{X}_0)
\approx_\alpha
u(X_0),
\qquad\text{for every }z\in \Delta(\widetilde{x}_0,\alpha\delta(\widetilde{X}_0)).
\end{equation*}
Note also that if $z\in \Delta(\widetilde{x}_0,\alpha\delta(\widetilde{X}_0))$ then
\[
|z-x_0|
\le
|z-\widetilde{x}_0|+|\widetilde{x}_0-\widetilde{X}_0|+ |\widetilde{X}_0-x_0|
<
(\alpha+1)\delta(\widetilde{X}_0)+|\widetilde{X}_0-x_0|
\le
(\alpha+2)|\widetilde{X}_0-x_0|
\le
r_0,
\]
hence  $\Delta(\widetilde{x}_0,\alpha\delta(\widetilde{X}_0))\subset\Delta_0$. Additionally, if $z\in\Delta_0$ then
\[
|z-\widetilde{x}_0|
\le |z-x_0|+|x_0-\widetilde{X}_0|+ |\widetilde{X}_0-\widetilde{x}_0|
<
r_0+|x_0-\widetilde{X}_0|+ \delta(\widetilde{X}_0)
\le
r_0+2|x_0-\widetilde{X}_0|
\le
\Big(1+\frac2{2+\alpha}\Big)r_0
\le
2r_{0},
\]
and this shows that $\Delta_0\subset\Delta(\widetilde{x}_0, 2r_0)$. This together with Lemma~\ref{lemma:proppde} gives
\[
1
\lesssim
\omega_{L_0}^{X_0}(\Delta_0)
\le
\omega_{L_0}^{X_0}(\Delta(\widetilde{x}_0, 2r_0))
\lesssim_\alpha
\omega_{L_0}^{X_0}(\Delta(\widetilde{x}_0,\alpha\,c_0\,r_0/(2+\alpha)))
\le
\omega_{L_0}^{X_0}(\Delta(\widetilde{x}_0,\alpha\delta(\widetilde{X}_0)))
\]
and the previous estimates readily give \eqref{AWfeve}:
\[
\Big|\int_{\Delta_0} f(y)\,d\omega_L^{X_{\Delta_0}}(y)\Big|
\le
u(X_0)
\lesssim_\alpha
u(\widetilde{X}_0)\omega_{L_0}^{X_0}(\Delta(\widetilde{x}_0,\alpha\delta(\widetilde{X}_0)))^{\frac1p}
\le
\|\mathcal{N}^\alpha_{r_0} u\|_{L^{p}(\Delta_0, \omega_{L_0}^{X_{0}})}
\lesssim_{\alpha, N}
\|f\|_{L^{p}(N \Delta_0, \omega_{L_0}^{X_{0}})}.
\]

To proceed we fix $\Delta_0 = \Delta(x_0,r_0)$, $x_0 \in \partial \Omega$ and $0 < r_0 < \diam(\partial \Omega)/2$. Let $F \subset \Delta_0$ be a Borel set.
Since $\omega_{L_0}^{X_{2\,\Delta_0}}$ and $\omega_L^{X_{2\,\Delta_0}}$ are Borel regular, for each $\varepsilon > 0$, there exist compact set $K$ and an open set $U$ such that $K \subset F \subset U \subset 2\,\Delta_0$ satisfying
\begin{equation}\label{we54srbrt}
	\omega_{L_0}^{X_{2\,\Delta_0}} (U \setminus K) + \omega_{L}^{X_{2\,\Delta_0}} (U \backslash K) < \varepsilon.
\end{equation}
Using Urysohn's lemma we can construct $f_F \in \mathscr{C}_c(\partial \Omega)$ such that $\mathbf{1}_K \leq f_F \leq \mathbf{1}_U$. Then, by \eqref{AWfeve} (applied with $2\Delta_0$) and \eqref{we54srbrt} yield
\begin{multline*}
\omega_L^{X_{2\,\Delta_0}}(F)
<
\varepsilon + \omega_L^{X_{2\,\Delta_0}}(K)
\leq
\varepsilon + \int_{\partial \Omega} f_F(z) \, d \omega_L^{X_{2\,\Delta_0}}(z) \\
\le
\varepsilon + C_{\alpha, N}\|f_F\|_{L^{p}(\Delta_0, \omega_{L_0}^{X_{2\,\Delta_0}})}
\lesssim
\varepsilon +  C_{\alpha, N}\omega_{L_0}^{X_{2 \Delta_0}} (U)^{\frac1p}
<
\varepsilon +  C_{\alpha, N}(\omega_{L_0}^{X_{2 \Delta_0}} (F) + \varepsilon)^{\frac1p} .
\end{multline*}
Letting $\varepsilon \to 0+$, we obtain that $\omega_L^{X_{2\,\Delta_0}}(F) \lesssim _{\alpha, N}\omega_{L_0}^{X_{2 \Delta_0}} (F)^{\frac1p}$. Hence,
$\omega_L^{X_{2\,\Delta_0}}\ll \omega_{L_0}^{X_{2 \Delta_0}}$ in $\Delta_0$. By Harnack's inequality and the fact that we can cover $\pom$ with surface balls like $\Delta_0$ we conclude that $\omega_L \ll \omega_{L_0}$ in $\pom$. We can write $h(\cdot; L, L_0, X) = \frac{d \omega_L^{X}}{d \omega_{L_0}^{X}} \in L_\loc^1(\pom, \omega_{L_0}^{X})$ which is well-defined  $\omega_{L_0}^{X}$-a.e. in $\pom$. Thus, for every $ f\in\mathscr{C}(\pom)$ with $\supp f\subset 2\Delta_0$ we obtain from \eqref{AWfeve}
\[
\Big|\int_{2\Delta_0} f(y)\,h(y; L, L_0, X_{2\Delta_0})\,d\omega_{L_0}^{X_{2\Delta_0}}(y)\Big|
=
\Big|\int_{2\Delta_0} f(y)\,d\omega_L^{X_{2\Delta_0}}(y)\Big|
 \lesssim_{\alpha, N}\|f\|_{L^{p}(2\Delta_0, \omega_{L_0}^{X_{2\Delta_0}})}.
\]
Using the ideas in \cite[Lemma~3.38]{AHMT} and with the help of  \cite[Lemma~3.29]{AHMT}, we can then conclude that
\begin{equation*}
	\|h(\cdot\,; L, L_0, X_{2\Delta_0})\|_{L^{p'}(\Delta_0, \omega_{L_0}^{X_{2\Delta_0}})} \lesssim_{\alpha, N} 1.
\end{equation*}
This, Harnack's inequality, and the fact that $\Delta_0 = \Delta(x_0,r_0)$ with $x_0 \in \partial \Omega$ and $0 < r_0 < \diam(\partial \Omega)/2$ arbitrary easily yield that
\begin{equation*}
	\|h(\cdot\,; L, L_0, X_{\Delta(x,r)})\|_{L^{p'}(\Delta(x,r), \omega_{L_0}^{X_{\Delta(x,r)}})} \lesssim_{\alpha, N} 1,
	\qquad\text{for every $x \in \partial \Omega$ and  $0 < r < \diam(\partial \Omega)$}.
\end{equation*}
This and Remark~\ref{RHp-scale-invariant} readily imply that $\omega_L\in RH_{p'}(\pom,\omega_{L_0})$ and the proof is complete.
\qed

\subsection{Proof of \texorpdfstring{$\mathrm{(b)}_p\ \Longrightarrow\ \mathrm{(b)}_p'$}{(b)p implies (b)'p}}
Assume that $L$ is  $L^{p}(\omega_{L_0})$-solvable with $p\in (1,\infty)$. Fix $\alpha>0$, $N\ge 1$, a surface ball $\Delta_0$, and a Borel set $S\subset N\Delta_0$. Take an arbitrary $\varepsilon>0$ and since $\omega_{L_0}^{X_{\Delta_0}}$ and $\omega_L^{X_{\Delta_0}}$ are Borel regular, we can find a closed set $F$ and an open set $U$ such that
$F\subset S\subset U\subset (N+1)\Delta_0$ and
\[
\omega_{L_0}^{X_{\Delta_0}}(U\setminus F)+\omega_L^{X_{\Delta_0}}(U\setminus F)<\varepsilon.
\]
Using Urysohn's lemma we can then construct $f \in \mathscr{C}_c(\partial \Omega)$ such that $\mathbf{1}_S \leq f \leq \mathbf{1}_U$. Set,
\[
u(X):
=
\omega_L^X(S),\qquad
v(X):=
\int_{\pom} f(y)\,d\omega_{L}^X(y),
\qquad X\in\Omega.
\]
For every $M\ge c_0^{-1}$ define the truncated cone and truncated non-tangential maximal function
\[
\Gamma_{r_0, M}^{\alpha}(x)
:=
\Gamma_{r_0}^{\alpha}(x)\cap\{X\in\Omega: \delta(X)\ge r_0/M\}, \qquad
\mathcal{N}^\alpha_{r_0,M} u(x)
:=
\sup_{X \in \Gamma^\alpha_{r_0,M} (x)} |u(X)|, \qquad
x\in\pom.
\]
Note that if $x\in\Delta_0$ and $X\in \Gamma_{r_0, M}^{\alpha}(x)$ then $r_0/M\le \delta(X)\le r_0$, $c_0\,r_0\le \delta(X_{\Delta_0})\le r_0$, and $|X-X_{\Delta_0}|<2r_0$. Hence, by the Harnack chain condition and Harnack's inequality, there is a constant $C_M$ depending on $M$ such that
\[
\omega_L^X(U\setminus F)\le C_M\,\omega_L^{X_{\Delta_0}}(U\setminus F)\le C_M\,\varepsilon,
\]
and
\[
0
\le
u(X)
=
\omega_L^X(S)
\le
C_M\,\varepsilon
+
\omega_L^X(F)
\le
C_M\,\varepsilon+
\int_{\pom}f(y)\,d\omega_{L}^X(y)
=
C_M\,\varepsilon+v(X).
\]
Thus
\[
\mathcal{N}^\alpha_{r_0,M} u(x)
\le
C_M\,\varepsilon+
\mathcal{N}^\alpha_{r_0} v(x),\qquad\forall\,x\in\Delta_0.
\]
Note that our assumption is that $L^{p}(\omega_{L_0})$-solvability holds with the fixed parameters $\alpha>0$ and $N\ge 1$, but since we already know that $\mathrm{(a)} \Longleftrightarrow\ \mathrm{(b)}$ it follows that the $L^{p}(\omega_{L_0})$-solvability holds with $\alpha>0$ and $N+1$.  Thus, the fact that $f \in \mathscr{C}_c(\partial \Omega)$ with $\supp f\subset U\subset (N+1)\Delta_0$ gives
\begin{multline*}
\|\mathcal{N}^\alpha_{r_0,M} u\|_{L^p(\Delta_0, \omega_{L_0}^{X_{\Delta_0}})}
\leq
C_M\,\varepsilon\,  \omega_{L_0}^{X_{\Delta_0}}(\Delta_0)^{\frac1p}+
\|\mathcal{N}^\alpha_{r_0} v\|_{L^p(\Delta_0, \omega_{L_0}^{X_{\Delta_0}})}
\leq
C_M\,\varepsilon+
C_{\alpha,N}
\|f\|_{L^p((N+1)\Delta_0, \omega_{L_0}^{X_{\Delta_0}})}
\\
\le
C_M\,\varepsilon+
C_{\alpha,N} \omega_{L_0}^{X_{\Delta_0}}(U)^{\frac1{p}}
<
C_M\,\varepsilon+
C_{\alpha,N} (\omega_{L_0}^{X_{\Delta_0}}(S)+\varepsilon)^{\frac1{p}}
=
C_M\,\varepsilon+
C_{\alpha,N} \big(\|\mathbf{1}_S\|^p_{L^p(N\Delta_0, \omega_{L_0}^{X_{\Delta_0}})} +\varepsilon)^{\frac1{p}}.
\end{multline*}
We let $\varepsilon\to 0^+$ and obtain $\|\mathcal{N}^\alpha_{r_0,M} u\|_{L^p(\Delta_0, \omega_{L_0}^{X_{\Delta_0}})} \le C_{\alpha,N} \|\mathbf{1}_S\|_{L^p(N\Delta_0, \omega_{L_0}^{X_{\Delta_0}})}$. Since $\mathcal{N}^\alpha_{r_0,M} u(x)\nearrow \mathcal{N}^\alpha_{r_0} u(x)$ for every $x\in\pom$ as $M\to\infty$ we conclude the desired estimate by simply applying the monotone convergence theorem. \qed

\subsection{Proof of \texorpdfstring{$\mathrm{(b)}'\ \Longrightarrow\ \mathrm{(a)}$}{(b)' implies (a)}}
Fix $p\in (1,\infty)$ and assume that $L$ is $L^{p}(\omega_{L_0})$-solvable for characteristic functions. That is for some $\alpha>0$ and some $N\ge 1$ there exists $C_{\alpha, N}\ge 1$ (depending only on $n$, the 1-sided NTA constants, the CDC constant, the ellipticity of $L_0$ and $L$, $\alpha$, $N$, and $p$) such that \eqref{solv-Lp} holds for $u$ as in \eqref{u-elliptic:L-sol} for any $f=\mathbf{1}_S$ with $S$ being a Borel set $S\subset N\Delta_0$.

Take an arbitrary $\Delta_0 = \Delta(x_0,r_0)$, $x_0 \in \partial \Omega$ and $0 < r_0 < \diam(\partial \Omega)$. We follow the proof of $\mathrm{(b)}_p\ \Longrightarrow\ \mathrm{(a)}_{p'}$ and observe that the same argument we used to obtain \eqref{AWfeve} easily gives, taking $f=\mathbf{1}_S$ with $S$ being a Borel set $S\subset N \Delta_0$, that
\begin{equation}\label{AWfeve:char}
\omega_L^{X_{\Delta_0}}(S)=	\int_{\Delta_0} \mathbf{1}_S(y)\,d\omega_L^{X_{\Delta_0}}(y)
\lesssim_{\alpha, N}
\|\mathbf{1}_S\|_{L^{p}(\Delta_0, \omega_{L_0}^{X_{\Delta_0}})}
=
\omega_{L_0}^{X_{\Delta_0}}(S)^\frac1p.
\end{equation}
This readily implies that $\omega_L^{X_{\Delta_0}}\ll \omega_{L_0}^{X_{\Delta_0}}$ in $\Delta_0$, and since $\Delta_0$ is arbitrary we conclude that $\omega_L\ll \omega_{L_0}$ in $\pom$. To proceed, fix  $B_0=B(x_0,r_0)$ and $B=B(x,r)$ with $B\subset B_0$, $x_0,x \in \partial \Omega$ and $0 < r_0,r < \diam(\pom)$. Write $\Delta_0=B_0\cap\pom$ and $\Delta=B\cap\pom$. Let $S\subset \Delta$ be an arbitrary Borel set. If $r\approx r_0$ we have by Harnack's inequality and Lemma~\ref{lemma:proppde} part $(a)$
\begin{align*}
\frac{\omega_L^{X_{\Delta_0}}(S)}{\omega_L^{X_{\Delta_0}}(\Delta)}	
\approx
\frac{\omega_L^{X_{\Delta}}(S)}{\omega_L^{X_{\Delta}}(\Delta)}
\approx
\omega_L^{X_{\Delta}}(S)
\lesssim_{\alpha, N}
\omega_{L_0}^{X_{\Delta}}(S)^{\frac1p}
\approx
\Big(\frac{\omega_{L_0}^{X_{\Delta}}(S)}{\omega_{L_0}^{X_{\Delta}}(\Delta)}\Big)^{\frac1p}
\approx
\Big(\frac{\omega_{L_0}^{X_{\Delta_0}}(S)}{\omega_{L_0}^{X_{\Delta_0}}(\Delta)}\Big)^{\frac1p},	
\end{align*}
where in the third estimate we have used \eqref{AWfeve:char} with $\Delta$ in place of $\Delta_0$. On the other hand, if $r\ll r_0$ we have by Lemma~\ref{lemma:proppde} part $(d)$ that $\omega_L\ll\omega_{L_0}$ with
\begin{align*}
\frac{\omega_L^{X_{\Delta_0}}(S)}{\omega_L^{X_{\Delta_0}}(\Delta)}	
\approx
	\omega_L^{X_{\Delta}}(S)
	\lesssim_{\alpha, N}
	\omega_{L_0}^{X_{\Delta}}(S)^{\frac1p}
	\approx
	\Big(\frac{\omega_{L_0}^{X_{\Delta_0}}(S)}{\omega_{L_0}^{X_{\Delta_0}}(\Delta)}\Big)^{\frac1p},	
\end{align*}
where again we have used \eqref{AWfeve:char} with $\Delta$ in place of $\Delta_0$ in the middle estimate. In short we have proved that
\[
\frac{\omega_L^{X_{\Delta_0}}(S)}{\omega_L^{X_{\Delta_0}}(\Delta)}	\lesssim_{\alpha, N} \Big(\frac{\omega_{L_0}^{X_{\Delta_0}}(S)}{\omega_{L_0}^{X_{\Delta_0}}(\Delta)}\Big)^{\frac1p},
\qquad \text{for any Borel set $S\subset\Delta$}.
\]
Using the fact that the implicit constants do not depend on $\Delta$ (nor on $\Delta_0$) and Lemma~\ref{lemma:proppde} part $(c)$, this readily implies that $\omega_L^{X_{\Delta_0}}\in RH_q(\Delta_0, \omega_{L_0}^{X_{\Delta_0}})$ for some $q\in (1,\infty)$ where $q$ and the implicit constants do not depend on $\Delta_0$, see \cite{CF-1974,GR}. Hence, we readily conclude that $\omega_L\in RH_q(\pom, \omega_{L_0})$ (see Definition~\ref{def:Ainfty}). This completes the proof of the present implication. \qed

\subsection{Proof of \texorpdfstring{$\mathrm{(a)}\ \Longrightarrow\ \mathrm{(d)}$}{(a) implies (d)}}
Assume that $\omega_L\in A_\infty(\pom,\omega_{L_0})$. By the classical theory of weights (cf.~\cite{CF-1974, GR}) and Lemma~\ref{lemma:proppde} part $(c)$ it is not hard to see that $\omega_{L_0}\in A_\infty(\pom,\omega_{L})$, hence $\omega_{L_0}\in RH_{p}(\pom,\omega_{L})$ for some $1<p<\infty$. In particular for every $Q_0\in\dd(\pom)$ and $Q\in\dd_{Q_0}$,
by Lemma~\ref{lemma:proppde} part $(c)$ we have
\begin{align*}
	\left(\fint_{Q}h(y;L_0, L,X_{Q_0} )^p d \omega_{L}^{X_{Q_0}}(y)\right)^{\frac1p}
	\leq
	C
	\fint_{Q} h(y;L_0, L,X_{Q_0} ) d \omega_{L}^{X_{Q_0}}(y)
	=
	C\frac{\omega_{L_0}^{X_{Q_0}}(Q)}{\omega_{L}^{X_{Q_0}}(Q)}.
\end{align*}
Thus, for $F\subset Q$ we obtain, by H\"older's inequality,
\begin{multline}\label{341frqfa}
\frac{\omega_{L_0}^{X_{Q_0}}(F)}{\omega_{L_0}^{X_{Q_0}}(Q)}
=
\fint_{Q} \mathbf{1}_F(y) d\omega_{L_0}^{X_{Q_0}}(y)
=
\frac{\omega_{L}^{X_{Q_0}}(Q)}{\omega_{L_0}^{X_{Q_0}}(Q)} \fint_{Q} \mathbf{1}_F(y)  h(y;L_0, L,X_{Q_0} )d \omega_{L}^{X_{Q_0}}(y)
\\
\le
\frac{\omega_{L}^{X_{Q_0}}(Q)}{\omega_{L_0}^{X_{Q_0}}(Q)}
\Big(\fint_{Q} h(y;L_0, L,X_{Q_0}) ^pd \omega_{L}^{X_{Q_0}}(y)\Big)^\frac1p
\Big(\frac{\omega_{L}^{X_{Q_0}}(F)}{\omega_{L}^{X_{Q_0}}(Q)}\Big)^\frac1{p'}
\lesssim
\Big(\frac{\omega_{L}^{X_{Q_0}}(F)}{\omega_{L}^{X_{Q_0}}(Q)}\Big)^\frac1{p'}.
\end{multline}

To continue we need a dyadic version of \eqref{S<N}: for every $Q_0\in\dd(\pom)$ and for every $\vartheta\ge \vartheta_0$ we claim that
\begin{equation}\label{S<N:dyadic}
	\|\mathcal{S}^\vartheta_{Q_0} u\|_{L^q(Q_0, \omega_{L_0}^{X_{Q_0}})}
	\leq
	C_\vartheta \|\mathcal{N}^\vartheta_{Q_0} u\|_{L^q(Q_0, \omega_{L_0}^{X_{Q_0}})}, \quad 0<q<\infty.
\end{equation}
This estimate can be proved following the argument in \cite[Section~5.2]{AHMT} with the following changes. Recall \cite[(5.9)]{AHMT} (here we note that the argument in \cite[Section~5.2]{AHMT} was done with a fixed value of $\vartheta$ sufficiently large, but it is routine to see that one can repeat it with this parameter with harmless changes)
\begin{align}\label{good-lambda:goal:local}
	\omega_L^{X_{Q_0}}\big(\big\{x\in Q_j: \mathcal{S}_{Q_j}^{\vartheta, k_0}u(x)>\beta\,\lambda,\ \mathcal{N}_{Q_0}^\vartheta u(x)\le \gamma\,\lambda \big\}\big)
	\lesssim
	\Big(\frac{\gamma}{\beta}\Big)^\vartheta\, \omega_L^{X_{Q_0}}(Q_j),
\end{align}
where $\lambda$, $\beta$, $\gamma > 0$; $Q_j$ is some dyadic cube (see \cite[Section~5.2]{AHMT});
$\mathcal{S}_{Q_j}^{\vartheta, k_0}u$ is a truncated localized  dyadic conical square function with respect to the cones
\[
\Gamma_{Q_j}^{\vartheta, k_0}(x) := \bigcup_{\substack{x\in Q'\in\dd_{Q} \\ \ell(Q')\ge 2^{-k_0}\,\ell(Q_0)}}  U_{Q'}^\vartheta;
\]
and $k_0$ is large enough (eventually $k_0\to\infty$). It should be noted that the implicit constant in the inequality \eqref{good-lambda:goal:local} does not depend on $k_0$. Combining \eqref{good-lambda:goal:local} with \eqref{341frqfa} we easily arrive at
\begin{align}\label{good-lambda:goal:local:omega0}
	\omega_{L_0}^{X_{Q_0}}\big(\big\{x\in Q_j: \mathcal{S}_{Q_j}^{\vartheta, k_0}u(x)>\beta\,\lambda,\ \mathcal{N}_{Q_0}^\vartheta u(x)\le \gamma\,\lambda \big\}\big)
	\lesssim
	\Big(\frac{\gamma}{\beta}\Big)^{\frac{\vartheta}{p'}}\, \omega_{L_0}^{X_{Q_0}}(Q_j).
\end{align}
From this we can derive \cite[(5.7)]{AHMT} with $\omega_{L_0}^{X_{Q_0}}$ in place of $\omega_{L}^{X_{Q_0}}$ and a typical good-$\lambda$ argument much as in \cite[Section~5.2]{AHMT} readily leads to \eqref{S<N:dyadic}.

With \eqref{S<N:dyadic} at our disposal we can then proceed to obtain \eqref{S<N}. Fix $\Delta_0 = \Delta(x_0, r_0)$ with $x_0 \in \partial \Omega, 0 < r_0 < \diam (\partial \Omega)$. Let $M\ge 1$ be large enough to be chosen and set
\[
\mathcal{F}_{\Delta_0}
:=
\big\{
Q\in\dd(\pom): r_0/(2 M)\le \ell(Q)<r_0/M, Q\cap\Delta_0\neq \emptyset
\big\}.
\]
One has that $\mathcal{F}_{\Delta_0}$ is a pairwise disjoint family and
\[
\Delta_0
\subset
\bigcup_{Q\in \mathcal{F}_{\Delta_0}} Q
\subset
\tfrac54\Delta_0,
\]
provided $M$ is large enough.

Write $\widetilde{r}_0:=r_0/2M$. Let $x\in Q_0\in \mathcal{F}_{\Delta_0}$ and $X\in\Gamma_{\widetilde{r}_0}^\alpha(x)$.
Let $I_X\in\W$ be so that $I_X\ni X$ and pick $Q_X\in\dd(\pom)$ with $x\in Q_X$ and $\ell(Q_X)=\ell(I_X)$. Note that
\[
\ell(Q_X)
=
\ell(I_X)
\le
\diam(I_X)
\le
\dist(I_X,\pom)
\le
\delta(X)
\le
|X-x|
<
\widetilde{r_0}
=
\frac{r_0}{2M}
\le
\ell(Q_0).
\]
This and the fact that $x\in Q_0\cap Q_X$ gives $Q_X\subset Q_0$. On the other hand,
\begin{multline*}
	\dist(I_X,Q_X)
\le
|X-x|
\le
(1+\alpha)\delta(X)
\le
(1+\alpha)(\diam(I_X)+\dist(I_X,\pom))
\\
\le
41\sqrt{n+1}(1+\alpha)\ell(I_X)
=
41\sqrt{n+1}(1+\alpha)\ell(Q_X).
\end{multline*}
This shows that if we fix $\vartheta=\vartheta(\alpha)$ so that $2^\vartheta\ge 41\sqrt{n+1}(1+\alpha)$ then $I_X\in\W^\vartheta_{Q_X}\subset \W^{\vartheta,*}_{Q_X}$. As a result,
$X\in I_X\subset U_{Q_X}^{\vartheta}$ and $X\in \Gamma_{Q_0}^\vartheta(x)$. All these show that for every $Q_0\in \mathcal{F}_{\Delta_0}$ and  $x\in Q_0\in \mathcal{F}_{\Delta_0}$ we have $\Gamma_{\widetilde{r}_0}^\alpha(x)\subset \Gamma_{Q_0}^\vartheta(x)$. Thus \eqref{S<N:dyadic} yields
\begin{multline*}
\|\mathcal{S}^\alpha_{\widetilde{r}_0} u\|_{L^q(\Delta_0, \omega_{L_0}^{X_{\Delta_0}})}^q
\le
\sum_{Q_0\in \mathcal{F}_{\Delta_0}} \int_{Q_0} \mathcal{S}^\alpha_{\widetilde{r}_0} u(x)^q\,d\omega_{L_0}^{X_{\Delta_0}}(x)
\\
\le
\sum_{Q_0\in \mathcal{F}_{\Delta_0}} \int_{Q_0} \mathcal{S}^\vartheta_{Q_0} u(x)^q\,d\omega_{L_0}^{X_{\Delta_0}}(x)
\lesssim_\alpha
\sum_{Q_0\in \mathcal{F}_{\Delta_0}} \int_{Q_0} \mathcal{N}^\vartheta_{Q_0} u(x)^q\,d\omega_{L_0}^{X_{\Delta_0}}(x).
\end{multline*}
To continue let $Q_0\in\F_{\Delta_0}$, $x\in Q_0$ and $X\in \Gamma_{Q_0}^{\vartheta, *}(x)$. Then $X\in I^{**}$ with $I\in \W_{Q}^{\vartheta,*}$ and $x\in Q\subset Q_0$. As a consequence,
\[
|X-x|
\le
\diam(I^{**})+\dist(I,Q_0)+\diam(Q_0)
\lesssim_{\vartheta}
\ell(I)
\approx
\delta(X)
\le
\kappa_0 \ell(Q_0)
<
2\kappa_0 \widetilde{r}_0
\]
where we have used  \eqref{definicionkappa12}  and the last estimate holds provided $M$ is large enough. This shows that $X\in\Gamma^{\alpha'}_{2\kappa_0 \widetilde{r}_0}(x)$ for some $\alpha'=\alpha'(\vartheta)$ (hence depending on $\alpha$). As a consequence of these, we obtain
\begin{multline*}
	\sum_{Q_0\in \mathcal{F}_{\Delta_0}} \int_{Q_0} \mathcal{N}^\vartheta_{Q_0} u(x)^q\,d\omega_{L_0}^{X_{\Delta_0}}(x)
\le
\int_{\frac54\Delta_0} \mathcal{N}^{\alpha'}_{2\kappa_0 \widetilde{r}_0} u(x)^q\,d\omega_{L_0}^{X_{\Delta_0}}(x)
\\
\lesssim_\alpha
\int_{5\Delta_0} \mathcal{N}^{\alpha}_{8\kappa_0 \widetilde{r}_0} u(x)^q\,d\omega_{L_0}^{X_{\Delta_0}}(x)
\le
\int_{5\Delta_0} \mathcal{N}^{\alpha}_{r_0} u(x)^q\,d\omega_{L_0}^{X_{\Delta_0}}(x),
\end{multline*}
where we have used \eqref{eqn:ChAO:N} and the last estimate follows provided $M$ is large enough. \qed

\subsection{Proof of \texorpdfstring{$\mathrm{(d)}\ \Longrightarrow\ \mathrm{(d)}'$}{(d) implies (d)'}}
This is trivial since for any arbitrary Borel set $S\subset\pom$, the solution $u(X)=\omega_L^X(S)$, $X\in\Omega$, belongs to $u\in W^{1,2}_{\loc}(\Omega)$. \qed

\subsection{Proof of \texorpdfstring{$\mathrm{(d)}'\ \Longrightarrow\ \mathrm{(a)}$}{(d)' implies (a)}}
Assume that  \eqref{S<N} holds for some fixed $\alpha_0$ and $q\in (0,\infty)$ and for $u(X)=\omega_L^X(S)$, $X\in\Omega$, for any arbitrary Borel set $S\subset\pom$.
By Lemma~\ref{lemma-ChAO} (applied to $F(X)=|\nabla u(X)|\delta(X)^{(1-n)/2}$), for any $\alpha$ large enough to be chosen we have
\begin{equation}\label{S<N:proof}
	\|\mathcal{S}^\alpha_{r_0} u\|_{L^q(\Delta_0, \omega_{L_0}^{X_{\Delta_0}})}
	\lesssim_{\alpha,\alpha_0}
	\|\mathcal{S}^{\alpha_0}_{3r_0} u\|_{L^q(3\Delta_0, \omega_{L_0}^{X_{3\Delta_0}})}
	\lesssim_{\alpha_0}
	\omega_{L_0}^{X_{15\Delta_0}}(15\Delta_0)^{\frac1q}
	\approx
	\omega_{L_0}^{X_{\Delta_0}}(\Delta_0)^{\frac1q}
	,
\end{equation}	
for every $\Delta_0 = \Delta(x_0, r_0)$ with $x_0 \in \partial \Omega, 0 < r_0< \diam (\partial \Omega)/3$, and where we have used that $0\le u\le 1$. Let us see how to extend the previous estimate, in the case $\pom$ is bounded, to any  $\diam(\pom)/3\le r_0<\diam(\pom)$. Note that if $x\in\Delta_0$ and $X\in\Gamma^\alpha_{\diam (\partial \Omega)}(x)\setminus \Gamma^\alpha_{\diam (\partial \Omega)/4}(x)$, then
\[
\frac14\diam (\partial \Omega) \le |X-x|\le (1+\alpha)\delta(X) \le (1+\alpha) |X-x|<(1+\alpha)\diam (\partial \Omega).
\]
Set $\W_{x}=\{I\in\W: I\cap (\Gamma^\alpha_{\diam (\partial \Omega)}(x)\setminus \Gamma^\alpha_{\diam (\partial \Omega)/4}(x))\neq\emptyset\}$, whose cardinality is uniformly bounded (depending in dimension and $\alpha$). Thus, since $\|u\|_{L^\infty(\Omega)}\le 1$, Caccioppoli's inequality gives
\begin{multline*}
\iint_{\Gamma^\alpha_{\diam (\partial \Omega)}(x)\setminus \Gamma^\alpha_{\diam (\partial \Omega)/4}(x)} |\nabla u(X)|^2 \delta(X)^{1-n}\,dX
\lesssim
\sum_{I\in \W_{x}} \ell(I)^{1-n} \iint_I |\nabla u(X)|^2\,dX
\\
\lesssim
\sum_{I\in \W_{x}} \ell(I)^{-1-n} \iint_{I^*} |u(X)|^2\,dX
\lesssim
\# \W_{x}
\lesssim_\alpha 1.
\end{multline*}
With this in hand and \eqref{S<N:proof}  applied with $r_0=\diam(\pom)/4<\diam(\pom)/3$, we readily obtain
\begin{multline*}
\|\mathcal{S}^\alpha_{r_0} u\|_{L^q(\Delta_0, \omega_{L_0}^{X_{\Delta_0}})}
\le
\|\mathcal{S}^\alpha_{\diam(\pom)} u\|_{L^q(\Delta_0, \omega_{L_0}^{X_{\Delta_0}})}
\\
\le
\|\mathcal{S}^\alpha_{\diam(\pom)} u-\mathcal{S}^\alpha_{\diam(\pom)/4} u\|_{L^q(\Delta_0, \omega_{L_0}^{X_{\Delta_0}})}
+
\|\mathcal{S}^\alpha_{\diam(\pom)/4} u\|_{L^q(\Delta_0, \omega_{L_0}^{X_{\Delta_0}})}
\lesssim
\omega_{L_0}^{X_{\Delta_0}}(\Delta_0)^{\frac1q}.
\end{multline*}

We next see that given $\gamma\in (0,1)$ there exists $\beta\in (0,1)$ so that for every $Q_0\in\dd(\pom)$ and for every Borel set $F\subset Q_0$ we have
\begin{equation}\label{Ainfty-CF}
	\frac{\omega_{L}^{X_{Q_0}}(F)}{\omega_{L}^{X_{Q_0}}(Q_0)}\le \beta
	\quad\Longrightarrow\quad \frac{\omega_{L_0}^{X_{Q_0}}(F)}{\omega_{L_0}^{X_{Q_0}}(Q_0)}\le \gamma.
\end{equation}
Indeed, fix $\gamma\in (0,1)$ and $Q_0\in\mathbb{D}(\pom)$, and take a Borel set $F\subset Q_0$ so that
${\omega_L^{X_{Q_0}}(F)}\le \beta\omega_L^{X_{Q_0}}(Q_0) $, where $\beta\in (0,1)$ is small enough to be chosen. Applying Lemma \ref{sq-function->M}, if we assume that $0<\beta<\beta_0$, then  $u(X)=\omega^X_L(S)$ satisfies \eqref{def-mod-sq-function} and therefore
\begin{align}\label{endproof1}
	C_\eta^{-q} \log{(\beta^{-1})}^{\frac{q}{2}}\omega_{L_0}^{X_{Q_0}}(F)
	\leq
	\int_F \mathcal{S}_{Q_0,\eta}^{\vartheta_0}u(x)^q\,d \omega_{L_0}^{X_{Q_0}} (x)
	\leq
	\int_{Q_0} \mathcal{S}_{Q_0,\eta}^{\vartheta_0}u(x)^q\,d \omega_{L_0}^{X_{Q_0}} (x).
\end{align}

We claim that there exists $\alpha_0=\alpha_0(\vartheta_0, \eta)$ (hence, depending on the allowable parameters) such that
\begin{equation}\label{daydic-cone:regular-cone}
\Gamma_{Q_0,\eta}^{\vartheta_0}(x) \subset \Gamma^{\alpha_0}_{r_{Q_0}^*}(x), \qquad x\in Q_0,
\end{equation}
with $r_{Q_0}^*=2 \kappa_0 r_{Q_0}$ (cf.~\eqref{definicionkappa12}). To see this, let $x\in Q_0$ and $X\in \Gamma_{Q_0,\eta}^{\vartheta_0}(x)$. Then $X\in I^*$ for some $I\in\W_{Q'}^{\vartheta_0,*}$, where $Q'\subset Q\in\dd_{Q_0}$ with $Q\ni x$ and $\ell(Q')>\eta^3\ell(Q)$. Then $X\in T_{Q_0}^{\vartheta_0,*}\subset B_{Q_0}^*\cap\Omega$ (see \eqref{definicionkappa12})
and
\[
|X-x|\le |X-x_{Q_0}|+|x_{Q_0}-x|
<
\kappa_0 r_{Q_0}+\Xi r_{Q_0}
\le
2\kappa_0 r_{Q_0}:=r_{Q_0}^*,
\]
and also
\[
|X-x|
\le
\diam(I^*)+\dist(I,Q')+\diam(Q)
\lesssim_{\vartheta_0, \eta}
\ell(I)
\approx
\delta(X).
\]
Hence, there exists $\alpha_0=\alpha_0(\vartheta_0, \eta)$  such that $X\in\Gamma^{\alpha_0}_{r_{Q_0}^*}(x)$, that is,  \eqref{daydic-cone:regular-cone} holds.

To continue, observe first that by \eqref{deltaQ} and the fact that $\kappa_0\ge 16 \Xi$ (cf.~\eqref{definicionkappa12}) we have $Q_0\subset \Delta_{Q_0}^*$. This, \eqref{daydic-cone:regular-cone}, Harnack's inequality, \eqref{S<N:proof}, and Lemma \ref{lemma:proppde} imply
\begin{multline}\label{34wg4sg4s}
	\int_{Q_0} \mathcal{S}_{Q_0,\eta}^{\vartheta_0}u(x)^q\,d \omega_{L_0}^{X_{Q_0}} (x)
	\lesssim
	\int_{\Delta_{Q_0}^*} \mathcal{S}_{r_{Q_0}^*}^{\alpha}u(x)^q\,d \omega_{L_0}^{X_{Q_0}} (x)
	\\
	\approx
	\int_{\Delta_{Q_0}^*} \mathcal{S}_{r_{Q_0}^*}^{\alpha}u(x)^q\,d \omega_{L_0}^{X_{\Delta_{Q_0}^*}} (x)
	\lesssim_\alpha
	\omega_{L_0}^{X_{\Delta_{Q_0}^*}} (2\Delta_{Q_0}^*)
	\approx
	\omega_{L_0}^{X_{Q_0}}(Q_0).
\end{multline}
Combining \eqref{endproof1} and \eqref{34wg4sg4s} we conclude that
\[
\frac{\omega_{L_0}^{X_{Q_0}}(F)}{\omega_{L_0}^{X_{Q_0}}(Q_0)}
\le
C_{\eta,\vartheta_0,q} \, \log{(\beta^{-1})}^{-\frac{q}{2}}.
\]
This readily gives \eqref{Ainfty-CF} by choosing $\beta$ small enough so that $C_{\eta,\vartheta_0,q} \, \log{(\beta^{-1})}^{-\frac{q}{2}}<\gamma$.

Next, we show that \eqref{Ainfty-CF} implies $\omega_{L}\in A_\infty(\partial\Omega, \omega_{L_0})$. To see this we first obtain a dyadic-$A_\infty$ condition. Fix $Q^0, Q_0\in\mathbb{D}$ with $Q_0\subset Q^0$. Remark~\ref{remark:chop-dyadic} gives for every $F\subset Q_0$
\begin{equation}\label{cop-dyadic}
	\frac{1}{C_1}\frac{\omega_{L_0}^{X_{Q_0}}(F)}{\omega_{L_0}^{X_{Q_0}}(Q_0)}
	\leq
	\frac{\omega_{L_0}^{X_{Q^0}}(F)}{\omega_{L_0}^{X_{Q^0}}(Q_0)}
	\leq
	C_1\frac{\omega_{L_0}^{X_{Q_0}}(F)}{\omega_{L_0}^{X_{Q_0}}(Q_0)}
	\ \ \ \text{and}\ \ \
	\frac{1}{C_1}\frac{\omega_L^{X_{Q_0}}(F)}{\omega_L^{X_{Q_0}}(Q_0)}
	\leq
	\frac{\omega_L^{X_{Q^0}}(F)}{\omega_L^{X_{Q^0}}(Q_0)}
	\leq
	C_1\frac{\omega_L^{X_{Q_0}}(F)}{\omega_L^{X_{Q_0}}(Q_0)},
\end{equation}
for some $C_1>1$.
Thus, given $\gamma\in (0,1)$, take the corresponding $\beta\in (0,1)$ so that \eqref{Ainfty-CF} holds with $\gamma/C_1$ in place of $\gamma$. Then, \begin{equation}\label{Ainfty-daydic}
	\frac{\omega_L^{X_{Q^0}}(F)}{\omega_L^{X_{Q^0}}(Q_0)}\le \frac{\beta}{C_1}
	\implies
	\frac{\omega_L^{X_{Q_0}}(F)}{\omega_L^{X_{Q_0}}(Q_0)}\le \beta
	\implies
	\frac{\omega_{L_0}^{X_{Q_0}}(F)}{\omega_{L_0}^{X_{Q_0}}(Q_0)}\le \frac{\gamma}{C_1}
	\implies
	\frac{\omega_{L_0}^{X_{Q^0}}(F)}{\omega_{L_0}^{X_{Q^0}}(Q_0)}\le \gamma.
\end{equation}

To complete the proof we need to see that \eqref{Ainfty-daydic} gives $\omega_{L}\in A_\infty(\partial\Omega, \omega_{L_0})$. Fix $\gamma\in (0,1)$ and  a surface ball $\Delta_0=B_0\cap\partial\Omega$, with $B_0=B(x_0,r_0)$, $x_0\in\partial\Omega$, and $0<r_0<\diam(\partial\Omega)$. Take an arbitrary surface ball $\Delta=B\cap\partial\Omega$ centered at $\partial\Omega$ with $B=B(x,r)\subset B_0$, and let $F\subset\Delta$ be a Borel set such that $\omega_{L_0}^{X_{\Delta_0}}(F)>\gamma \omega_{L_0}^{X_{\Delta_0}}(\Delta)$. Consider the pairwise disjoint family $\mathcal{F}=\{Q\in \mathbb{D}: Q\cap \Delta\neq\emptyset, \frac{r}{4\,\Xi}<\ell(Q)\le \frac{r}{2\,\Xi}\}$ where $\Xi$ is the constant in \eqref{deltaQ}. In particular,  $\Delta\subset \bigcup_{Q\in \mathcal{F}} Q\subset 2\Delta$. The pigeon-hole principle yields that there is a constant $C'>1$ depending just on the doubling constant of $\omega_{L_0}^{X_{\Delta_0}}$ so that ${\omega_{L_0}^{X_{\Delta_0}}(F\cap Q_0)}/{\omega_{L_0}^{X_{\Delta_0}}(Q_0)}>{\gamma}/{C'}$ for some $Q_0\in\mathcal{F}$. Let $Q^0\in\mathbb{D}$ be the unique dyadic cube such that $Q_0\subset Q^0$ and  $\frac{r_0}{2}<\ell(Q^0)\le r_0$. We can then invoke the contrapositive of \eqref{Ainfty-daydic} with ${\gamma}/{C'}$ in place of $\gamma$   to find $\beta\in (0,1)$ such that by Lemma \ref{lemma:proppde}, and Harnack's inequality we arrive at
\[
\frac{\omega_L^{X_{\Delta_0}}(F)}{\omega_L^{X_{\Delta_0}}(\Delta)}
\ge
\frac{\omega_L^{X_{\Delta_0}}(F\cap Q_0)}{\omega_L^{X_{\Delta_0}}(\Delta)}
\approx
\frac{\omega_L^{X_{\Delta_0}}(F\cap Q_0)}{\omega_L^{X_{\Delta_0}}(Q_0)}
\approx
\frac{\omega_L^{X_{Q^0}}(F\cap Q_0)}{\omega_L^{X_{Q^0}}(Q_0)}
>\frac{\beta}{C_1}.
\]
In short, we have obtained that for every $\gamma\in (0,1)$ there exists $\widetilde{\beta}\in (0,1)$ such that
\[
\frac{\omega_{L_0}^{X_{\Delta_0}}(F)}{\omega_{L_0}^{X_{\Delta_0}}(\Delta)}>\gamma
\implies
\frac{\omega_L^{X_{\Delta_0}}(F)}{\omega_L^{X_{\Delta_0}}(\Delta)}>\widetilde{\beta}.
\]
This and the classical theory of weights (cf.~\cite{CF-1974, GR}) show that $\omega_{L}\in A_\infty(\partial\Omega, \omega_{L_0})$, and the proof is complete. \qed

\subsection{Proof of \texorpdfstring{$\mathrm{(c)}\ \Longrightarrow\ \mathrm{(c)}'$}{(c) implies (c)'}}
This is trivial since for any arbitrary Borel set $S\subset\pom$, the solution $u(X)=\omega_L^X(S)$, $X\in\Omega$, belongs to $W^{1,2}_{\loc}(\Omega)\cap L^\infty(\Omega)$. \qed

\subsection{Proof of \texorpdfstring{$\mathrm{(e)} \ \Longrightarrow\ \mathrm{(f)}$}{(e) implies (f)}}
Let $\Delta_\varepsilon=B_\varepsilon\cap\pom$, $\Delta'=B'\cap\pom$,  where $B_\varepsilon=B(x,\varepsilon r)$ with $x\in \pom$ and $0<r<\diam(\pom)$,
and $B'=B(x',r')$ with $x'\in2\Delta_\varepsilon$ and $0<r'<\varepsilon r c_0/4$, and $c_0$ is the Corkscrew constant. Using \eqref{CME:BMO} and Lemma~\ref{lemma:CME-uOmega} we easily obtain
\begin{align*}
& \frac{1}{\omega_{L_0}^{X_{\Delta_\varepsilon}} (\Delta')} \iint_{B'\cap \Omega} |\nabla u(X)|^2 G_{L_0}(X_{\Delta_\varepsilon},X) \, d X
 \\
&\qquad\lesssim
\|f\|^2_{\mathrm{BMO}(\partial \Omega, \omega_{L_0})}
+|f_{\Delta, L_0}|^2
\frac{1}{\omega_{L_0}^{X_{\Delta_\varepsilon}} (\Delta')} \iint_{B'\cap \Omega} |\nabla u_{L,\Omega}(X)|^2 G_{L_0}(X_{\Delta_\varepsilon},X) \, d X
\\
&\qquad \lesssim \|f\|^2_{\mathrm{BMO}(\partial \Omega, \omega_{L_0})}+\|f\|^2_{L^\infty(\partial \Omega, \omega_{L_0})} \Big(\frac{r'}{\diam(\pom)}\Big)^{2\rho}
\\
&\qquad \lesssim \|f\|^2_{\mathrm{BMO}(\partial \Omega, \omega_{L_0})}+\|f\|^2_{L^\infty(\partial \Omega, \omega_{L_0})}  \varepsilon^{2\rho}.
\end{align*}
Taking the sup over $B_\varepsilon$ and $B'$ we readily arrive at \eqref{CME:BMO-Linfty}.
\qed

\subsection{Proof of \texorpdfstring{$\mathrm{(f)} \ \Longrightarrow\ \mathrm{(c)}'$}{(f)  implies (c)'}}
We first observe that (f) applied with $\varepsilon=1$ gives
	\begin{multline}\label{Qwerfqeavegv:1}
	\sup_{B} \sup_{B'} \frac{1}{\omega_{L_0}^{X_{\Delta}} (\Delta')} \iint_{B'\cap \Omega} |\nabla u(X)|^2 G_{L_0}(X_{\Delta},X) \, d X
	\\
	\le C \big(\|f\|^2_{\mathrm{BMO}(\partial \Omega, \omega_{L_0})} + \varrho(1)\|f\|_{L^\infty(\pom, \omega_{L_0})}^2\big)
		\lesssim
	\|f\|_{L^\infty(\pom, \omega_{L_0})}^2
	,
\end{multline}
where $\Delta=B\cap\pom$, $\Delta'=B'\cap\pom$,  and the sups are taken respectively over all balls $B=B(x,r)$ with $x\in \pom$ and $0<r<\diam(\pom)$,
and $B'=B(x',r')$ with $x'\in2\Delta$ and $0<r'<r c_0/4$, and $c_0$ is the Corkscrew constant.

With this in place we are now ready to establish (c)$'$.
Take an arbitrary Borel set $S\subset\pom$ and let $u(X)=\omega_L^X(S)$, $X\in\Omega$. Fix $X_0\in\Omega$ and use that $\omega_L^{X_0}$ is Borel regular to see that for every $j\ge 1$ there exist a closed set $F_j$ and an open set $U_j$ so that $F_j\subset S\subset U_j$ and $\omega_L^{X_0}(U_j\setminus F_j)<j^{-1}$. Using Urysohn's lemma we can construct $f_j \in \mathscr{C}(\partial \Omega)$ such that $\mathbf{1}_{F_j} \leq f_j \leq \mathbf{1}_{U_j}$, and for $X\in\Omega$ set
\[
v_j(X):=\int_{\pom}f_j(y)d\omega_L^X(y).
\]
It is straightforward to see that $|\mathbf{1}_S(x)-f_j(x)|\le \mathbf{1}_{U_j\setminus F_j}(x)$ for every $x\in\pom$, hence for every compact set $K\subset\Omega$ and for every $X\in K$ we have by Harnack's inequality
\[
|u(X)-v_j(X)|
\le
\int_{\pom} |\mathbf{1}_S(x)-f_j(x)| \,d\omega_L^{X}(x)
\le
\omega_L^{X}(U_j\setminus F_j)
\le
C_{K,X_0} \omega_L^{X_0}(U_j\setminus F_j)
< C_{K,X_0} j^{-1}.
\]
Thus $v_j\longrightarrow u$ uniformly on compacta in $\Omega$. This together with Caccioppoli's inequality readily imply that $\nabla v_j \longrightarrow\nabla u$ in $L^2_{\loc }(\Omega)$. In particular,  $\nabla v_j \longrightarrow\nabla u$ in $L^2(K)$ for every compact set $K\subset\Omega$.

Fix $\Delta=B\cap\pom$, $\Delta'=B'\cap\pom$,  with $x\in \pom$ and $0<r<\diam(\pom)$, and $B'=B(x',r')$ with $x'\in2\Delta$ and $0<r'<r c_0/4$, and $c_0$ is the Corkscrew constant. Let $f_{j,\Delta, L_0}:=	\fint_{\Delta} f_j\,d\omega_{L_0}^{X_\Delta}$ and $u_{L,\Omega}(X):=\omega_L^X(\pom)$, $X\in\Omega$. For every compact set $K\subset\Omega$ we then have by  \eqref{Qwerfqeavegv:1}  applied to each $f_j$
\begin{align*}
\frac{1}{\omega_{L_0}^{X_{\Delta}} (\Delta')} \iint_{K\cap B'\cap \Omega} |\nabla u(X)|^2 G_{L_0}(X_\Delta,X) \, d X
=
\lim_{j\to\infty} \frac{1}{\omega_{L_0}^{X_{\Delta}} (\Delta')} \iint_{K\cap B'\cap \Omega} |\nabla v_j(X)|^2 G_{L_0}(X_\Delta,X) \, d X
\lesssim
1.
\end{align*}
Taking the sup over $B$ and $B'$ we then conclude that $(c)'$ holds since by the maximum principle one has $\|u\|_{L^\infty(\Omega)}=1$. \qed

\subsection{Proof of \texorpdfstring{$\mathrm{(e)}' \ \Longrightarrow\ \mathrm{(f)}'$}{(e)' implies (f)'}}
The argument used to see that $\mathrm{(e)} \ \Longrightarrow\ \mathrm{(f)}$ can be carried out in the present scenario with no changes. \qed

\subsection{Proof of \texorpdfstring{$\mathrm{(f)}'\ \Longrightarrow\ \mathrm{(c)}'$}{(f)' implies (c)'}}
Let $f=\mathbf{1}_S$ with $S\subset \pom$ a Borel set such that $\omega_{L_0}^X(S)\not= 0$ for some (or all) $X\in\Omega$. Note that $\|f\|_{\mathrm{BMO}(\partial \Omega, \omega_{L_0})}\le \|f\|_{L^\infty(\partial \Omega, \omega_{L_0})}=1$. This and the fact that $u(X)=\omega_L^X(S)$, $X\in \Omega$, verifies $\|u\|_{L^\infty(\Omega)}=1$, we readily see that \eqref{CME:BMO-Linfty} with $\varepsilon=1$ readily implies \eqref{CME:Linfty}. \qed

\subsection{Proof of \texorpdfstring{$\mathrm{(c)}'\ \Longrightarrow\ \mathrm{(a)}$}{(c)' implies (a)}}

Let $u(X)=\omega_L^X(S)$, $X\in\Omega$, for an arbitrary Borel set $S\subset\pom$. Let $\vartheta\ge \vartheta_0$  and $\eta\in (0,1)$.  Then
\begin{multline}\label{5qwtgg4}
	\int_{Q_0} \mathcal{S}_{Q_0,\eta}^{\vartheta}u(x)^2\,d \omega_{L_0}^{X_{Q_0}} (x)
	=
	\int_{Q_0}\bigg(\iint_{\Gamma_{Q_0,\eta}^{\vartheta}(x)}|\nabla u(Y)|^2\delta(Y)^{1-n}\,dY\bigg)\,d\omega_{L_0}^{X_{Q_0}}(x)
	\\
	=\iint_{B_{Q_0}^*\cap \Omega}|\nabla u(Y)|^2\delta(Y)^{1-n}\bigg(\int_{Q_0}\mathbf{1}_{\Gamma_{Q_0,\eta}^{\vartheta}(x)}(Y)\,d\omega_{L_0}^{X_{Q_0}}(x)\bigg)\,dY,
\end{multline}
	where we have used that $\Gamma_{Q_0,\eta}^{\vartheta}(x)\subset T_{Q_0}^{\vartheta,*}\subset B_{Q_0}^*\cap \overline{\Omega}$ (see \eqref{definicionkappa12}), and Fubini's theorem. To estimate the inner integral we fix $Y\in B_{Q_0}^*\cap\Omega$ and $\widehat{y}\in \partial\Omega$ such that $|Y-\widehat{y}|=\delta(Y)$. We claim that
\begin{equation}\label{claimendproof}
	\big\{x\in Q_0:\:Y\in\Gamma_{Q_0,\eta}^{\vartheta}(x)\big\}\subset\Delta(\widehat{y},C_\vartheta\eta^{-3}\delta(Y)).
\end{equation}
To show this let $x\in Q_0$ be such that $Y\in\Gamma_{Q_0,\eta}^{\vartheta}(x)$. This means that there exists $Q\in\mathbb{D}_{Q_0}$ such that $x\in Q$ and $Y\in U^\vartheta_{Q,\eta^3}$. Hence, there is $Q'\in\mathbb{D}_Q$ with $\ell(Q')>\eta^3\ell(Q)$ such that $Y\in U^\vartheta_{Q'}$ and consequently $\delta(Y)\approx_\vartheta \dist(Y,Q')\approx_\vartheta \ell(Q')$. As a result,
$$
|x-\widehat{y}|\leq\diam(Q)+\dist(Y,Q')+\delta(Y)
\lesssim_\vartheta\ell(Q)+\delta(Y)\lesssim_\vartheta \eta^{-3}\delta(Y),
$$
thus $x\in\Delta(\widehat{y},C_\vartheta\eta^{-3}\delta(Y))$ as desired. If we now use \eqref{claimendproof}  we conclude that for every $Y\in B_{Q_0}^*\cap\Omega$
\begin{equation}\label{342qw3f}
	\int_{Q_0}\mathbf{1}_{\Gamma_{Q_0,\eta}^{\vartheta}(x)}(Y)\,d\omega_{L_0}^{X_{Q_0}}(x)
\le
\omega_{L_0}^{X_{Q_0}}\big(\Delta(\widehat{y},C_\vartheta \eta^{-3}\delta(Y))\big)
\lesssim_{\vartheta,\eta}
\omega_{L_0}^{X_{Q_0}}\big(\Delta(\widehat{y},\delta(Y))\big).
\end{equation}
Write $B=8c_0^{-1} B_{Q_0}^*$, $B'=B_{Q_0}^*$, $\Delta=B\cap\pom$,  $\Delta'=B'\cap\pom$. Assuming that  $r_B=16c_0^{-1} \kappa_0 r_{Q_0}<\diam(\pom)$ we have by Lemma~\ref{lemma:proppde} part (b) and Harnack's inequality
\begin{equation}\label{342qw3f:1}
\omega_{L_0}^{X_{Q_0}}\big(\Delta(\widehat{y},\delta(Y))\big)
\approx
\omega_{L_0}^{X_{\Delta}}\big(\Delta(\widehat{y},\delta(Y))\big)
\approx
\delta(Y)^{n-1} G_{L_0}(X_{\Delta}, Y),
\qquad Y\in B_{Q_0}^*\cap \overline{\Omega}=B'\cap \overline{\Omega}.
\end{equation}
If we then combine \eqref{5qwtgg4}, \eqref{342qw3f}, and \eqref{342qw3f:1} we conclude that (c)$'$ and Lemma~\ref{lemma:proppde} yield
\begin{equation}\label{S<CME}
\int_{Q_0} \mathcal{S}_{Q_0,\eta}^{\vartheta}u(x)^2\,d \omega_{L_0}^{X_{Q_0}} (x)
\lesssim_{\vartheta,\eta}
\iint_{B' \cap\Omega}|\nabla u(Y)|^2 G_{L_0}(X_{\Delta}, Y)\,dY
\lesssim
\omega_{L_0}^{X_{\Delta}}(\Delta')\,\|u\|^2_{L^\infty(\Omega)}
\lesssim
\omega_{L_0}^{X_{Q_0}}(Q_0).
\end{equation}
Note that this estimate corresponds to \eqref{34wg4sg4s} for $q=2$. Hence the same argument we use in $\mathrm{(d)}'\ \Longrightarrow\ \mathrm{(a)}$ applies in this scenario. Note however, that we have assumed that $16c_0^{-1} \kappa_0 r_{Q_0}<\diam(\pom)$ and this causes that
\eqref{Ainfty-daydic} is valid under this restriction. If $\pom$ is unbounded then the same argument applies. When $\pom$ is bounded we can replace the family $\F$ by $\F'$ consisting on all $Q'\in\dd$ with $Q'\subset Q$ for some $Q\in\F$ and $\ell(Q')=2^{-M}\ell(Q)$ where $M$ is large enough so that $2^{-M}<\Xi c_0/(8\kappa_0)$. This guarantees that $16c_0^{-1} \kappa_0 r_{Q'}<\diam(\pom)$ for every $Q'\in\F'$ and thus \eqref{Ainfty-daydic} holds for every $Q'\in\F'$. At this point the rest of the argument can be carried out \textit{mutatis mutandis}, details are left to the reader. \qed

\subsection{Proof of \texorpdfstring{$\mathrm{(a)}\ \Longrightarrow\ \mathrm{(c)}$}{(a) implies (c)}}
Note that we have already proved that (a) implies (d), in particular we know that \eqref{S<N} holds with $q=2$ and for any $\alpha\ge c_0^{-1}$. Our goal is to see that the latter estimate implies (c). With this goal in mind consider $u\in W^{1,2}_{\loc}(\Omega)\cap L^\infty(\Omega)$ satisfying $Lu=0$ in the weak sense in $\Omega$.
Fix $B=B(x,r)$ with $x\in \pom$ and $0<r<\diam(\pom)$ and $B'=B(x',r')$ with $x'\in2\Delta$ and $0<r'<r c_0/4$. Let $\Delta=B\cap\pom$, $\Delta'=B'\cap\pom$.  Note that $2r'<r c_0/2<\diam(\pom)$ and we can now invoke Lemma~\ref{lemma:CME-S} and \eqref{S<N} with $q=2$  to conclude that
	\begin{align*}
	&\frac1{\omega_{L_0}^{X_{\Delta}}(\Delta')}\iint_{B'\cap \Omega} |\nabla u(X)|^2 G_{L_0}(X_\Delta,X) \, d X
	\\
	&\qquad \lesssim
	\int_{2\Delta'} \mathcal{S}_{2r'}^{C\alpha} u(y)^2\,d\omega_{L_0}^{X_{2\Delta'}}(y)+ \sup \{|u(Y)|: Y\in 2\,B', \delta(Y)\ge r'/C\}^2
	\\
	&\qquad \lesssim
	\int_{4\Delta'} \mathcal{N}_{2r'}^{\alpha'} u(y)^2\,d\omega_{L_0}^{X_{2\Delta'}}(y) + \|u\|_{L^\infty(\Omega)}^2
		\\
	&\qquad \lesssim
	\|u\|_{L^\infty(\Omega)}^2.
\end{align*}
Taking the sup over $B$ and $B'$ we have then shown \eqref{CME:Linfty}. \qed

\subsection{Proof of \texorpdfstring{$\mathrm{(a)}\ \Longrightarrow\ \mathrm{(e)}$}{(a) implies (e)}}
Fix $f \in \mathscr{C}(\partial \Omega)\cap L^\infty(\pom)$ and let $u$ be its associated solution as in \eqref{u-elliptic:L-sol}. Let $u_{L,\Omega}(X):=\omega_L^X(\pom)$, $X\in\Omega$. Fix $B=B(x,r)$ with $x\in \pom$ and $0<r<\diam(\pom)$ and $B'=B(x',r')$ with $x'\in2\Delta$ and $0<r'<r c_0/4$. Let $\Delta=B\cap\pom$, $\Delta'=B'\cap\pom$.
Let $\varphi\in \mathscr{C}(\re)$ with $\mathbf{1}_{[0,4)}\le \varphi\le \mathbf{1}_{[0,8)}$ and $\varphi_{\Delta'}:=\varphi(|\cdot-x'|/r')$ so that $\mathbf{1}_{4\Delta'}\le \varphi_{\Delta'} \le \mathbf{1}_{8\Delta'}$.
Recall that for every surface ball $\widetilde{\Delta}$ we write $f_{\widetilde{\Delta}, L_0}:=\displaystyle\fint_{\widetilde{\Delta}} f\,d\omega_{L_0}^{X_{\widetilde{\Delta}}}$. Then,
\begin{multline*}
f-f_{\Delta, L_0}
=
(f-f_{8\Delta', L_0})+
(f_{8\Delta', L_0}-f_{\Delta, L_0})
=
(f-f_{8\Delta', L_0})\varphi_{\Delta'}+ (f-f_{8\Delta', L_0})(1-\varphi_{\Delta'})+(f_{8\Delta', L_0}-f_{\Delta, L_0})
\\
=:
h_{\loc}+h_\mathrm{glob}+ (f_{8\Delta', L_0}-f_{\Delta, L_0}).
\end{multline*}
Hence,
\begin{multline}\label{decomp-v}
	v(X)
:=
u(X)-f_{\Delta, L_0}u_{L,\Omega}(X)
=
\int_{\pom} (f(y)-f_{\Delta, L_0})\,d\omega_{L}^{X}(y)
\\
=
\int_{\pom} h_\loc(y)\,d\omega_{L}^{X}(y)
+
\int_{\pom} h_\mathrm{glob}(y)\,d\omega_{L}^{X}(y) +(f_{8\Delta', L_0}-f_{\Delta, L_0})u_{L,\Omega}(X)
\\
=:
v_\loc(X)+v_\mathrm{glob}(X)+(f_{8\Delta', L_0}-f_{\Delta, L_0})u_{L,\Omega}(X).
\end{multline}
Note that $h_\loc, h_\mathrm{glob} \in \mathscr{C}(\partial \Omega) \cap L^\infty(\partial \Omega)$.

Let us observe that we have already proved that (a) implies (d), in particular we know that \eqref{S<N} holds with $q=2$ and for any $\alpha\ge c_0^{-1}$. Hence, since $2r'<r c_0/2<\diam(\pom)$ and we can now invoke Lemma~\ref{lemma:CME-S} and \eqref{S<N} with $q=2$  to conclude that
\begin{align}\label{BMO-loc}
	&\frac1{\omega_{L_0}^{X_{\Delta}}(\Delta')}\iint_{B'\cap \Omega} |\nabla v_\loc(X)|^2 G_{L_0}(X_\Delta,X) \, d X
	\\    \nonumber
	&\qquad \lesssim
	\int_{2\Delta'} \mathcal{S}_{2r'}^{C\alpha} v_\loc(y)^2\,d\omega_{L_0}^{X_{2\Delta'}}(y)+ \sup \{|v_\loc(Y)|: Y\in 2\,B', \delta(Y)\ge r'/C\}^2
	\\ \nonumber
	&\qquad \lesssim
	\int_{4\Delta'} \mathcal{N}_{2r'}^{\alpha'} v_\loc(y)^2\,d\omega_{L_0}^{X_{2\Delta'}}(y) + \Big(\int_{\pom} |h_\loc(y)|\,d\omega_{L_0}^{X_{\Delta'}}(y)\Big)^2
	\\ \nonumber
	&\qquad \lesssim
	\int_{4\Delta'} \mathcal{N}_{4r'}^{\alpha'} v_\loc(y)^2\,d\omega_{L_0}^{X_{4\Delta'}}(y)
	+ \Big(\int_{\pom} |h_\loc(y)|\,d\omega_{L_0}^{X_{8\Delta'}}(y)\Big)^2
		\\ \nonumber
	&\qquad =: \mathcal{I}_1+\mathcal{I}_2.
\end{align}
Regarding $\mathcal{I}_2$ we note that by Lemma~\ref{lemma:proppde} part $(a)$ there holds
\begin{align}
	\label{BMO-loc:I2}
\mathcal{I}_2
\le
\Big(\int_{8\Delta'} |f(y)-f_{8\Delta', L_0}|d\omega_{L_0}^{X_{8\Delta'}}(y)\Big)^2
\lesssim  \|f\|^2_{\mathrm{BMO}(\partial \Omega, \omega_{L_0})}.
\end{align}
To estimate $\mathcal{I}_1$ we first observe that since $\omega_L\in A_\infty(\pom, \omega_{L_0})$, there is $q\in (1,\infty)$ so that $\omega_L\in RH_q(\pom, \omega_{L_0})$. Note that by Jensen's inequality we may assume that $q<2$ (since $RH_{q_1}(\pom, \omega_{L_0})\subset RH_{q_2}(\pom, \omega_{L_0})$ if $q_2\le q_1$). Note that we have already proved that (a)${}_{q}$ implies (b)${}_{q'}$, hence \eqref{solv-Lp} holds with $p=q' > 2$. This, H\"older's inequality and the fact that $h_\loc\in \mathscr{C}(\partial \Omega)$ with $\supp h_\loc\subset 8\Delta'$ readily lead to
\begin{multline}\label{BMO-loc:I1}
	 \mathcal{I}_1
 \le
 \|\mathcal{N}_{4r'}^{\alpha'} v_\loc\|_{L^{q'}(4\Delta',\omega_{L_0}^{X_{4\Delta'}})}^2 \omega_{L_0}^{X_{4\Delta'}}(4\Delta')^{\frac1{(q'/2)'}}
 \lesssim
 \|h_\loc\|_{L^{q'}(8\Delta',\omega_{L_0}^{X_{4\Delta'}})}^2
 \\
 \lesssim
 \Big(\int_{8\Delta'} |f(y)-f_{8\Delta', L_0}|^{q'}d\omega_{L_0}^{X_{4\Delta'}}(y) \Big)^\frac{2}{q'}
 \lesssim  \|f\|^2_{\mathrm{BMO}(\partial \Omega, \omega_{L_0})},
\end{multline}
where the last estimate uses Lemma~\ref{lemma:proppde} part $(a)$ and John-Nirenberg's inequality (cf.~\eqref{eqn:JN}).

We next turn our attention to the estimate involving $v_\mathrm{glob}$. Note that
\begin{multline*}
|h_\mathrm{glob}|
\le
|f-f_{8\Delta', L_0}|\mathbf{1}_{\pom\setminus 4\Delta'}
=
\sum_{k=2}^\infty |f-f_{8\Delta', L_0}|\mathbf{1}_{2^{k+1}\Delta'\setminus 2^k\Delta'}
\\
\le
\sum_{k=2}^\infty |f-f_{8\Delta', L_0}|(\varphi_{2^{k-1}\Delta'}-\varphi_{2^{k-3}\Delta'})
=:
\sum_{k\ge 2: 2^kr' \leq \diam(\pom)}h_\mathrm{glob,k},
\end{multline*}
with the understanding that the sum runs from $k=2$ to infinity when $\pom$ is unbounded.

Fix $k\ge 2$ with $2^kr' \leq \diam(\pom)$ and note that $h_\mathrm{glob, k}\in \mathscr{C}(\partial \Omega)$ with $\supp h_\mathrm{glob, k}\subset 2^{k+2}\Delta'\setminus 2^{k-1}\Delta'$. Thus, for every $X\in B'\cap\Omega$, by Lemma~\ref{lemma:proppde} part $(f)$ we have
\begin{align}\label{sdfffrgq}
v_\mathrm{glob,k}(X):=\int_{\pom} h_\mathrm{glob,k}(y) \, d{\omega_{L}^X} (y)
\lesssim
\Big(\frac{\delta(X)}{2^{k-1}r'}\Big)^\rho v_\mathrm{glob,k}(X_{2^{k-1}\Delta'}).
\end{align}
Next we estimate $v_\mathrm{glob,k}(X_{2^{k-1} \Delta'}), k \geq 2,$ via a telescopic argument. Indeed applying Harnack's inequality,  that $\omega_L \in RH_q(\partial \Omega, \omega_{L_0})$, Lemma~\ref{lemma:proppde},  and John-Nirenberg's inequality  (cf.~\eqref{eqn:JN}) we arrive at
\begin{align*}
v_\mathrm{glob,k}(X_{2^{k-1} \Delta'})
&
\le
\int_{2^{k+2} \Delta' } |f(y) - f_{8 \Delta', L_0}| \, d \omega_L^{X_{2^{k-1} \Delta'}}(y)
\\
&
\lesssim
\int_{2^{k+2} \Delta' } |f(y) - f_{8 \Delta', L_0}| \, d \omega_L^{X_{2^{k+2} \Delta'}}(y)
\\
&
\lesssim
\Big(\fint_{2^{k+2} \Delta' } |f(y) - f_{8 \Delta', L_0}|^{q'} \, d \omega_{L_0}^{X_{2^{k+2} \Delta'}}(y) \Big)^{\frac1{q'}}
\\
&
\le
\Big(\fint_{2^{k+2} \Delta' } |f(y) - f_{2^{k+2} \Delta', L_0}|^{q'} \, d \omega_{L_0}^{X_{2^{k+2} \Delta'}}(y) \Big)^{\frac1{q'}}
+
\sum_{j=3}^{k+1} |f_{2^{j+1} \Delta', L_0} - f_{2^{j} \Delta', L_0}|
\\
&
\le
\Big(\fint_{2^{k+2} \Delta' } |f(y) - f_{2^{k+2} \Delta', L_0}|^{q'} \, d \omega_{L_0}^{X_{2^{k+2} \Delta'}}(y) \Big)^{\frac1{q'}}
+
\sum_{j=3}^{k+1} \fint_{2^{j}\Delta' } |f(y) - f_{2^{j+1} \Delta', L_0}| \, d \omega_{L_0}^{X_{2^{j} \Delta'}}(y)
\\
&
\lesssim
\sum_{j=3}^{k+1} \Big(\fint_{2^{j+1}\Delta' } |f(y) - f_{2^{j+1} \Delta', L_0}|^{q'} \, d \omega_{L_0}^{X_{2^{j+1} \Delta'}}(y) \Big)^{\frac1{q'}}
\\
&
\lesssim
k\,\|f\|_{\mathrm{BMO}(\partial \Omega, \omega_{L_0})}.
\end{align*}
This and \eqref{sdfffrgq} give for every $X\in B' \cap \Omega$
\begin{multline*}
\int_{\pom} |h_\mathrm{glob}(y)|\,d\omega_{L}^{X}(y)
\le
\sum_{k\ge 2: 2^kr' \leq \diam(\pom)} \int_{\pom}  h_\mathrm{glob,k}(y) \,d\omega_{L}^{X}(y)
=
\sum_{k\ge 2: 2^kr' \leq \diam(\pom)} v_\mathrm{glob,k}(X)
\\
\lesssim
\sum_{k\ge 2} k\,\Big(\frac{\delta(X)}{2^{k-1}r'}\Big)^\rho \|f\|_{\mathrm{BMO}(\partial \Omega, \omega_{L_0})}
\approx
\Big(\frac{\delta(X)}{r'}\Big)^\rho \|f\|_{\mathrm{BMO}(\partial \Omega, \omega_{L_0})}
.
\end{multline*}
If we next write  $\W_{B'}:=\{I\in\W: I\cap B'\neq\emptyset\}$ and pick $Z_{I, B'}\in I\cap B'$, the previous estimate gives for every $I\in\W_{B'}$
\begin{multline*}
\iint_I |\nabla v_\mathrm{glob}(X)|^2 dX
	\lesssim
	\ell(I)^{-2} \iint_{I^*} v_\mathrm{glob}(X)^2 dX
	\le
	\ell(I)^{-2} \iint_{I^*} \Big(\int_{\pom} |h_\mathrm{glob}(y)|\,d\omega_{L}^{X}(y)\Big)^2\,dX
	\\
	\approx
	\ell(I)^{n-1} \Big(\int_{\pom} h_\mathrm{glob}(y)\,d\omega_{L}^{Z_{I, B'}}(y)\Big)^2
	\lesssim
	\ell(I)^{n-1} \Big(\frac{\ell(I)}{r'}\Big)^{2\rho} \|f\|^2_{\mathrm{BMO}(\partial \Omega, \omega_{L_0})}.
\end{multline*}
Thus, Lemma~\ref{lemma:proppde} gives
\begin{align*}
	\iint_{B'\cap \Omega} |\nabla v_\mathrm{glob}(X)|^2 G_{L_0}(X_\Delta,X) \, d X
	&\lesssim
	\sum_{I\in \W_{B'}} \omega_{L_0}^{X_{\Delta}}(Q_I)\ell(I)^{1-n} \iint_{I} |\nabla v_\mathrm{glob}(X)|^2 \, d X
	\\
	&\lesssim \|f\|^2_{\mathrm{BMO}(\partial \Omega, \omega_{L_0})}
	\sum_{I\in \W_{B'}} \omega_{L_0}^{X_{\Delta}}(Q_I)   \Big(\frac{\ell(I)}{r'}\Big)^{2\rho}
	\\
	&\lesssim \|f\|^2_{\mathrm{BMO}(\partial \Omega, \omega_{L_0})}
	\sum_{k: 2^{-k}\lesssim r'} \Big(\frac{2^{-k}}{r'}\Big)^{2\rho}
	\sum_{I\in \W_{B'}: \ell(I)=2^{-k}} \omega_{L_0}^{X_{\Delta}}(Q_I),
\end{align*}
where $Q_I\in\dd(\pom)$ is so that $\ell(Q_I)=\ell(I)$ and contains $\widehat{y}_I \in \partial \Omega$ such that
$\dist(I, \partial \Omega)=\dist(\widehat{y}_I, I)$. It is easy to see that for every $k$ with $2^{-k}\lesssim r'$, the family  $\{Q_I\}_{I\in\W_{B'}, \ell(I)=2^{-k}}$ has bounded overlap and also that  $Q_I\subset C\Delta'$ for every $I\in\W_{B'}$, where $C$ is some harmless dimensional constant. Hence,
\begin{multline}\label{BMO-glob}
	\iint_{B'\cap \Omega} |\nabla v_\mathrm{glob}(X)|^2 G_{L_0}(X_\Delta,X) \, d X
	\lesssim
	 \|f\|^2_{\mathrm{BMO}(\partial \Omega, \omega_{L_0})}
	\sum_{k: 2^{-k}\lesssim r'} \Big(\frac{2^{-k}}{r'}\Big)^{2\rho} \omega_{L_0}^{X_{\Delta}}(C\,\Delta')
	\\
	\lesssim
 \|f\|^2_{\mathrm{BMO}(\partial \Omega, \omega_{L_0})}\omega_{L_0}^{X_{\Delta}}(\Delta').
\end{multline}

To continue we pick $k_0\ge 3$ such that $r<2^{k_0}r'\le 2\,r$. Note that $2^{k_0+1} \Delta'$ and $\Delta$ have comparable radius and $x'\in 2\Delta\cap 2^{k_0+1} \Delta'$. Hence, Lemma~\ref{lemma:proppde} and Harnack's inequality yield
\begin{align}\label{34qr34fa}
|f_{8\Delta', L_0}-f_{\Delta, L_0}|
&\le
\sum_{k=3}^{k_0} |f_{2^k\Delta', L_0}-f_{2^{k+1}\Delta, L_0}|+ |f_{2^{k_0+1}\Delta', L_0}-f_{\Delta, L_0}|
\\ \nonumber
&\le
\sum_{k=3}^{k_0} \fint_{2^{k}\Delta' } |f(y) - f_{2^{k+1} \Delta', L_0}| \, d \omega_{L_0}^{X_{2^{k} \Delta'}}(y)
+
\fint_{\Delta} |f(y) - f_{2^{k_0+1} \Delta', L_0}| \, d \omega_{L_0}^{X_{\Delta}}(y)
\\ \nonumber
&\lesssim
\sum_{k=3}^{k_0} \fint_{2^{k+1}\Delta' } |f(y) - f_{2^{k+1} \Delta', L_0}| \, d \omega_{L_0}^{X_{2^{k+1} \Delta'}}(y)
\\ \nonumber
&\lesssim
k_0\,\|f\|_{\mathrm{BMO}(\partial \Omega, \omega_{L_0})}
\\ \nonumber
& \le
(1+\log(r/r'))\,\|f\|_{\mathrm{BMO}(\partial \Omega, \omega_{L_0})}.
\end{align}
This and Lemma~\ref{lemma:CME-uOmega} imply
\begin{multline}\label{BMO-rest}
\frac{1}{\omega_{L_0}^{X_{\Delta}} (\Delta')} \iint_{B'\cap \Omega} \big| (f_{8\Delta', L_0}-f_{\Delta, L_0}) \nabla u_{L,\Omega}(X)\big|^2 G_{L_0}(X_\Delta,X) \, d X
\\
\lesssim
(1+\log(r/r'))^2\,\Big(\frac{r'}{\diam(\pom)}\Big)^{2\rho} \|f\|^2_{\mathrm{BMO}(\partial \Omega, \omega_{L_0})}
\\
\le
(1+\log(r/r'))^2\,\Big(\frac{r'}{r}\Big)^{2\rho} \|f\|^2_{\mathrm{BMO}(\partial \Omega, \omega_{L_0})}
\lesssim
\|f\|^2_{\mathrm{BMO}(\partial \Omega, \omega_{L_0})}
.
\end{multline}
Here we note in passing that if $\diam(\pom)=\infty$ (or if both $\pom$ and $\Omega$ are bounded)  then the left-hand side of the previous estimate vanishes as we know that $u_{L,\Omega}\equiv 1$.

To complete the proof we just collect \eqref{decomp-v}--\eqref{BMO-loc:I1}, \eqref{BMO-glob}, and \eqref{BMO-rest}:
\begin{align*}
\iint_{B'\cap \Omega} |\nabla v(X)\big|^2 G_{L_0}(X_\Delta,X) \, d X
&\lesssim
\iint_{B'\cap \Omega} |\nabla v_{\loc}(X)\big|^2 G_{L_0}(X_\Delta,X) \, d X
+\iint_{B'\cap \Omega} |\nabla v_\mathrm{glob}(X)\big|^2 G_{L_0}(X_\Delta,X) \, d X
\\
&
\qquad\qquad+
\iint_{B'\cap \Omega} \big| (f_{8\Delta', L_0}-f_{\Delta, L_0}) \nabla u_{L,\Omega}(X)\big|^2 G_{L_0}(X_\Delta,X) \, d X
\\
&	\lesssim
\|f\|^2_{\mathrm{BMO}(\partial \Omega, \omega_{L_0})}\omega_{L_0}^{X_{\Delta}}(\Delta').
\end{align*}
This completes the proof.\qed

\begin{remark}\label{remark:BMO-other-cons:proof}
It is not difficult to see that in \eqref{CME:BMO} one can replace $f_{\Delta, L_0}$ by  $f_{\Delta', L_0}$. Indeed, this is what we have essentially done in the proof: much as in \eqref{34qr34fa} one has that
	\[
	|f_{\Delta, L_0}-f_{\Delta', L_0}|\lesssim (1+\log(r/r'))\,\|f\|_{\mathrm{BMO}(\partial \Omega, \omega_{L_0})}.
	\]
	With this we can proceed as in \eqref{BMO-rest} to see that
	\begin{align*}
		\frac{1}{\omega_{L_0}^{X_{\Delta}} (\Delta')} \iint_{B'\cap \Omega} \big| (f_{\Delta, L_0}-f_{\Delta', L_0}) \nabla u_{L,\Omega}(X)\big|^2 G_{L_0}(X_\Delta,X) \, d X
		\lesssim
		\|f\|^2_{\mathrm{BMO}(\partial \Omega, \omega_{L_0})}.		
	\end{align*}
	Hence, \eqref{CME:BMO} with $f_{\Delta, L_0}$ is equivalent to \eqref{CME:BMO} with $f_{\Delta', L_0}$.
	
	On the other hand, when $\Omega$ is unbounded and $\pom$ bounded,  in \eqref{CME:BMO} one can replace $f_{\Delta, L_0}$  by $f_{\pom, L_0}:=\displaystyle\fint_{\pom} f\,d\omega_{L_0}^{X_\Omega}$, where $X_\Omega\in\Omega$ satisfy $\delta(X_\Omega)\approx \diam(\pom)$ (say, $X_\Omega=X_{\Delta(x_0, r_0)}$ with $x_0\in \pom$ and $r_0\approx\diam(\pom)$). To see this, one proceeds as in \eqref{34qr34fa} to see that
	\[
	|f_{\Delta, L_0}-f_{\pom, L_0}|\lesssim (1+\log(\diam(\pom)/r))\,\|f\|_{\mathrm{BMO}(\partial \Omega, \omega_{L_0})}.
	\]
	This and Lemma~\ref{lemma:CME-uOmega} readily give
	\begin{multline*}
		\frac{1}{\omega_{L_0}^{X_{\Delta}} (\Delta')} \iint_{B'\cap \Omega} \big| (f_{\Delta, L_0}-f_{\pom, L_0}) \nabla u_{L,\Omega}(X)\big|^2 G_{L_0}(X_\Delta,X) \, d X
		\\
		\lesssim
		(1+\log(\diam(\pom)/r))^2\,\Big(\frac{r'}{\diam(\pom)}\Big)^{2\rho} \|f\|^2_{\mathrm{BMO}(\partial \Omega, \omega_{L_0})}
		\\
		\le
		(1+\log(\diam(\pom)/r))^2\,\Big(\frac{r}{\diam(\pom)}\Big)^{2\rho} \|f\|^2_{\mathrm{BMO}(\partial \Omega, \omega_{L_0})}
		\lesssim
		\|f\|^2_{\mathrm{BMO}(\partial \Omega, \omega_{L_0})}
		.
	\end{multline*}
	Hence, \eqref{CME:BMO} with $f_{\Delta, L_0}$ is equivalent to \eqref{CME:BMO} with $f_{\pom, L_0}$.

\end{remark}

\subsection{Proof of \texorpdfstring{$\mathrm{(a)}\ \Longrightarrow\ \mathrm{(e)}'$}{(a) implies (e)'}}
The proof is almost the same as the previous one with the following modifications. We work with $f=\mathbf{1}_S$ with $S\subset\pom$ an arbitrary Borel set. We replace $\varphi$ by $\mathbf{1}_{[0,4)}$ and use in \eqref{BMO-loc} that Lemma~\ref{lemma:CME-S} is also valid for the associated $v_{\loc}$ since it belongs to $W^{1,2}_{\loc}(\Omega)\cap L^\infty(\Omega)$. Also, in \eqref{BMO-loc:I2} we need to invoke that  $\mathrm{(a)}{}_{q}\ \Longrightarrow\ \mathrm{(b)}_{q'}\ \Longrightarrow\  \mathrm{(b)}'_{q'}$. The rest of the proof remains the same, details are left to the interested reader.\qed

\section{Proof of Theorem~\ref{thm:abs}}\label{sec:abs}

The implications $\mathrm{(b)}\ \Longrightarrow\ \mathrm{(c)}\ \Longrightarrow\ \mathrm{(d)}$, $\mathrm{(b)}'\ \Longrightarrow\ \mathrm{(c)}'\ \Longrightarrow\ \mathrm{(d)'}$ are trivial. Also, since for any Borel set $S\subset\pom$ the solution $u(X)=\omega_L^X(S)$ belongs to $W^{1,2}_\loc(\Omega)\cap L^\infty(\Omega)$, it is also straightforward that $\mathrm{(b)}\ \Longrightarrow\ \mathrm{(b)}'$, $\mathrm{(c)}\ \Longrightarrow\ \mathrm{(c)}'$, and $\mathrm{(d)}\ \Longrightarrow\ \mathrm{(d)}'$.

We next observe that for every $\alpha>0$, $0<r<r'$, and $\varpi\in \R$, if $F \subset\pom$ is a bounded set  and $v\in L^2_{\loc}(\Omega)$, then
	\begin{equation}\label{est:two-trunc}
		\sup_{x\in F}\iint_{\Gamma^{\alpha}_{r'}(x) \setminus \Gamma^{\alpha}_{r}(x)}
		|v(Y)|^2 \delta(Y)^{\varpi} dY
		<\infty.
	\end{equation}
	To see this we first note that  since $F$ is bounded we can find $R$ large enough so that $F\subset B(0,R)$. Then, if $x\in F$ one readily sees that
	\begin{align*}
		\Gamma^{\alpha}_{r'}(x) \backslash \Gamma^{\alpha}_{r}(x)
		\subset \overline{B(0, r'+R)} \cap
		\Big\{Y \in \Omega: \frac{r}{1+\alpha} \leq \delta(Y) \leq r' \Big\} =:K.
	\end{align*}
	Note that $K \subset \Omega$ is a compact set.  Then, since  $v \in L^2_{\loc}(\Omega)$, we conclude that
	\begin{align}\label{eq:Gr-Gr0}
		\sup_{x\in F}\iint_{\Gamma^{\alpha}_{r'}(x) \setminus \Gamma^{\alpha}_{r}(x)}
		|v(Y)|^2 \delta(Y)^{\varpi} dY
		\leq \max\left\{r',\frac{1+\alpha}{r}\right\}^{|\varpi|} \iint_{K}  |v(Y)|^2 dY<\infty.
	\end{align}

Using then \eqref{est:two-trunc} it is also trivial to see that $\mathrm{(d)}\ \Longrightarrow\ \mathrm{(c)}$ and $\mathrm{(d)}'\ \Longrightarrow\ \mathrm{(c)}'$. Hence we are left with showing
\[
\mathrm{(a)}\ \Longrightarrow\ \mathrm{(b)}\quad\text{ and }\quad\mathrm{(c)}'\ \Longrightarrow\ \mathrm{(a)}.
\]

\subsection{Proof of \texorpdfstring{$\mathrm{(a)}\ \Longrightarrow\ \mathrm{(b)}$}{(a) implies (b)}}\label{sec:local}
Assume that $\w_{L_0} \ll \w_L$. Let $\vartheta\ge \vartheta_0$ large enough to be chosen (this choice will depend on $\alpha$). Fix an arbitrary $Q_0 \in \D_{k_0}$ where $k_0 \in \Z$ is taken so that $2^{-k_0}=\ell(Q_0) < \diam(\partial \Omega)/M_0$, where $M_0>8\kappa_0 c_0^{-1}$, $\kappa_0$ is taken from \eqref{definicionkappa12}, and $c_0$ is the Corkscrew constant. Let $X_0:=X_{M_0 \Delta_{Q_0}}$ be a Corkscrew point relative to $M_0 \Delta_{Q_0}$ so that $X_0 \notin 4B_{Q_0}^*$ by the choice of $M_0$. By Lemma~\ref{lemma:proppde} part $(a)$ and Harnack's inequality, there exists $C_0>1$ such that
\begin{equation}\label{eq:lowerbdd}
	\w_L^{X_0}(Q_0) \geq C_0^{-1}.
\end{equation}
Set
\begin{align}\label{eq:normalize}
	\w_0:=\w_{L_0}^{X_0},\quad \w:= C_0 \w_{L_0}^{X_0}(Q_0) \w_L^{X_0}, \quad  \G_0:= G_{L_0}(X_0, \cdot), \quad\text{ and }\quad
	\G:= C_0 \w_{L_0}^{X_0}(Q_0) G_L(X_0, \cdot).
\end{align}
By assumption,  we have $\w_0 \ll \w$ and by \eqref{eq:lowerbdd},
\begin{equation}\label{eq:QC}
	1 \leq \frac{\w(Q_0)}{\w_0(Q_0)} = C_0 \w_{L}^{X_0}(Q_0) \leq C_0.
\end{equation}
For $N > C_0$, we let $\F_N^+ :=\{Q_j\} \subset \D_{Q_0} \backslash \{Q_0\}$, respectively, $\F_N^- :=\{Q_j\} \subset \D_{Q_0} \backslash \{Q_0\}$, be the collection of descendants of $Q_0$ which are maximal (and therefore pairwise disjoint) with respect to the property that
\begin{align}\label{eq:stopping}
	\frac{\w(Q_j)}{\w_0(Q_j)} < \frac{1}{N}, \qquad\text{ respectively}\quad \frac{\w(Q_j)}{\w_0(Q_j)} >N.
\end{align}
Write $\F_N:=\F_N^+\cup\F_N^-$ and note that $\F_N^+\cap\F_N^-=\emptyset$. By maximality, there holds
\begin{align}\label{eq:NN}
	\frac{1}{N}\leq \frac{\w(Q)}{\w_0(Q)} \leq N, \qquad \forall\,Q \in \D_{\F_N, Q_0}.
\end{align}
Denote, for every $N>C_0$,
\begin{align}\label{eq:E0N-EN}
	E_N^\pm := \bigcup_{Q \in \F_N^\pm} Q,
	\qquad
	E_N^0:=E_N^+\cup E_N^-, \qquad E_N := Q_0\setminus E_N^0,
\end{align}
and
\begin{align}\label{eq:Q-decom}
	Q_0 = \bigg(\bigcap_{N>C_0} E_N^0\bigg)\cup \bigg(\bigcup_{N>C_0} E_N \bigg)
	=: E_0\cup \bigg(\bigcup_{N>C_0} E_N \bigg).
\end{align}

By Lemma~\ref{lemma:CDC-inherit}, $\Omega_{\F_N, Q_0}^\vartheta$ is a bounded 1-sided NTA satisfying the CDC for any $\vartheta\ge \vartheta_0$. As in \cite[Proposition~6.1]{HM1}
\begin{align*}
	E_N \subset F_N:=\partial \Omega \cap \partial \Omega_{\F_N, Q_0}^{\vartheta}
	\subset \overline{Q}_0 \setminus \bigcup_{Q_j \in \F_N} \interior(Q_j).
\end{align*}
Hence,
\begin{align*}
	F_N \backslash E_N
	\subset \bigg(\overline{Q}_0 \backslash \bigcup_{Q_j \in \F_N} \interior(Q_j) \bigg)
	\backslash \bigg(Q_0 \backslash \bigcup_{Q_j \in \F_N} Q_j \bigg)
	\subset \partial Q_0 \cup \bigg(\bigcup_{Q_j \in \F_N} \partial Q_j\bigg).
\end{align*}
This, \cite[Lemma~2.37]{AHMT}, and Lemma~\ref{lemma:proppde} imply
\begin{align}\label{eq:ENFN}
	\w_0(F_N \setminus E_N) =0.
\end{align}

Next, we are going to show
\begin{equation}\label{eq:WE}
	\w_0(E_0)=0.
\end{equation}
Let $x \in E_{N+1}^{ \pm}$. Then there exists $Q_x \in \F_{N+1}^\pm$ such that $x \in Q_x$. By \eqref{eq:stopping}, we have
\begin{align*}
	\frac{\w(Q_x)}{\w_0(Q_x)} < \frac{1}{N+1} <\frac{1}{N} \quad\text{if $Q_x \in \F_{N+1}^+$} \qquad \text{or} \qquad
	\frac{\w(Q_x)}{\w_0(Q_x)} >N+1>N \quad\text{if $Q_x \in \F_{N+1}^-$}.
\end{align*}
By the maximality of the cubes in $\F_N^\pm$, one has $Q_x \subset Q'_x$ for some $Q'_x \in \F_N^\pm$ with $x \in Q'_x \subset E_N^\pm$. Thus, $\{E_N^+\}_N$, $\{E_N^-\}_N$ and $\{E_N^0\}_N$ are decreasing sequences of sets. This, together with the fact that $\w(E_N^\pm)\le \w(Q_0)\le C_0\omega_0(Q_0) \le C_0$ and $\w_0(E_N^\pm)\le \w_0(Q_0) \le 1$, implies that
\begin{equation}\label{wrqfawfvrw}
	\w\bigg(\bigcap_{N>C_0} E_N^\pm\bigg)=\lim_{N\to\infty} \w(E_N^\pm)
	\qquad\text{and}\qquad
	\w_0\bigg(\bigcap_{N>C_0} E_N^\pm\bigg)=\lim_{N\to\infty} \w_0(E_N^\pm).
\end{equation}
By \eqref{eq:stopping} and \eqref{eq:E0N-EN},
\[
\omega(E_N^+) = \sum_{Q\in \F_N^+} \omega(Q)
<\frac1N\sum_{Q\in \F_N^+} \w_0(Q)
=\frac1N\w_0(E_N^+) \le \frac1N,
\]
which together with \eqref{wrqfawfvrw} yields
\[
\omega\bigg(\bigcap_{N>C_0} E_N^+\bigg)=\lim_{N\to\infty} \omega(E_N^+)=0.
\]
In view of the fact that by assumption $\w_0 \ll \w$, we then conclude that
\begin{equation}\label{eq:EN+}
	0=\w_0\bigg(\bigcap_{N>C_0} E_N^+\bigg)=\lim_{N\to\infty} \w_0(E_N^+).
\end{equation}
On the other hand, \eqref{eq:stopping} yields
\[
\w_0(E_N^-) = \sum_{Q\in \F_N^-} \w_0(Q)
<\frac1N\sum_{Q\in \F_N^-} \omega(Q)
=\frac1N\omega(E_N^-) \le \frac{C_0}N,
\]
and hence,
\begin{equation}\label{eq:EN-}
	\w_0\bigg(\bigcap_{N>C_0} E_N^-\bigg)=\lim_{N\to\infty} \w_0(E_N^-)=0.
\end{equation}
Since $\{E_N^0\}_N$ is a decreasing sequence of sets with $\w_0(E_N^0) \le \w_0(Q_0) \le 1$, \eqref{eq:EN+} and \eqref{eq:EN-} readily imply \eqref{eq:WE}:
\begin{equation*}
	\w_0(E_0) = \lim_{N\to\infty} \w_0(E_N^0)
	\le \lim_{N\to\infty} \w_0(E_N^+) + \lim_{N\to\infty} \w_0(E_N^-)
	=0.
\end{equation*}

Now we turn our attention to the square function estimates in $L^q(F_N, \w_0)$ for $q\in (0,\infty)$. Let $u \in W_{\loc}^{1,2}(\Omega) \cap L^{\infty}(\Omega)$ be a weak solution of $Lu=0$ in $\Omega$. To continue, we observe that if $Q \in \D_{Q_0}$ is so that $Q \cap E_N \neq \emptyset$, then necessarily $Q \in \D_{\F_N, Q_0}$,
otherwise, $Q \subset Q' \in \F_N$, hence $Q \subset Q_0 \backslash E_N$ which is a contradiction. As a result,
\[
\frac{\omega_0(Q)}{\omega(Q)}
\approx_N 1,
\qquad
\forall x\in E_N,\ Q\in\dd_{Q_0},\ Q\ni x.
\]
By the (dyadic) Lebesgue differentiation theorem with respect to $\omega$, along with the fact that $\omega_0\ll\omega$ (cf. \eqref{eq:normalize}), we conclude that $d\omega_0/d\omega(x)\approx_N 1$ for $\omega$-a.e.~$x\in E_N$, hence also for $\omega_0$-a.e.~$x\in E_N$. Thus,
\begin{multline*}
\int_{E_N} \mathcal{S}_{Q_0}^\vartheta u(x)^q d\w_0(x)
=
\int_{E_N} \mathcal{S}_{Q_0}^\vartheta u(x)^q \frac{d\w_0}{d\w}(x) \,d\w(x)
\approx_N
	\int_{E_N} \mathcal{S}_{Q_0}^\vartheta u(x)^q \,d\w(x)
\\
\lesssim
	\int_{Q_0} \mathcal{S}_{Q_0}^\vartheta u(x)^q \,d\w(x)
\lesssim
	\int_{Q_0} \mathcal{N}_{Q_0}^\vartheta u(x)^q \,d\w(x)
\lesssim
\|u\|_{L^\infty(\Omega)}^q \w(Q_0)
\lesssim
\|u\|_{L^\infty(\Omega)}^q
,
\end{multline*}
where in the third estimate we have used \eqref{S<N:dyadic} with $\omega_{L_0}=\omega_L$ (see also \cite[Theorem 5.3]{AHMT}) which holds since $\w_L\in A_\infty(\pom,\w_L)$. This and
\eqref{eq:ENFN} imply
\begin{align}\label{eq:square-FNj}
\mathcal{S}_{Q_0}^\vartheta u \in L^q(F_N, \w_0).
\end{align}

Now, note that for fixed $\alpha>0$, we can find $\vartheta$ sufficiently large depending on $\alpha$ such that for any $r_0 \ll 2^{-k_0}$,
\begin{equation}\label{q343ffg}
	\Gamma^{\alpha}_{r_0}(x) \subset \Gamma_{Q_0}^\vartheta(x),\qquad \forall\,x \in Q_0.
\end{equation}
Indeed, let $Y \in \Gamma^{\alpha}_{r_0}(x)$. Pick $I \in \W$ so that $I\ni Y$, hence
$\ell(I) \approx \delta(Y) \leq |Y-x|<r_0 \ll 2^{-k_0} = \ell(Q_0)$.
Pick $Q_I \in \D_{Q_0}$ such that $x \in Q_I$ and $\ell(Q_I)=\ell(I) \ll \ell(Q_0)$.
Thus,
\begin{align*}
	\dist(I, Q_I) \leq |Y-x| < (1+\alpha) \delta(Y) \leq C(1+\alpha) \ell(I) = C(1+\alpha) \ell(Q_I).
\end{align*}
Recalling \eqref{eq:WQ}, if we take $\vartheta\ge \vartheta_0$ large enough so that $	2^\vartheta \geq C(1+\alpha)$,
then $Y \in I \in \W_{Q_I}^\vartheta \subset \W^{\vartheta,*}_{Q_I}$. The latter gives that
$Y \in U_{Q_I}^\vartheta \subset \Gamma_{Q_0}^\vartheta(x)$ and consequently \eqref{q343ffg} holds.
We would like to mention that the dependence of $\vartheta$ on $\alpha$ implies that all the sawtooth
regions $\Omega_{\F_N, Q_0}^\vartheta$ above as well as all the implicit constants depend on $\alpha$.

Next, \eqref{q343ffg} readily yields that $\mathcal{S}^{\alpha}_{r_0}u(x) \leq \mathcal{S}_{Q_0}^\vartheta u(x)$ for every $x \in Q_0$. This, together with \eqref{eq:square-FNj}, implies that $\mathcal{S}^{\alpha}_{r_0} u \in L^q(F_N, \w_0)$. If we next take an arbitrary $X\in\Omega$, by Harnack's inequality (albeit with constants depending on $X$) and by \eqref{est:two-trunc}, then we have
\begin{equation}\label{eq:Sr-u-L2-FN}
	\mathcal{S}^{\alpha}_r u \in L^q(F_N, \w_{L_0}^{X}),\quad \text{for any } r>0.
\end{equation}
Note also that by \eqref{eq:WE} and Harnack's inequality
\begin{equation}\label{eq:WE:alt}
\w_{L_0}^{X}(E_0)=0.
\end{equation}

To complete the proof, we perform the preceding operation for an arbitrary $Q_0\in \D_{k_0}$. Therefore, invoking  \eqref{eq:E0N-EN}, \eqref{eq:Q-decom}, and \eqref{eq:ENFN} with $Q_k\in \D_{k_0}$,
we conclude, with the induced notation, that
\begin{multline}\label{eq:EE-FF}
	\partial \Omega = \bigcup_{Q_k \in \D_{k_0}} Q_k
=\bigg(\bigcup_{Q_k \in \D_{k_0}} E^k_0\bigg) \bigcup
	\bigg(\bigcup_{Q_k \in \D_{k_0}} \bigcup_{N>C_0} E^k_N \bigg)
\\
=\bigg(\bigcup_{Q_k \in \D_{k_0}} E^k_0\bigg) \bigcup
	\bigg(\bigcup_{Q_k \in \D_{k_0}} \bigcup_{N>C_0} F^k_N \bigg)
	=: F_0 \cup \bigg(\bigcup_{k, N} F^k_N \bigg).
\end{multline}
where $\w_{L_0}^{X}(F_0)=0$ (by \eqref{eq:WE:alt}) and $F^k_N=\partial \Omega \cap \partial \Omega_{\F^k_N, Q_k}^\vartheta$ where each $\Omega_{\F^k_N, Q_k}^\vartheta \subset \Omega$ is a bounded 1-sided NTA domain satisfying the capacity density condition. Combining \eqref{eq:EE-FF} and \eqref{eq:Sr-u-L2-FN} with $F_N^k$ in place of $F_N$, the proof of $\mathrm{(a)}\ \Longrightarrow\ \mathrm{(b)}$ is complete.  \qed

\subsection{Proof of \texorpdfstring{$\mathrm{(c)}'\ \Longrightarrow\ \mathrm{(a)}$}{(c)' implies (a)}}
Let $\alpha_0$ be so that \eqref{daydic-cone:regular-cone} holds. Suppose that $\mathrm{(c)}'$ holds where throughout it is assumed that  $\alpha\ge \alpha_0$. In order to prove that $\w_{L_0} \ll \w_L$ on $\partial \Omega$, by Lemma~\ref{lemma:dyadiccubes} and the fact that by Harnack's inequality $\w_L^X\ll \w_L^Y$ and $\w_{L_0}^X\ll\w_{L_0}^Y$ for any $X,Y\in\Omega$, it suffices to show that
for any given $Q_0 \in \D$,
\begin{align}\label{eq:absQ}
	F \subset Q_0,\quad \w_L^{X_{Q_0}}(F)=0 \quad \Longrightarrow \quad \w_{L_0}^{X_{Q_0}}(F)=0.
\end{align}
Consider then $F \subset Q_0$ with $\w_L^{X_{Q_0}}(F)=0$. Lemma~\ref{sq-function->M} applied to $F$ gives a Borel set $S\subset Q_0$ such that $u(X):=\w_L^{X}(S)$, $X\in\Omega$, satisfies
\begin{align}\label{eq:S-lower}
\mathcal{S}^{\alpha}_{r_{Q_0}^*} u(x)\ge 	\mathcal{S}^{\vartheta_0}_{Q_0,\eta} u(x) =\infty,
\qquad \forall\,x\in F,
\end{align}
where the first inequality follows from \eqref{daydic-cone:regular-cone} and the fact that $\alpha\ge \alpha_0$, and $r_{Q_0}^*=2\kappa_0 r_{Q_0}$. By assumption and \eqref{est:two-trunc} we have that $\mathcal{S}^{\alpha}_{r_{Q_0}^*} u(x)<\infty$ for $\omega_{L_0}^{X_{Q_0}}$-a.e.~$x\in \pom$. Hence, $\omega_{L_0}^{X_{Q_0}}(F)=0$ as desired and the proof of $\mathrm{(c)}'\ \Longrightarrow\ \mathrm{(a)}$ is complete. \qed

%%%%%%%%%%%%%%%%%%%%%%% SECTION SECTION SECTION %%%%%%%%%%%%%%%%%%%%%%%
%%%%%%%%%%%%%%%%%%%%%%% SECTION SECTION SECTION %%%%%%%%%%%%%%%%%%%%%%%

\section{Proof of Theorems~\ref{thm:wL} and \ref{thm:wLT}}\label{sec:perturbation}

We will obtain Theorems~\ref{thm:wL} and \ref{thm:wLT} as a consequence of the following qualitative version of \cite[Theorem~4.13]{CHMT}:

%%%%%%%%%%%%%%%%%%%%%% THEOREM THEOREM THEOREM %%%%%%%%%%%%%%%%%%%%%%
\begin{theorem}\label{thm:AAAD}
	Let $\Omega \subset \R^{n+1}$, $n \ge 2$, be a 1-sided NTA domain (cf. Definition \ref{def1.1nta})  satisfying the capacity density condition (cf. Definition \ref{def-CDC}). There exists $\widetilde{\alpha}_0>0$ (depending only on the 1-sided NTA and CDC constants) such that the following holds. Assume that $L_0 u = -\div(A_0 \nabla u)$ and $L_1 u = -\div(A_1 \nabla u)$ are real (not necessarily symmetric) elliptic operators such that $A_0-A_1=A+D$, where $A, D \in L^{\infty}(\Omega)$ are real matrices satisfying the following conditions:

	\begin{list}{\textup{(\theenumi)}}{\usecounter{enumi}\leftmargin=1.2cm \labelwidth=1.2cm \itemsep=0.2cm \topsep=.2cm \renewcommand{\theenumi}{\roman{enumi}}}

		\item  There exist $\alpha_1 \geq \widetilde{\alpha}_0$ and $r_1>0$ such that
		\begin{align}\label{eq:a(X)-delta}
			\iint_{\Gamma^{\alpha_1}_{r_1}(x)} a(X)^2 \delta(X)^{-n-1}  dX < \infty,
			\qquad \text{for $\w_{L_0}$-a.e.~} x \in \partial \Omega,
		\end{align}
		where $a(X):=\sup\limits_{Y \in B(X, \delta(X)/2)}|A(Y)|$, $X \in \Omega$.
		
		\item $D \in \Lip_{\loc}(\Omega)$ is antisymmetric and there exist $\alpha_2 \geq \widetilde{\alpha}_0$ and $r_2>0$ such that
		\begin{align}\label{eq:divCD}
			\iint_{\Gamma^{\alpha_2}_{r_2}(x)} |\div_{C} D(X)|^2 \delta(X)^{1-n} dX < \infty,
			\quad \text{for $\w_{L_0}$-a.e.~} x \in \partial \Omega.
		\end{align}
	\end{list}
	Then $\w_{L_0} \ll \w_{L_1}$.
\end{theorem}
%%%%%%%%%%%%%%%%%%%%%% THEOREM THEOREM THEOREM %%%%%%%%%%%%%%%%%%%%%%

Assuming this result momentarily we deduce Theorems \ref{thm:wL} and \ref{thm:wLT}:

%%%%%%%%%%%%%%%%%%%%%%%%% PROOF PROOF PROOF %%%%%%%%%%%%%%%%%%%%%%%%
\begin{proof}[Proof of Theorem \ref{thm:wL}]
For $L_0$ and $L$ as in the statement set $\widetilde{A}_0=A_0$, $\widetilde{A}_1=A$, $\widetilde{A}=A_0-A$, and $D=0$ so that $\widetilde{A}_0-\widetilde{A}_1=
\widetilde{A}+D$. Note that \eqref{eq:a(X)-delta} follows at once from \eqref{eq:rhoAA},  and also that \eqref{eq:divCD} holds automatically. With all these in hand Theorem \ref{thm:AAAD} gives  $\w_{L_0}= \w_{\widetilde{L}_0}\ll \w_{\widetilde{L}_1}= \w_{L}$.
\end{proof}
%%%%%%%%%%%%%%%%%%%%%%%%% PROOF PROOF PROOF %%%%%%%%%%%%%%%%%%%%%%%%

%%%%%%%%%%%%%%%%%%%%%%%%% PROOF PROOF PROOF %%%%%%%%%%%%%%%%%%%%%%%%
\begin{proof}[Proof of Theorem \ref{thm:wLT}]
	Set $A_0=A$, $A_1=A^{\top}$, $\widetilde{A}=0$, and
	$D=A-A^{\top}$ so that $A_0-A_1=\widetilde{A}+D$. Observe that $D \in \Lip_{\loc}(\Omega)$ is antisymmetric, \eqref{eq:a(X)-delta} holds trivially, and  \eqref{eq:divCD} agrees with \eqref{eq:divCAA}. Thus, Theorem \ref{thm:AAAD} implies that $\w_L \ll \w_{L^{\top}}$.
	
	On the other hand, $\w_L \ll \w_{L^{\rm sym}}$ follows similarly if we set $A_0=A$, $A_1=(A+A^{\top})/2$, $\widetilde{A}=0$ and $D=(A-A^{\top})/2$.
	
	Finally, $\w_{L^{\top}} \ll \w_L$ follows from what has been proved by switching the roles of $L$ and $L^\top$ and the fact that $\mathscr{F}^{\alpha}_r(x; A)<\infty$ for $\omega_{L^\top}$-a.e.~$x\in\pom$.	
\end{proof}
%%%%%%%%%%%%%%%%%%%%%%%%% PROOF PROOF PROOF %%%%%%%%%%%%%%%%%%%%%%%%

Before proving Theorem~\ref{thm:AAAD} we need the following auxiliary which adapts \cite[Lemma~4.44]{HMT1} and \cite[Lemma~3.11]{AHMT} to our current setting:

\begin{lemma}\label{lemma:approx-sawtooth}
Let $\Omega\subset\mathbb{R}^{n+1}$ be a 1-sided NTA domain (cf. Definition \ref{def1.1nta})  satisfying the capacity density condition  (cf. Definition \ref{def-CDC}).	Given $Q_0\in\dd$, a pairwise disjoint collection $\F\subset\dd_{Q_0}$, and $N\ge 4$ let
$\F_N$ be the family of maximal cubes of the collection ${\F}$ augmented by adding all the cubes $Q \in \D_{Q_0}$ such that $\ell(Q) \leq 2^{-N} \ell(Q_0)$.
There exist $\Psi_N^\vartheta\in \mathscr{C}_c^\infty(\ree)$ and a constant $C\ge 1$ depending only on dimension $n$, the 1-sided NTA constants, the CDC constant, and $\vartheta$, but independent of $N$, $\F$, and $Q_0$ such that the following hold:
	
	\begin{list}{$(\theenumi)$}{\usecounter{enumi}\leftmargin=.8cm
			\labelwidth=.8cm\itemsep=0.2cm\topsep=.1cm
			\renewcommand{\theenumi}{\roman{enumi}}}
		
		\item $C^{-1}\,\mathbf{1}_{\Omega_{\F_N,Q_0}^\vartheta}\le \Psi_N^\vartheta \le \mathbf{1}_{\Omega_{\F_N,Q_0}^{\vartheta,*}}$.
		
		\item $\sup_{X\in \Omega} |\nabla \Psi_N^\vartheta(X)|\,\delta(X)\le C$.
		
		\item Setting
		\begin{equation}
			\label{eq:defi-WN}
			\W_N^\vartheta:=\bigcup_{Q\in\dd_{\F_{N},Q_0}} \W_Q^{\vartheta,*},
			\qquad
			\W_N^{\vartheta,\Sigma}:=
			\big\{I\in \W_{N}^\vartheta:\, \exists\,J\in \W\setminus \W_N^\vartheta\ \mbox{with}\ \partial I\cap\partial J\neq\emptyset
			\big\},
		\end{equation}
	\end{list}
	one has
	\begin{equation}
		\nabla \Psi_N^\vartheta\equiv 0
		\quad
		\mbox{in}
		\quad
		\bigcup_{I\in \W_N^\vartheta \setminus \W_N^{\vartheta,\Sigma} }I^{**},
		\label{eq:fregtgtr}
	\end{equation}
	and there exists a family $\{\widehat{Q}_I\}_{I\in 	\W_N^{\vartheta,\Sigma}}$ so that
	\begin{equation}\label{new-QI}
		C^{-1}\,\ell(I)\le \ell(\widehat{Q}_I)\le C\,\ell(I),
		\qquad
		\dist(I, \widehat{Q}_I)\le C\,\ell(I),
		\qquad
		\sum_{I\in 	\W_N^{\vartheta,\Sigma}} \mathbf{1}_{\widehat{Q}_I} \le C.
	\end{equation}
\end{lemma}

\begin{proof}
The proof combines ideas from \cite[Lemma~4.44]{HMT1},  \cite[Lemma~3.11]{AHMT}, and  \cite[Appendix~A.2]{HMM-UR}. 	The parameter $\vartheta\ge \vartheta_0$ will remain fixed in the proof and then constants are allowed to depend on it. To ease the notation we will omit the superscript $\vartheta$ everywhere in the proof.
Recall that given $I$, any closed dyadic cube in $\ree$, we set $I^{*}=(1+\lambda)I$ and $I^{**}=(1+2\,\lambda)I$. Let us introduce $\widetilde{I^{*}}=(1+\frac32\,\lambda)I$ so that
\begin{equation}
	I^{*}
	\subsetneq
	\interior(\widetilde{I^{*}})
	\subsetneq \widetilde{I^{*}}
	\subset
	\interior(I^{**}).
	\label{eq:56y6y6}
\end{equation}
Given $I_0:=[-\frac12,\frac12]^{n+1}\subset\ree$, fix $\phi_0\in \mathscr{C}_c^\infty(\ree)$ such that
$1_{I_0^{*}}\le \phi_0\le 1_{\widetilde{I_0^{*}}}$ and $|\nabla \phi_0|\lesssim 1$ (the implicit constant depends on the parameter $\lambda$). For every $I\in \W=\W(\Omega)$ we set $\phi_I(\cdot)=\phi_0\big(\frac{\,\cdot\,-X(I)}{\ell(I)}\big)$ so that $\phi_I\in \mathscr{C}^\infty(\ree)$, $1_{I^{*}}\le \phi_I\le 1_{\widetilde{I^{*}}}$ and $
|\nabla \phi_I|\lesssim \ell(I)^{-1}$ (with implicit constant depending only on $n$ and $\lambda$).

For every $X\in\Omega$, we let $\Phi(X):=\sum_{I\in \W} \phi_I(X)$. It then follows that $\Phi\in \mathscr{C}^\infty(\Omega)$ since for every compact subset of $\Omega$, the previous sum has finitely many non-vanishing terms. Also, $1\le \Phi(X)\le C_{\lambda}$ for every $X\in \Omega$ since the family $\{\widetilde{I^{*}}\}_{I\in \W}$ has bounded overlap by our choice of $\lambda$. Hence we
can set $\Phi_I=\phi_I/\Phi$ and one can easily see that $\Phi_I\in \mathscr{C}_c^\infty(\ree)$, $C_\lambda^{-1}1_{I^{*}}\le \Phi_I\le 1_{\widetilde{I^{*}}}$ and $
|\nabla \Phi_I|\lesssim \ell(I)^{-1}$. With this at hand, set
\[
\Psi_N(X)
:=
\sum_{I\in \W_N} \Phi_I(X)
=
\frac{\sum\limits_{I\in \W_N} \phi_I(X)}{\sum\limits_{I\in \W} \phi_I(X)},
\qquad
X\in\Omega.
\]
We first note that the number of terms in the sum defining $\Psi_N$ is bounded depending on $N$. Indeed, if $Q\in \dd_{\F_N, Q_0}$ then $Q\in \dd_{Q_0}$ and $2^{-N}\ell(Q_0)<\ell(Q)\le \ell(Q_0)$ which implies that $\dd_{\F_N, Q_0}$ has finite cardinality with bounds depending on dimension and  $N$ (here we recall that the number of dyadic children of a given cube is uniformly controlled). Also, by construction $\W_Q^*$ has cardinality depending only on the allowable parameters. Hence, $\# \W_N\lesssim C_N<\infty$. This and the fact that each $\Phi_I\in \mathscr{C}_c^\infty(\ree)$ yield that $\Psi_N\in \mathscr{C}_c^\infty(\ree)$. Note also that \eqref{eq:56y6y6} and the definition of $\W_N$  give
\begin{align*}
	\supp \Psi_N
	\subset
	\bigcup_{I\in \W_N} \widetilde{I^{*}}
	=
	\bigcup_{Q\in\dd_{\F_{N},{Q}_0}}
	\bigcup_{I\in \W_Q^*} \widetilde{I^{*}}
	\subset
	\interior\Big(
	\bigcup_{Q\in\dd_{\F_{N},{Q}_0}}
	\bigcup_{I\in \W_Q^*} I^{**}
	\Big)
	=
	\interior\Big(
	\bigcup_{Q\in\dd_{\F_{N},{Q}_0}}
	U_Q^{*}
	\Big)
	=
	\Omega_{\F_N,Q_0}^{*}.
\end{align*}
This, the fact that $\W_N\subset \W$, and the definition of $\Psi_N$ immediately give that
$\Psi_N\le \mathbf{1}_{\Omega_{\F_N,Q_0}^{*}}$. On the other hand, if $X\in \Omega_N=\Omega_{\F_{N},Q_0}$, then there exists $I\in \W_N$ such that $X\in I^{*}$, in which case $\Psi_N(X)\ge \Phi_I(X)\ge C_\lambda^{-1}$. All these imply $(i)$.
Note that $(ii)$  follows by observing that for every $X\in \Omega$ we have
$$
|\nabla \Psi_N(X)|
\le
\sum_{I\in \W_N} |\nabla\Phi_I(X)|
\lesssim
\sum_{I\in \W} \ell(I)^{-1}\,1_{\widetilde{I^{*}}}(X)
\lesssim
\delta(X)^{-1},
$$
where we have used that if $X\in \widetilde{I^{*}}$ then $\delta(X)\approx \ell(I)$ and also that the family $\{\widetilde{I^{*}}\}_{I\in \W}$ has bounded overlap.

To see $(iii)$ fix $I\in\W_N\setminus \W^{\Sigma}_N$ and  $X\in I^{**}$, and set $\W_X:=\{J\in \W: \phi_J(X)\neq 0\}$ so that $I\in \W_X$. We first note that  $\W_X\subset \W_N$. Indeed, if $\phi_J(X)\neq 0$ then $X\in \widetilde{J^{*}}$.
Hence $X\in I^{**}\cap J^{**}$ and our choice of $\lambda$ gives that $\partial I$ meets $\partial J$, this in turn implies that $J\in \W_N$ since $I\in\W_N\setminus \W^{\Sigma}_N$. All these yield
$$
\Psi_N(X)
=
\frac{\sum\limits_{J\in \W_N} \phi_J(X)}{\sum\limits_{J\in \W} \phi_J(X)}
=
\frac{\sum\limits_{J\in \W_N\cap \W_X} \phi_J(X)}{\sum\limits_{J\in \W_X} \phi_J(X)}
=
\frac{\sum\limits_{J\in \W_N\cap \W_X} \phi_J(X)}{\sum\limits_{J\in \W_N\cap \W_X} \phi_J(X)}
=
1.
$$
Hence $\Psi_N\big|_{I^{**}}\equiv 1$  for every $I\in\W_N\setminus \W^{\Sigma}_N$. This and the fact that $\Psi_N\in \mathscr{C}_c^\infty(\ree)$ immediately give that $\nabla \Psi_N\equiv 0$ in $\bigcup_{I\in \W_N \setminus \W_N^\Sigma }I^{**}$.

We are left with showing the last part of $(iii)$ and for that we borrow some ideas from \cite[Appendix A.2]{HMM-UR}. Fix $I\in  \W_N^\Sigma$ and let $J$ be so that $J\in \W\setminus \W_N$ with $\partial I\cap\partial J\neq\emptyset$, in particular $\ell(I)\approx\ell(J)$. Since $I\in  \W_N^\Sigma$ there exists $Q_I\in\dd_{\F_N,Q_0}$. Pick $Q_J\in\dd$ so that $\ell(Q_J)=\ell(J)$ and it contains any fixed $\widehat{y}\in\partial\Omega$ such that $\dist(J,\partial\Omega)=\dist(J,\widehat{y})$. Then, as observed in Section \ref{section:dyadic}, one has $J\in \W_{Q_J}^*$.
But, since $J\in \W\setminus \W_N$, we necessarily have $Q_J\notin \dd_{\F_N,Q_0}=\dd_{\F_N}\cap \dd_{Q_0}$. Hence, $\W_N^\Sigma=\W_N^{\Sigma,1}\cup \W_N^{\Sigma,2}\cup \W_N^{\Sigma,3}$ where
\begin{align*}
	\W_N^{\Sigma,1}:&=\{I\in \W_N^{\Sigma}: Q_0\subset Q_J\},
	\\
	\W_N^{\Sigma,2}:&=\{I\in \W_N^{\Sigma}: \ Q_J\subset Q\in\F_N\},
	\\
	\W_N^{\Sigma,3}:&=\{I\in \W_N^{\Sigma}: Q_J\cap Q_0=\emptyset\}.
\end{align*}
For later use it is convenient to observe that
\begin{equation}\label{dist:QJ-I}
	\dist(Q_J,I)\le \dist(Q_J,J) +\diam(J)+\diam(I)\approx\ell(J)+\ell(I)\approx\ell(I).
\end{equation}

Let us first consider $\W_N^{\Sigma,1}$. If $I\in \W_N^{\Sigma,1}$ we clearly have
\[
\ell(Q_0)\le \ell(Q_J)=\ell(J)\approx\ell(I)\approx\ell(Q_I)\le \ell(Q_0)
\]
and since $Q_I\in\dd_{Q_0}$
\[
\dist(I,x_{Q_0})\le
\dist(I, Q_I)+\diam(Q_0)
\approx \ell(I).
\]
In particular, $\# \W_N^{\Sigma,1}\lesssim 1$. Thus if we set $\widehat{Q}_I:=Q_J$ it follows from \eqref{dist:QJ-I} that the two first conditions in \eqref{new-QI} hold and also $\sum_{I\in 	\W_N^{\Sigma,1}} \mathbf{1}_{\widehat{Q}_I} \le
\# \W_N^{\Sigma,1}\lesssim 1$.

To see that \eqref{new-QI} holds for $\W_N^{\Sigma,2}$ and $\W_N^{\Sigma,3}$, we proceed as follows. For any $I\in \W_N^{\Sigma,2}\cup \W_N^{\Sigma,3}$  we pick $\widehat{Q}_I\in\dd$ so that $ \widehat{Q}_I\ni x_{Q_J}$ and $\ell(\widehat{Q}_I)=2^{-M'}\,\ell(Q_J)$ with $M'\ge 3$ large enough so that $2^{M'}\ge 2 \Xi^2$ (cf. \eqref{deltaQ}). Note that $\widehat{Q}_I\subset \Delta_{Q_J}\subset Q_J$ which, together with \eqref{dist:QJ-I}, implies
\begin{equation}\label{3434t53qt5}
	\dist(I,\widehat{Q}_I)
	\le
	\dist(I,Q_J)
	+\diam(Q_J)
	\lesssim
	\ell(I)
\end{equation}
and
\begin{equation}\label{4q43tfg3f:1}
	\diam(\widehat{Q}_I)
	\le
	2\,\Xi\,r_{\widehat{Q}_I}
	\le
	2\,\Xi\,\ell(\widehat{Q}_I)
	=
	2^{-M'+1}\,\Xi\,\ell(Q_J)
	\le
	\Xi^{-1}\,\ell(Q_J).
\end{equation}
Hence, the first two conditions in \eqref{new-QI} hold for $I\in \W_N^{\Sigma,2}\cup \W_N^{\Sigma,3}$.

To see that the last condition in \eqref{new-QI} holds, we start with the family $\W_N^{\Sigma,2}$. For any $I\in \W_N^{\Sigma,2}$  there is a unique $Q_j\in \F_N$ such that $Q_J\subset Q_j$. But,  since $Q_I\in\dd_{\F_N,Q_0}$ then necessarily $Q_I\not\subset Q_j$ and $Q_I\setminus Q_j\neq\emptyset$.
This and the fact that $2\Delta_{Q_J}\subset Q_J\subset Q_j$ imply
\begin{multline*}
	2\Xi^{-1}\,\ell(Q_J)
	\le
	 \dist(x_{Q_J}, \pom\setminus  Q_j)
	\le
	\dist(x_{Q_J}, Q_I\setminus Q_j)
	\\
	\le
	\diam(Q_J)+\dist(Q_J, J)+\diam(J)
	+\diam(I)+\dist(I,Q_I)+\diam(Q_I)
	\approx
	\ell(J)\approx\ell(I).
\end{multline*}
Thus, $2\,\Xi^{-1}\,\ell(Q_J)\le \dist(x_{Q_J}, \pom \setminus Q_j) \le C\,\ell(J)$.  Suppose next that $I, I'\in \W_N^{\Sigma,2}$ are so that $\widehat{Q}_I\cap\widehat{Q}_{I'}\neq\emptyset$ (it could even happen that they are indeed the same cube) and assume without loss of generality that $\widehat{Q}_{I'}\subset\widehat{Q}_{I}$, hence $\ell(I')\le \ell(I)$. Let $Q_j, Q_{j'}\in\F_N$ be so that $Q_J\subset Q_j$ and $Q_{J'}\subset Q_{j'}$. Then, $x_{Q_J}\in \widehat{Q}_I$ and $x_{Q_{J'}}\in \widehat{Q}_{I'} \subset \widehat{Q}_I\subset Q_J$. As a consequence, $x_{Q_{J'}}\in Q_{J'}\cap Q_J\subset Q_{j}\cap Q_j'$ and this forces $Q_j=Q_{j'}$ (since $\F_N$ is a pairwise disjoint family). This and
\eqref{4q43tfg3f:1} readily imply
\begin{multline*}
	2\,\Xi^{-1}\,\ell(Q_{J})
	\le
	 \dist(x_{Q_J}, \pom\setminus  Q_j)
	\le
	|x_{Q_{J}}- x_{Q_{J'}}|+	\dist(x_{Q_{J'}}, \pom\setminus  Q_j)
	\\
	\le
	\diam(\widehat{Q}_{I})+\dist(x_{Q_{J'}}, \pom\setminus  Q_{j'})
	\le
	\diam(\widehat{Q}_{I})+C\ell(J')
	\le
	\Xi^{-1}\,\ell(Q_J)+C\ell(J')
\end{multline*}
and therefore $\Xi^{-1}\,\ell(Q_{J})\le C\,\ell(J')$. This in turn gives $\ell(I)\approx \ell(J)\approx \ell(J')\approx\ell(I')$. Note also that since $I$ touches $J$, $I'$ touches $J'$, and $\widehat{Q}_I\cap\widehat{Q}_{I'}\neq\emptyset$ we obtain
\begin{multline*}
	\dist(I,I')
	\le
	\diam(J)+\dist(J, Q_J)+\diam(Q_J)+\diam(Q_{J'})
	\\+\dist(Q_{J'}, J')+\diam(J')
	\approx
	\ell(J)+\ell(J')\approx\ell(I).
\end{multline*}
As a result, for fixed  $I\in \W_N^{\Sigma,2}$ there is a uniformly bounded number of $I'\in \W_N^{\Sigma,2}$ with $\widehat{Q}_I\cap \widehat{Q}_{I'}\neq\emptyset$, thus $\sum_{I\in 	\W_N^{\Sigma,2}} \mathbf{1}_{\widehat{Q}_I} \lesssim 1$.

We finally take into consideration  $\W_N^{\Sigma,3}$. Let $I\in\W_N^{\Sigma,3}$. Then,  $Q_0\cap2\Delta_{Q_J}\subset Q_0\cap Q_J=\emptyset$ and therefore $2\Xi^{-1}\,\ell(Q_J)\le \dist(x_{Q_J}, Q_0)$. Besides, since $Q_I\subset Q_0$, we have
\begin{align*}
	\dist(x_{Q_J}, Q_0)
	\le
	\diam(Q_J)+\dist(Q_J, J)+\diam(J)
		+\diam(I)+\dist(I,Q_I)+\diam(Q_I)
	\approx
	\ell(J)\approx\ell(I).
\end{align*}
Thus, $2\,\Xi^{-1}\,\ell(Q_J)\le \dist(x_{Q_J}, Q_0) \le C\,\ell(J)$. Suppose next that $I, I'\in \W_N^{\Sigma,3}$ are so that $\widehat{Q}_I\cap\widehat{Q}_{I'}\neq\emptyset$ (it could even happen that they are indeed the same cube) and assume without loss of generality that $\widehat{Q}_{I'}\subset\widehat{Q}_{I}$, hence $\ell(J')\le \ell(J)$. Then, since
$x_{Q_J}\in \widehat{Q}_I$ and $x_{Q_{J'}}\in \widehat{Q}_{I'} \subset \widehat{Q}_I$ we get from \eqref{4q43tfg3f:1} that
\begin{multline*}
	2\,\Xi^{-1}\,\ell(Q_{J})
	\le
	\dist(x_{Q_{J}}, Q_0)
	\le
	|x_{Q_{J}}- x_{Q_{J'}}|+
	\dist(x_{Q_{J'}}, Q_0)
	\\
	\le
	\diam(\widehat{Q}_{I})+C\ell(J')
	\le
	\Xi^{-1}\,\ell(Q_J)+C\ell(J'),
\end{multline*}
and therefore $\Xi^{-1}\,\ell(Q_{J})\le C\,\ell(J')$. This yields $\ell(I)\approx \ell(J)\approx \ell(J')\approx\ell(I')$. Note also that since $I$ touches $J$, $I'$ touches $J'$, and $\widehat{Q}_I\cap\widehat{Q}_{I'}\neq\emptyset$ we obtain
\begin{multline*}
	\dist(I,I')
	\le
	\diam(J)+\dist(J, Q_J)+\diam(Q_J)+\diam(Q_{J'})
	\\+\dist(Q_{J'}, J')+\diam(J')
	\approx
	\ell(J)+\ell(J')\approx\ell(I).
\end{multline*}
Consequently, for fixed  $I\in \W_N^{\Sigma,3}$ there is a uniformly bounded number of $I'\in \W_N^{\Sigma,3}$ with $\widehat{Q}_I\cap \widehat{Q}_{I'}\neq\emptyset$. As a result, $\sum_{I\in 	\W_N^{\Sigma,3}} \mathbf{1}_{\widehat{Q}_I} \lesssim 1$.  This completes the proof of $(iii)$ and hence that of Lemma \ref{lemma:approx-sawtooth}.
\end{proof}

We are now ready to prove Theorem~\ref{thm:AAAD}.

%%%%%%%%%%%%%%%%%%%%%%%%% PROOF PROOF PROOF %%%%%%%%%%%%%%%%%%%%%%%%%
\begin{proof}[Proof of Theorem~\ref{thm:AAAD}]
	We use some ideas from \cite[Section 4]{CHMT} and \cite[Section~4]{CMO}. 	Let $u \in W_{\loc}^{1,2}(\Omega) \cap L^{\infty}(\Omega)$ be a weak solution of $L_1u=0$ in $\Omega$ and assume that $\|u\|_{L^{\infty}(\Omega)}=1$. Applying Theorem \ref{thm:abs} $\mathrm{(c)} \Longrightarrow \mathrm{(a)}$ to $u$,  we are reduced to showing that for some $r>0$,
	\begin{align*}
		\mathcal{S}^{\alpha_0}_r u(x) < \infty, \qquad \text{for $\w_{L_0}$-a.e.~} x \in \partial \Omega,
	\end{align*}
	where $\alpha_0$ is given in Theorem \ref{thm:abs}. By \eqref{q343ffg} and Lemma~\ref{lemma:dyadiccubes}, it suffices to see that for every fixed $Q_0 \in \D_{k_0}$ and for some fixed large $\vartheta$ (which depends on $\alpha_0$ and hence solely on the 1-sided NTA and CDC constants) one has
	\begin{align}\label{eq:Q0-SQ0u-EN}
		Q_0 = \bigcup_{N \geq 0} \widehat{E}_N,\quad \w_{L_0}^{X_0}(\widehat{E}_0)=0
		\quad\text{and}\quad \mathcal{S}_{Q_0}^\vartheta u \in L^2(\widehat{E}_N,\w_{L_0}), \ \forall\,N \geq 1,
	\end{align}
	where $X_0$ is given at the beginning of Section \ref{sec:local}.
	Fix then $Q_0 \in \D_{k_0}$ and write
	\begin{align}\label{eq:normalize:alt}
		\w_0:=\w_{L_0}^{X_0},\qquad \w:= \w_{L_1}^{X_0}, \qquad  \G_0:= G_{L_0}(X_0, \cdot), \qquad\text{ and }\qquad
		\G:= G_{L_1}(X_0, \cdot).
	\end{align}
	Much as in \eqref{daydic-cone:regular-cone} (with $\eta=2^{-1/3}$ so that $\Gamma^{\vartheta,*}_{Q_0}=\Gamma^{\vartheta,*}_{Q_0,\eta }$) there exist $\widetilde{\alpha}_0>0$ and $C$ (depending on the 1-sided NTA and CDC constants) such that if we set $\widetilde{r}:=C\,r_{Q_0}>0$, then
	\begin{equation}\label{eq:TQ-Tar}
		\Gamma^{\vartheta,*}_{Q_0}(x)
		:=\bigcup_{x \in Q \in \D_{Q_0}} U^{\vartheta, *}_{Q}
		\subset \Gamma^{\widetilde{\alpha}_0}_{\widetilde{r}}(x),\qquad x \in Q_0.
	\end{equation}
	As a result,
	\begin{multline}\label{eq:SQ0-a.e.}
		\mathcal{S}_{Q_0}\gamma^\vartheta(x)^2 :=
		\sum_{x \in Q \in \D_{Q_0}} \gamma_{Q} ^\vartheta
		:=
		\sum_{x \in Q \in \D_{Q_0}}  \iint_{U^{\vartheta, *}_Q} a(X)^2 \delta(X)^{-n-1}  dX
		+ \iint_{U^{\vartheta, *}_Q} |\div_{C} D(X)|^2 \delta(X)^{1-n} dX
		\\
		\lesssim
		\iint_{\Gamma^{\vartheta,*}_{Q_0}(x) } a(X)^2 \delta(X)^{-n-1}  dX
		+ \iint_{\Gamma^{\vartheta,*}_{Q_0}(x) } |\div_{C} D(X)|^2 \delta(X)^{1-n} dX
		\\\le  \iint_{\Gamma^{{\alpha}_1}_{\max\{\widetilde{r}, r_1\}}(x)} a(X)^2 \delta^{-n-1}  dX
		+ \iint_{\Gamma^{{\alpha}_2}_{\max\{\widetilde{r},r_2\}}(x)} |\div_{C} D(X)|^2 \delta^{1-n} dX
		< \infty,
	\end{multline}
	for $\w_{L_0} $-a.e.~ $x \in  Q_0$, 	where we have used the fact that the family $\{U^{\vartheta,*}_Q\}_{Q \in \D}$ has bounded overlap, that  $\alpha_1, \alpha_2\ge \widetilde{\alpha}_0$ and the last estimate follows from \eqref{eq:a(X)-delta}, \eqref{eq:divCD}, and \eqref{est:two-trunc}.
	
	Given $N>C_0$ ($C_0$ is the constant that appeared in Section \ref{sec:local}), let $\F_N\subset \D_{Q_0}$ be the collection of maximal cubes (with respect to the inclusion) $Q_j \in \D_{Q_0}$ such that
	\begin{align}\label{eq:FN-stopping}
		\sum_{Q_j \subset Q \in \D_{Q_0}} \gamma_{Q}^\vartheta > N^2.
	\end{align}
	Write
	\begin{align}\label{eq:q4we3g5}
		{E}_0 :=\bigcap_{N>C_0} (Q_0\setminus E_N),
		\quad
		E_N:=  Q_0 \backslash \bigcup_{Q_j \in {\F}_N}Q_j, \quad Q_0={E}_0\cup (Q_0\setminus E_0)
		=
		{E}_0\cup \Big(\bigcup_{N>C_0} {E}_N \Big).
	\end{align}
	Let us observe that
	\begin{equation}\label{eq:Sa<N}
		\mathcal{S}_{Q_0}\gamma^\vartheta(x) \leq N, \qquad \forall\,x \in E_N.
	\end{equation}
	Otherwise, there exists a cube $Q_x\ni x$ such that
	$\sum_{Q_x \subset Q \in  \D_{Q_0}} \gamma_Q^\vartheta > N^2$, hence $x \in Q_x \subset Q_j$ for some
	$Q_j \in {\F}_N$, which is a contradiction.

Note that if $x \in {E}_0$, then for every $N>C_0$ there exists $Q_x^{N} \in {\F}_N$ such that
$Q_x^N \ni x$. By the definition of ${\F}_N$, we then have
\begin{align*}
	\mathcal{S}_{Q_0}\gamma^\vartheta(x)^2 = \sum_{x \in Q \in \D_{Q_0}} \gamma_Q^\vartheta
	\geq  \sum_{Q_x^N \subset Q \in \D_{Q_0}} \gamma_Q^\vartheta > N^2.
\end{align*}
On the other hand, if $x \in Q_0\setminus {E}_{N+1}$ there exists $Q_x \in {\F}_{N+1}$ such that $x \in Q_x$. By \eqref{eq:FN-stopping} one has
	\begin{align*}
		\sum_{Q_x \subset Q \in \D_{Q_0}} \gamma_{Q}^\vartheta > (N+1)^2>N^2,
	\end{align*}
	and the maximality of the cubes in ${\F}_N$ gives that $Q_x \subset Q'_x$ for some $Q'_x \in {\F}_N$ with $x \in Q'_x \subset Q_0\setminus {E}_{N}$. This shows that $\{Q_0\setminus {E}_{N}\}_N$ is a decreasing sequence of sets, and since $Q_0\setminus {E}_{N}\subset Q_0$ for every $N$ we conclude that
		\begin{multline}\label{eaf4a3f4}
		\w_0({E}_0 )
		=\lim_{N\to\infty} \w_0(Q_0\setminus {E}_{N})
		\le \lim_{N\to\infty}  \w_0(\{x \in Q_0: \mathcal{S}_{Q_0}\gamma^\vartheta(x)>N\})
		\\
		=\w_0(\{x \in Q_0: \mathcal{S}_{Q_0}\gamma^\vartheta(x)=\infty\})=0,
	\end{multline}
	where the last equality uses \eqref{eq:SQ0-a.e.}.  This and \eqref{eq:q4we3g5} imply that to get \eqref{eq:Q0-SQ0u-EN} we are left with proving
	\begin{equation}\label{eq:SQ0u-EN-tilde}
		\mathcal{S}_{Q_0}^\vartheta u \in L^2({E}_N, \w_0), \qquad \forall\,N>C_0.
	\end{equation}
	With this goal in mind,  note that if $Q \in \D_{Q_0}$ is so that $Q \cap {E}_N \neq \emptyset$, then necessarily $Q \in \D_{{\F}_N, Q_0}$, 	otherwise, $Q \subset Q' \in {\F}_N$, hence $Q \subset Q_0 \backslash {E}_N$. Recalling \eqref{eq:normalize:alt} and the fact $X_0 \not\in 4B_{Q_0}^*$, we use Lemma~\ref{lemma:proppde} and Harnack's inequality to conclude that
	\begin{align}\label{eq:SQk-ENk}
		\int_{E_N}\mathcal{S}_{Q_0}^\vartheta u(x)^2 d\w_0(x)
		&=\int_{E_N} \iint_{\bigcup\limits_{x \in Q \in \D_{Q_0}} U_Q^\vartheta}
		|\nabla u(Y)|^2 \delta(Y)^{1-n} dY d\w_0(x)
		\\%%%%%%%%%%
		&\lesssim \sum_{Q \in \D_{Q_0}} \ell(Q)^{1-n} \w_0(Q \cap E_N)
		\iint_{U_Q^\vartheta}  |\nabla u(Y)|^2 dY
		\nonumber \\
		&
		\le
		\sum_{Q \in \D_{\F_N, Q_0}}  \ell(Q)^{1-n} \w_0(Q)
		\iint_{U_Q^\vartheta}  |\nabla u(Y)|^2 dY
		\nonumber \\%%%%%%%%%%
		&\lesssim \sum_{Q \in \D_{\F_N, Q_0}}
		\iint_{U_Q^\vartheta}  |\nabla u(Y)|^2 \G_0(Y) dY
		\nonumber \\
			&\lesssim \iint_{\Omega_{\F_N, Q_0}^\vartheta}  |\nabla u(Y)|^2 \G_0(Y) dY, \nonumber
	\end{align}
	where we have used that the family $\{U_Q^\vartheta\}_{Q\in\D}$ has bounded overlap. To estimate the last term we make the following claim
		\begin{align}\label{eq:KAB*}
			\iint_{\Omega_{\F_N, Q_0}^\vartheta}  |\nabla u(Y)|^2 \G_0(Y) \, dY
			\lesssim \w_0(Q_0) + \sum_{Q \in \D_{\F_N, Q_0}} \gamma_Q^{\vartheta} \, \w_0(Q),
		\end{align}
		where the implicit constant is independent of $N$.

Assuming this momentarily we note that
	\begin{multline}\label{eq:dis}
		\sum_{Q \in \D_{{\F}_N, Q_0}} \gamma_Q^{\vartheta} \, \w_0(Q)
		=\int_{Q_0} \sum_{x \in Q \in \D_{{\F}_N, Q_0}} \gamma_Q^\vartheta \, d\w_0(x)
		\\%%%%%%%%%
		\leq \int_{{E}_N} \mathcal{S}_{Q_0} \gamma^\vartheta(x)^2 \, d\w_0(x)
		+ \sum_{Q_j \in {\F}_N} \sum_{Q \in \D_{{\F}_N, Q_0}} \gamma_Q^\vartheta\,\w_0(Q\cap Q_j)
		\\
		\le
		N^2 \w_0(Q_0)+ \sum_{Q_j \in {\F}_N} \sum_{Q \in \D_{{\F}_N, Q_0}} \gamma_Q^\vartheta\,\w_0(Q\cap Q_j),
	\end{multline}
where the last estimate follows from \eqref{eq:Sa<N}. In order to control the last term  we fix $Q_j \in {\F}_N$. Note that if
	$Q \in \D_{{\F}_N, Q_0}$ is so that
	$Q \cap Q_j \neq \emptyset$ then necessarily $Q_j \subsetneq Q\subset Q_0$.
	Write $\widehat{Q}_j$ for the dyadic parent of $Q_j$, that is,  $\widehat{Q}_j$ is the unique dyadic cube containing $Q_j$
	with $\ell(\widehat{Q}_j)=2\ell(Q_j)$. By the fact that $Q_j$ is the maximal cube so that \eqref{eq:FN-stopping} holds one obtains
	\begin{equation*}%\label{eq:QjQ-N2}
		\sum_{\widehat{Q}_j \subset Q \in \D_{Q_0}} \gamma_Q^\vartheta  =\sum_{Q_j \subsetneq Q \in \D_{Q_0}} \gamma_Q^\vartheta \leq N^2.
	\end{equation*}
	As a result,
	\begin{multline}\label{43r43r4}
		\sum_{Q_j \in {\F}_N} \sum_{Q \in \D_{{\F}_N, Q_0}} \gamma_Q^{\vartheta} \, \w_0(Q \cap Q_j)
		=\sum_{Q_j \in {\F}_N} \w_0(Q_j)  \sum_{Q_j \subsetneq Q \in \D_{Q_0}} \gamma_Q^\vartheta
		\\
		\leq N^2 \sum_{Q_j \in {\F}_N} \w_0(Q_j)
		\leq N^2 \w_0 \bigg(\bigcup_{Q_j \in {\F}_N} Q_j \bigg)
		\leq N^2 \w_0(Q_0).
	\end{multline}
	Collecting \eqref{eq:SQk-ENk}, \eqref{eq:KAB*}, \eqref{eq:dis},  and \eqref{43r43r4}, we deduce that
	\begin{align*}
		\int_{E_N} (\mathcal{S}_{Q_0}^\vartheta u(x))^2 \, d\w_0(x) \le C_N \, \w_0(Q_0)  \le C_N.
	\end{align*}
	This shows \eqref{eq:SQ0u-EN-tilde} and completes the proof of Theorem \ref{thm:AAAD} modulo proving \eqref{eq:KAB*}.

Let us then establish \eqref{eq:KAB*}.  For every $M \geq 4$, we consider the pairwise disjoint collection ${\F}_{N,M}$ given by the family of maximal cubes of the collection ${\F}_N$ augmented by adding all the cubes $Q \in \D_{Q_0}$ such that $\ell(Q) \leq 2^{-M} \ell(Q_0)$. In particular, $Q \in \D_{{\F}_{N,M}, Q_0}$
if and only if $Q \in \D_{{\F}_N, Q_0}$ and $\ell(Q)>2^{-M}\ell(Q_0)$.  Moreover,
$\D_{{\F}_{N,M}, Q_0} \subset \D_{{\F}_{N,M'}, Q_0}$ for all $M \leq M'$, and hence
$\Omega_{{\F}_{N,M}, Q_0}^\vartheta \subset \Omega_{{\F}_{N,M'}, Q_0}^\vartheta \subset \Omega_{{\F}_N, Q_0}^\vartheta$.  Then the monotone convergence theorem implies
\begin{align}\label{eq:JM-lim}
	\iint_{\Omega_{{\F}_{N}, Q_0}^\vartheta}  |\nabla u|^2 \G_0 \ dX
	=\lim_{M \to \infty}\iint_{\Omega_{{\F}_{N,M}, Q_0}^\vartheta}  |\nabla u|^2 \G_0 \, dX
	=:
	\lim_{M \to \infty}	\mathcal{K}_{N, M} .
\end{align}

Write $\mathcal{E}(X):=A_1(X)-A_0(X)$ and pick $\Psi_{N,M}$ from Lemma~\ref{lemma:approx-sawtooth}. By Leibniz's rule,
\begin{multline}\label{eq:LeibAA}
	A_1 \nabla u \cdot \nabla u \  \G_0 \Psi_{N,M}^2
	=A_1 \nabla u \cdot \nabla (u \G_0 \Psi_{N,M}^2)
	- \frac12 A_0 \nabla (u^2 \Psi_{N,M}^2) \cdot \nabla \G_0
	\\ 	
	+ \frac12 A_0 \nabla(\Psi_{N,M}^2) \cdot \nabla \G_0 \ u^2
	-\frac12 A_0 \nabla (u^2) \cdot \nabla(\Psi_{N,M}^2) \G_0
	- \frac12 \mathcal{E} \nabla(u^2) \cdot \nabla(\G_0 \Psi_{N,M}^2).
\end{multline}
Note that $u \in W^{1,2}_{\loc}(\Omega)\cap L^\infty(\Omega)$, $\G_0 \in W^{1,2}_{\loc}(\Omega \setminus \{X_0\})$, and that $\overline{\Omega_{\F_{N,M}, Q_0}^{\vartheta,**}}$ is a compact subset of $\Omega$ away from $X_0$ since  $X_0 \notin 4B_{Q_0}^*$ and \eqref{definicionkappa12}. Hence, $u \in W^{1,2}(\Omega_{\F_{N,M}, Q_0}^{\vartheta,**})$ and $u \G_0 \Psi_{N,M}^2 \in W_0^{1,2}(\Omega_{\F_{N,M}, Q_0}^{\vartheta,**})$. These together with the fact that $L_1u=0$ in the weak sense in $\Omega$ give
\begin{align}\label{eq:AA-1}
	\iint_{\Omega} A_1 \nabla u \cdot \nabla(u \G_0 \Psi_{N,M}^2) dX
	=\iint_{\Omega_{\F_{N,M}, Q_0}^{\vartheta,**}} A_1 \nabla u \cdot \nabla(u \G_0 \Psi_{N,M}^2) dX=0.
\end{align}
On the other hand, Lemma~\ref{lemma:Greensf} (see in particular \eqref{eq:G-delta})  implies that $\G_0 \in W^{1,2}(\Omega_{\F_{N,M}, Q_0}^{\vartheta,**})$ and $L_0^{\top}\G_0=0$ in the weak sense in $\Omega \setminus\{X_0\}$. Thanks to the fact that $u^2 \Psi_{N,M}^2 \in W^{1,2}_0(\Omega_{\F_{N,M}, Q_0}^{\vartheta,**})$, we then obtain
\begin{align}\label{eq:AA-2}
	\iint_{\Omega} A_0 \nabla (u^2 \Psi_{N,M}^2) \cdot \nabla \G_0 \, dX
	=\iint_{\Omega_{\F_{N,M}, Q_0}^{\vartheta,*}} A_0^{\top} \nabla \G_0 \cdot \nabla(u^2 \Psi_{N,M}^2)\, dX=0.
\end{align}
By Lemma~\ref{lemma:approx-sawtooth}, the ellipticity of $A_1$ and $A_0$, \eqref{eq:LeibAA}--\eqref{eq:AA-2}, the fact that $\|u\|_{L^{\infty}(\Omega)}=1$, and our assumption $\mathcal{E}=A_1-A_0=-(A+D)$  we then arrive at
\begin{align}\label{eq:KN}
	\widetilde{\mathcal{K}}_{N,M}
	&:= \iint_{\Omega} |\nabla u|^2 \G_0 \Psi_{N,M}^2 \, dX \lesssim \iint_{\Omega} A_1 \nabla u \cdot \nabla u \ \G_0 \Psi_{N,M}^2\, dX
	\\ \nonumber
	&\lesssim
	\iint_{\Omega} |\nabla\Psi_{N,M}|\,|\nabla \G_0|\, dX
	+
	\iint_{\Omega} |\nabla u|\, |\nabla \Psi_{N,M}|\,\G_0 \, dX
	\\  \nonumber
	&\qquad\qquad+ \bigg|\iint_{\Omega} A \nabla(u^2) \cdot \nabla(\G_0 \Psi_{N,M}^2)\,dX \bigg|
	+ \bigg|\iint_{\Omega} D \nabla(u^2) \cdot \nabla(\G_0 \Psi_{N,M}^2)\,dX \bigg|
	\\ \nonumber
	&=: \mathcal{I}_1+\mathcal{I}_2+\mathcal{I}_3+\mathcal{I}_4.
\end{align}

We estimate each term in turn. Regarding $\mathcal{I}_1$ we use Lemma~\ref{lemma:approx-sawtooth}, Caccioppoli's and Harnack's inequalities, and Lemma~\ref{lemma:proppde}:
\begin{multline}\label{eq:part-1}
\mathcal{I}_1
\lesssim \sum_{I \in \W_{N,M}^{\vartheta, \Sigma}} \iint_{I^*} |\nabla \Psi_{N,M}|\, |\nabla \G_0|\, dX
\lesssim \sum_{I \in \W_{N,M}^{\vartheta,\Sigma}}
	\ell(I)^{-1} |I|^{\frac12} \bigg(\iint_{I^*} |\nabla \G_0|^2 dX  \bigg)^{\frac12}
\\
		\lesssim \sum_{I \in \W_{N,M}^{\vartheta,\Sigma}} \ell(I)^{n-1} \G_0(X(I))
\lesssim \sum_{I \in \W_{N,M}^{\vartheta,\Sigma}} \w_0(\widehat{Q}_I)
	\lesssim \w_0\bigg(\bigcup_{I \in \W_{N,M}^{\vartheta,\Sigma}} \widehat{Q}_I \bigg)
	\leq \w_0(C \Delta_{Q_0}) \lesssim \w_0(Q_0),
\end{multline}
where the implicit constants do not depend on $N$ nor $M$. We  estimate $\mathcal{I}_2$ similarly:
\begin{multline}\label{eq:part-2}
	\mathcal{I}_2
	\lesssim \sum_{I \in \W_{N,M}^{\vartheta,\Sigma}} \iint_{I^*} |\nabla \Psi_{N,M}|\, |\nabla u|\, \G_0\, dX
	\lesssim \sum_{I \in \W_{N,M}^{\vartheta,\Sigma}}
	\ell(I)^{-1} |I|^{\frac12} \G_0(X(I)) \bigg(\iint_{I^*} |\nabla u|^2 dX  \bigg)^{\frac12}
	\\
	\lesssim \sum_{I \in \W_{N,M}^{\vartheta,\Sigma}} \ell(I)^{n-1} \G_0(X(I)) \lesssim \w_0(Q_0).
\end{multline}

Concerning  $\mathcal{I}_3$ we use that $A \in L^{\infty}(\Omega)$ and $\|u\|_{L^{\infty}(\Omega)}=1$:
\begin{align}\label{eq:KN-3}
\mathcal{I}_3 \lesssim
	\iint_{\Omega} |A|\, |\nabla u|\, |\nabla \G_0|\, \Psi_{N,M}^2 \, dX
	+ \iint_{\Omega} |\nabla u|\, |\nabla\Psi_{N,M}|\,\Psi_{N,M}\, \G_0\, \, dX
	=: \mathcal{I}_3'+\mathcal{I}_3''.
\end{align}
Observe that $I^{**} \subset \{Y\in \Omega: |Y-X|<\delta(X)/2\}$ for every $X \in I^*$, and hence $\sup_{I^{**}} |A| \le \inf_{I^*} a$. By Cauchy-Schwarz inequality,
Caccioppoli's and Harnack's inequalities, and Lemma~\ref{lemma:proppde} we have
\begin{align}\label{eq:KN-31}
	\mathcal{I}_3'
	&
	\lesssim
	\sum_{Q\in\dd_{\F_{N}, Q_0}}\sum_{I \in \W_{Q}^{\vartheta,*}} \sup_{I^{**}} |A|
	\bigg(\iint_{I^{**}} |\nabla u|^2 \Psi_{N,M}^2\, dX \bigg)^{\frac12}
	\bigg(\iint_{I^{**}} |\nabla \G_0|^2 dX\bigg)^{\frac12}
	\\ \nonumber
	&\lesssim
	\sum_{Q\in\dd_{\F_{N}, Q_0}}\sum_{I \in \W_{Q}^{\vartheta,*}}
	\bigg(\iint_{I^{**}} |\nabla u|^2 \Psi_{N,M}^2\, dX \bigg)^{\frac12}
	\Big(\sup_{I^{**}} |A|^2  \G_0(X(I))^2 \ell(I)^{n-1}\Big)^{\frac12}
	\\ \nonumber
	&\lesssim \sum_{Q\in\dd_{\F_{N}, Q_0}}\sum_{I \in \W_{Q}^{\vartheta,*}}
	\bigg(\iint_{I^{**}} |\nabla u|^2 \G_0 \Psi_{N,M}^2\, dX \bigg)^{\frac12}
	\bigg(\w_0(Q) \iint_{I^*} a(X)^2 \delta(X)^{-n-1} dX \bigg)^{\frac12}
	\\ \nonumber
	&\lesssim
	\bigg(\iint_{\Omega} |\nabla u|^2 \G_0 \Psi_{N,M}^2\, dX \bigg)^{\frac12}
	\bigg(\sum_{Q \in \D_{\F_N, Q_0}} \w_0(Q) \iint_{U_Q^{\vartheta,*}} a(X)^2 \delta(X)^{-n-1}\, dX \bigg)^{\frac12}
	\\ \nonumber
	&
	\le \widetilde{\mathcal{K}}_{N,M}^{\frac12} \, \Big(\sum_{Q \in \D_{\F_N, Q_0}} \gamma_Q^{\vartheta} \, \w_0(Q)\Big)^{\frac12},
\end{align}
where we used the fact that the family $\{I^{**}\}_{I \in \W}$ has bounded overlap. Additionally, as in \eqref{eq:part-1}
\begin{multline}\label{eq:KN-32}
	\mathcal{I}_3'' \lesssim \bigg(\iint_{\Omega} |\nabla u|^2 \G_0\, \Psi_{N,M}^2 \, dX\bigg)^{\frac12}
	\bigg(\iint_{\Omega} |\nabla \Psi_{N,M}|^2 \G_0 \, dX\bigg)^{\frac12}
	\\
	\lesssim \widetilde{\mathcal{K}}_{N,M}^{\frac12} \bigg(\sum_{I \in \W_{N,M}^{\vartheta,\Sigma}} \ell(I)^{n-1}\, \G_0(X(I)) \bigg)^{\frac12}
	\lesssim \widetilde{\mathcal{K}}_{N,M}^{\frac12} \, \w_0(Q_0)^{\frac12}.
\end{multline}

Finally, to bound $\mathcal{I}_4$, we note that $u^2 \in W_{\loc}^{1,2}(\Omega)$, $\G_0 \Psi_{N,M}^2 \in W^{1,2}(\Omega)$ and $\supp(\G_0 \Psi_{N,M}^2) \subset \overline{\Omega_{\F_{N,M}, Q_0}^{\vartheta,*}}$ is compactly contained in $\Omega$.  Then \cite[Lemma~4.1]{CHMT} and Lemma~\ref{lemma:proppde} imply that
\begin{align}\label{eq:KN-4}
	\mathcal{I}_4 & = \bigg|\iint_{\Omega} \div_C D \cdot \nabla(u^2)\, \G_0\, \Psi_{N,M}^2\, dX \bigg|
	\\ \nonumber
	&\lesssim \bigg(\iint_{\Omega} |\nabla u|^2\, \G_0\, \Psi_{N,M}^2 \, dX\bigg)^{\frac12}
	\bigg(\iint_{\Omega} |\div_C D|^2\, \G_0\, \Psi_{N,M}^2 \, dX\bigg)^{\frac12}
	\\ \nonumber
	&\lesssim \widetilde{\mathcal{K}}_{N,M}^{\frac12} \bigg( \sum_{Q\in\dd_{\F_{N}, Q_0}}\sum_{I \in \W_{Q}^{\vartheta,*}} \G_0(X(I)) \iint_{I^{**}} |\div_C D|^2\, dX \bigg)^{\frac12}
	\\ \nonumber
	&\lesssim \widetilde{\mathcal{K}}_{N,M}^{\frac12} \bigg(\sum_{Q\in\dd_{\F_{N}, Q_0}}\sum_{I \in \W_{Q}^{\vartheta,*}} \w_0(Q)
	\iint_{I^{**}} |\div_C D(X)|^2\, \delta(X)^{1-n}\, dX \bigg)^{\frac12}
	\\ \nonumber
	&\lesssim \widetilde{\mathcal{K}}_{N,M}^{\frac12} \bigg(\sum_{Q \in \D_{\F_N, Q_0}} \w_0(Q)
	\iint_{U_Q^{\vartheta,*}} |\div_C D(X)|^2\, \delta(X)^{1-n}\, dX \bigg)^{\frac12}
	\\ \nonumber
	&
		\le \widetilde{\mathcal{K}}_{N,M}^{\frac12} \, \Big(\sum_{Q \in \D_{\F_N, Q_0}} \gamma_Q^{\vartheta} \, \w_0(Q)\Big)^{\frac12}.
\end{align}

Gathering \eqref{eq:KN}--\eqref{eq:KN-4} and using Young's inequality we obtain
\begin{multline*}
	\widetilde{\mathcal{K}}_{N,M}
\lesssim
\w_0(Q_0)+\widetilde{\mathcal{K}}_{N,M}^{\frac12}\,\w_0(Q_0)^{\frac12}
+
\widetilde{\mathcal{K}}_{N,M}^{\frac12}\, \Big(\sum_{Q \in \D_{\F_N, Q_0}} \gamma_Q^{\vartheta} \, \w_0(Q)\Big)^{\frac12}
\\
\le
C\,\w_0(Q_0)+C\,\sum_{Q \in \D_{\F_N, Q_0}} \gamma_Q^{\vartheta} \, \w_0(Q)+\frac12\,\widetilde{\mathcal{K}}_{N,M},
\end{multline*}
where the implicit constants are independent of $N$ and $M$.
Note that $\widetilde{\mathcal{K}}_{N,M}<\infty$ because $\supp \Psi_{N,M} \subset \overline{\Omega_{\F_{N,M}, Q_0}^{\vartheta,*}}$, which is a compact subset of $\Omega$ and $u \in W^{1,2}_{\loc}(\Omega)$. Thus, the last term can be hidden and  we eventually obtain
\begin{align*}
	\mathcal{K}_{N, M} \le\widetilde{\mathcal{K}}_{N,M} \lesssim \w_0(Q_0) +\sum_{Q \in \D_{\F_N, Q_0}} \gamma_Q^{\vartheta} \, \w_0(Q).
\end{align*}
This estimate (whose implicit constant is independent of $N$ and $M$) and \eqref{eq:JM-lim} readily yield \eqref{eq:KAB*} and the proof is complete.
\end{proof}
%%%%%%%%%%%%%%%%%%%%%%%%% END END END PROOF %%%%%%%%%%%%%%%%%%%%%%%%

\end{document}